\documentclass{lsthesis}
\usepackage{amsmath,amssymb,amsthm,multicol}
\title{Solving the Odd Perfect Number Problem:\\ Some Old and New Approaches}
\author{Jose Arnaldo B. Dris}
\papertype{Thesis}
\department{Mathematics Department\\College of Science} \degree{Master of Science in Mathematics}
\ddate{August 2008}
\begin{document}
\newtheorem{thm}{Theorem}[section]
\newtheorem{cor}{Corollary}[section]
\newtheorem{lemm}{Lemma}[section]
\newtheorem{conj}{Conjecture}[section]
\theoremstyle{definition}
\newtheorem{defn}{Definition}[section]
\newtheorem{exmpl}{Example}[section]
\newtheorem{remrk}{Remark}[section]
\doublespacing
\begin{preliminary}
\maketitle
\begin{acknowledgments}
The author of this thesis wishes to express his heartfelt gratitude to the following:
\begin{multicols}{2}
\begin{itemize}
\item The Mathematics Department, DLSU-Manila
\item The Commission on Higher Education - Center of Excellence
\item Dr.~Severino V.~Gervacio
\item Dr.~Leonor Aquino-Ruivivar
\item Dr.~Fidel Nemenzo
\item Ms.~Sonia Y.~Tan
\item Dr.~Blessilda Raposa
\item Dr.~Ederlina Nocon
\item Ms.~Gladis Habijan
\item Dr.~Isagani B.~Jos
\item Dr.~Arlene Pascasio
\item Dr.~Jose Tristan Reyes
\item Dr.~Yvette F.~Lim
\item Mr.~Frumencio Co
\item Dr.~Julius Basilla
\item Dr.~Rizaldi Nocon
\item Dr.~John McCleary
\item Dr.~Carl Pomerance
\item Dr.~Douglas Iannucci
\item Dr.~Judy Holdener
\item Prof.~Richard P.~Brent
\item Prof.~Richard F.~Ryan
\item Ms.~Laura Czarnecki
\item Mr.~William Stanton
\item Mr.~James Riggs
\item Mr.~Tim Anderton
\item Mr.~Dan Staley
\item Mr.~William Lipp
\item Mr.~Tim Roberts
\item Mr.~Rigor Ponsones
\item Mr.~Gareth Paglinawan
\item Mr.~Christopher Thomas Cruz
\item Ms.~Michele Tan
\item Mr.~Mark Anthony Garcia
\item Mr.~John Ruero
\item Mr.~Vincent Chuaseco
\item Mrs.~Abigail Arcilla
\item Mr.~Mark John Hermano
\end{itemize}
\end{multicols}
\end{acknowledgments}
\addcontentsline{toc}{chapter}{Table of Contents}
\tableofcontents\newpage
\begin{listofnotations}
\begin{tabular}{ll}
$\mathbb N, {\mathbb Z}^+$ & the set of all natural numbers/positive integers \\
$a\mid b$ & $a$ divides $b$, $a$ is a divisor/factor of $b$, $b$ is a multiple of $a$ \\
$P^{\alpha} \| N$ & $P^{\alpha}$ is the largest power of $P$ that divides $N$, i.e. $P^{\alpha}\mid N$ but $P^{\alpha + 1}\nmid N$ \\
$a \equiv b \pmod{n}$ & $a$ is congruent to $b$ modulo $n$ \\
$\gcd(a, b)$ & the greatest common divisor of $a$ and $b$ \\
$d(n)$ & the number of positive divisors of $n$ \\
$\sigma(n)$ & the sum of the positive divisors of $n$ \\
$\phi(n)$ & the number of positive integers less than or equal to $n$ which are also \\ & relatively prime to $n$ \\
$\omega(n)$ & the number of distinct primes that divide $n$ \\
$\Omega(n)$ & the number of primes that divide $n$, counting multiplicities \\
${{\sigma}_{-1}}(n), I(n)$ & the abundancy index of $n$, i.e. the sum of the reciprocals of all the \\ & positive divisors of $n$ \\
$\displaystyle\prod_{i=1}^r {{p_i}^{{\alpha}_i}}$ & the product ${p_1}^{\alpha_1}{p_2}^{\alpha_2}{p_3}^{\alpha_3}\cdots{p_{r-1}}^{\alpha_{r-1}}{p_r}^{\alpha_r}$ \\
$\displaystyle\sum_{j=0}^s {{q_i}^j}$ & the sum $1 + q_i + {q_i}^2 + {q_i}^3 + \ldots + {q_i}^{s - 1} + {q_i}^s$ \\
\end{tabular}
\end{listofnotations}
\begin{abstract}
\begin{paragraph}\indent A perfect number is a positive integer $N$ such that the sum of all the positive divisors of $N$ equals $2N$, denoted by $\sigma(N) = 2N$.  The question of the existence of odd perfect numbers (OPNs) is one of the longest unsolved problems of number theory.  This thesis presents some of the old as well as new approaches to solving the OPN Problem.  In particular, a conjecture predicting an injective and surjective mapping $X = \frac{\sigma(p^k)}{p^k}, \hspace{0.1in} Y = \frac{\sigma(m^2)}{m^2}$ between OPNs $N = {p^k}{m^2}$ (with Euler factor $p^k$) and rational points on the hyperbolic arc $XY = 2$ with $1 < X < 1.25 < 1.6 < Y < 2$ and $2.85 < X + Y < 3$, is disproved. Various results on the abundancy index and solitary numbers are used in the disproof.  Numerical evidence against the said conjecture will likewise be discussed.  We will show that if an OPN $N$ has the form above, then $p^k < {\frac{2}{3}}{m^2}$ follows from \cite{D10}.  We will also attempt to prove a conjectured improvement of this last result to $p^k < m$ by observing that ${\frac{\sigma(p^k)}{m}} \neq 1$ and ${\frac{\sigma(p^k)}{m}} \neq {\frac{\sigma(m)}{p^k}}$ in all cases.  Lastly, we also prove the following generalization: If $N = \displaystyle\prod_{i=1}^r {{p_i}^{{\alpha}_i}}$ is the canonical factorization of an OPN $N$, then $\sigma({p_i}^{{\alpha}_i}) \leq {\frac{2}{3}}{\frac{N}{{p_i}^{{\alpha}_i}}}$ for all $i$.  This gives rise to the inequality $N^{2 - r} \leq (\frac{1}{3})(\frac{2}{3})^{r - 1}$, which is true for all $r$, where $r = \omega(N)$ is the number of distinct prime factors of $N$. \\
\end{paragraph}
\end{abstract}
\end{preliminary}
\chapter{The Problem and Its Background}
\section{Introduction}
Number theory is that branch of pure mathematics concerned with the properties of integers.  It consists of various results and open problems that are easily understood, even by non-mathematicians.  More generally, the field has developed to a point where it could now tackle wider classes of problems that arise naturally from the study of integers.
\begin{paragraph}\indent A particular example of an unsolved problem in number theory, one which has long captured the interest of both amateur and professional mathematicians, is to determine whether an odd perfect number exists.  A positive integer $n$ which is equal to the sum of its positive proper divisors (that is, excluding $n$ itself) is called a \emph{perfect number}.  The smallest known example of a perfect number is $6$ since $1 + 2 + 3 = 6$, where $1$, $2$ and $3$ are the positive proper divisors of $6$.  The Pythagoreans considered such numbers to possess mystical properties, thus calling them perfect numbers. \\
\end{paragraph}
\begin{paragraph}\indent The sum of all positive divisors of a positive integer $n$ is called the sigma function of $n$, denoted by $\sigma(n)$.  The definition of a perfect number is thus equivalent to determining those $n$ for which $\sigma(n) - n = n$. We formalize this definition as follows:
\begin{defn}\label{perfect} A positive integer $N$ is perfect if $\sigma(N) = 2N$.
\end{defn}
\end{paragraph}
\begin{paragraph}\indent For example, $\sigma(6) = 1 + 2 + 3 + 6 = 12 = 2(6)$.
\end{paragraph}
\begin{paragraph}\indent Even though the Greeks knew of only four perfect numbers, it was Euclid who proved that if the sum $1 + 2 + 2^2 + \ldots + 2^{p - 2} + 2^{p - 1} = 2^p - 1$ is prime, then $2^{p - 1}(2^p - 1)$ is perfect.  Consider the sum $1 + 2 = 3$ as an example.  Since $3$ is prime, then $2^{2 - 1}(2^2 - 1) = 2(3) = 6$ is perfect.
\end{paragraph}
\begin{paragraph}\indent Euler subsequently proved the following theorem about (even) perfect numbers:
\begin{thm}\label{theorem1} Every even perfect number is of the form $N = 2^{P - 1}(2^P - 1)$, where $P$ and $2^P - 1$ are primes.
\end{thm}
\end{paragraph}
\begin{paragraph}\indent It is easy to prove that if $C$ is composite, then $2^C - 1$ is also composite.  However, if $P$ is prime, it does not necessarily follow that $2^P - 1$ is also prime.  (Consider the case $P = 11: 2^{11} - 1 = 2047 = 23\cdot89$, which is composite.)  Primes of the form $2^P - 1$ are called \emph{Mersenne primes}, after the French monk and mathematician Marin Mersenne.  In view of Theorem \ref{theorem1}, the problem of searching for even perfect numbers is thus reduced to looking for Mersenne primes, since the theorem essentially says that the even perfect numbers are in one-to-one correspondence with the Mersenne primes.
\end{paragraph}
\begin{paragraph}\indent As of this writing, 44 perfect numbers are known \cite{V55}, the last few of which have been found with the aid of high-speed computers.  Following is a list of all the known exponents $p$ for which $M_p = 2^p - 1$ is prime, along with other pertinent information (note that $P_p$ refers to the perfect number $N = 2^{p - 1}(2^p - 1)$):
\end{paragraph}
\vspace{0.03truein}
{\small
\begin{center}
\begin{tabular}{|l|l|l|l|l|l|}
\hline
number & $p$ (exponent) & digits in $M_p$ & digits in $P_p$ & year & discoverer \\
\hline
1   & 2 	  		&	1 	  		&	1 	   		 &	---- &	----              					\\
\hline
2   &	3 	  		&	1 	  		&	2 	   		 &	---- &	----   											\\
\hline
3   &	5 	  		&	2 	  		&	3 	   		 &	---- &	----   											\\
\hline
4   &	7 	  		&	3 	  		&	4 	   		 &	---- &	----   											\\
\hline
5   & 13 	  		&	4 	  		&	8 	   		 &	1456 &	anonymous										\\
\hline
6   & 17 	  		&	6 	  		&	10 	   		 &	1588 &	Cataldi   									\\
\hline
7   &	19 	  		&	6 	  		& 12 	   		 & 	1588 &	Cataldi   									\\
\hline
8   & 31 	  		&	10 	  		&	19 	   		 &	1772 &	Euler   										\\
\hline
9   & 61 	  		&	19 	  		&	37 	   		 & 	1883 &	Pervushin   								\\
\hline
10  &	89 	  		&	27 	  		&	54 	   		 & 	1911 &	Powers   										\\
\hline
11  &	107 	  	&	33 	  		&	65 	   		 & 	1914 &	Powers  										\\
\hline
12  &	127 	  	&	39 	  		&	77 	   		 &  1876 &	Lucas   										\\
\hline
13  &	521 	  	&	157 	  	&	314 	   	 &	1952 &	Robinson   									\\
\hline
14  &	607 	  	&	183 	  	&	366 	  	 &	1952 &	Robinson   									\\
\hline
15  &	1279 	  	&	386 	  	&	770 		   &	1952 &	Robinson   									\\
\hline
\end{tabular}
\end{center}
}
\newpage
{\small
\begin{center}
\begin{tabular}{|l|l|l|l|l|l|}
\hline
number & $p$ (exponent) & digits in $M_p$ & digits in $P_p$ & year & discoverer \\
\hline
16  &	2203 	  	&	664 	  	&	1327	 	   &	1952 &	Robinson   									\\
\hline
17  &	2281 	  	&	687 	  	&	1373       &	1952 &	Robinson   									\\
\hline
18  &	3217 	  	&	969 	  	&	1937       &	1957 &	Riesel   										\\
\hline
19  &	4253 	  	&	1281 	  	&	2561       &	1961 &	Hurwitz   									\\
\hline
20  &	4423 	  	&	1332 	  	&	2663       &	1961 &	Hurwitz   									\\
\hline
21  &	9689 	  	&	2917  		&	5834       &	1963 &	Gillies   									\\
\hline
22  &	9941 	  	&	2993  		&	5985 	     &	1963 &	Gillies   									\\
\hline
23  &	11213 	  &	3376 	  	&	6751 	   	 &	1963 &	Gillies   									\\
\hline
24  &	19937 	  &	6002 	  	&	12003 	   &	1971 &	Tuckerman										\\
\hline
25  &	21701 	  &	6533 	  	&	13066 	   &	1978 &	Noll, Nickel								\\
\hline
26  &	23209 	  &	6987 	  	&	13973 	   &	1979 &	Noll 												\\
\hline
27  &	44497 	  &	13395 	  &	26790 	   &	1979 &	Nelson, Slowinski						\\
\hline
28  &	86243 	  &	25962 	  &	51924 	   &	1982 &	Slowinski										\\
\hline
29  &	110503 	  &	33265 	  &	66530 	   &	1988 &	Colquitt, Welsh							\\
\hline
30  &	132049 	  &	39751 	  &	79502 	   &	1983 &	Slowinski   								\\
\hline
31  &	216091 	  &	65050 	  &	130100 	   &	1985 &	Slowinski   								\\
\hline
32  &	756839 	  &	227832 	  &	455663 	   &	1992 &	Slowinski, Gage \emph{et al.} 			\\
\hline
33  &	859433 	  &	258716 	  &	517430 	   &	1994 &	Slowinski, Gage   					\\
\hline
34  &	1257787   &	378632 	  &	757263 	   &	1996 &	Slowinski, Gage							\\
\hline
35  &	1398269   &	420921 	  &	841842     &	1996 &	Armengaud, Woltman, \emph{et al.}	\\
\hline
\end{tabular}
\end{center}
}
\newpage
{\small
\begin{center}
\begin{tabular}{|l|l|l|l|l|l|}
\hline
number & $p$ (exponent) & digits in $M_p$ & digits in $P_p$ & year & discoverer \\
\hline
36  &	2976221   &	895932 	  &	1791864    &	1997 &	Spence, Woltman, \emph{et al.} 		\\
\hline
37  &	3021377   &	909526 	  &	1819050    &	1998 &	Clarkson,Woltman, \emph{et al.}		\\
\hline
38  &	6972593   &	2098960   &	4197919    &	1999 &  Hajratwala, Woltman, \emph{et al.}	\\
\hline
39  &	13466917  & 4053946   &	8107892    &	2001 &	Cameron, Woltman, \emph{et al.} 		\\
\hline
??  &	20996011  &	6320430   &	12640858   &	2003 &	Shafer, Woltman, \emph{et al.} 		\\
\hline
??  &	24036583  &	7235733   &	14471465   &	2004 &	Findley, Woltman, \emph{et al.}		\\
\hline
??  &	25964951  &	7816230	  &	15632458   &	2005 &	Nowak, Woltman, \emph{et al.}			\\
\hline
??  &	30402457  &	9152052	  &	18304103   &	2005 &	Cooper, Woltman, \emph{et al.} 		\\
\hline
??  &	32582657  &	9808358   &	19616714   & 	2006 &	Cooper, Woltman, \emph{et al.} 		\\
\hline
\end{tabular}
\end{center}
}
\vspace{0.1truein} ~																															 \\
Question marks (??) were used instead of a number for the the last of the Mersenne primes because it will not be known if there are other Mersenne primes in between these until a check and double check is done by the Great Internet Mersenne Prime Search (GIMPS) \cite{V55} or other similar coordinated computing projects.
\begin{paragraph}\indent Note that all of the known perfect numbers are even.
\end{paragraph}
\subsection{Statement of the Problem}
Our primary object of interest would be odd perfect numbers, though some of the results that would be discussed apply to even perfect numbers as well. It is unknown whether there are any odd perfect numbers. Various results have been obtained, but none has helped to locate one or otherwise resolve the question of their existence.
\begin{paragraph}\indent This thesis explores some of the old as well as new approaches used in trying to solve the Odd Perfect Number (OPN) Problem, namely:
\end{paragraph}
\begin{itemize}
\item the use of the abundancy index to derive conditions for the existence of odd perfect numbers, in such instances as:
\begin{itemize}
\item bounding the prime factors of an OPN;
\item determining whether a particular rational number may be an abundancy index of a positive integer;
\item increasing the lower bound for the number of distinct prime factors, $\omega(N)$, that an OPN $N$ must have;
\end{itemize}
\item proving the inequality $p^k < {\frac{2}{3}}{m^2}$ where $N = {p^k}{m^2}$ is an OPN with $p^k$ the Euler factor of $N$, $p \equiv k \equiv 1 \pmod{4}$, and $\gcd(p, m) = 1$;
\item attempting to prove the conjectured improvement of the result $p^k < {\frac{2}{3}}{m^2}$ to $p^k < m$ by observing that ${\frac{\sigma(p^k)}{m}} \neq 1$ and ${\frac{\sigma(p^k)}{m}} \neq {\frac{\sigma(m)}{p^k}}$ apply in all cases;
\item generalizing the result $p^k < {\frac{2}{3}}{m^2}$ to: if $N = \displaystyle\prod_{i=1}^r {{p_i}^{{\alpha}_i}}$ is the canonical factorization of an OPN $N$, then $\sigma({p_i}^{{\alpha}_i}) \leq {\frac{2}{3}}{\frac{N}{{p_i}^{{\alpha}_i}}}$ for all $i$;
\item showing that $N^{2 - r} \leq (\frac{1}{3})(\frac{2}{3})^{r - 1}$ follows from $\sigma({p_i}^{{\alpha}_i}) \leq {\frac{2}{3}}{\frac{N}{{p_i}^{{\alpha}_i}}}$ for all $i$, where $r = \omega(N)$ is the number of distinct prime factors of an OPN $N$;
\item disproving the conjectured injectivity and surjectivity of the mapping \\
${X = \frac{\sigma(p^k)}{p^k}}, \hspace{0.1in} {Y = \frac{\sigma(m^2)}{m^2}}$ between OPNs $N = {p^k}{m^2}$ (with Euler factor $p^k$) and rational points on the hyperbolic arc $XY = 2$ with\\
$1 < X < 1.25 < 1.6 < Y < 2 < 2.85 < X + Y < 3$;
\item a host of other interesting results on perfect numbers, including:
\begin{itemize}
\item establishing that two consecutive positive integers cannot be both perfect;
\item the proof that an OPN is a sum of two squares.
\end{itemize}
\end{itemize}
Specifically, this thesis, among other things, aims to present (partial) expositions of the following articles/notes that deal with some of the above concerns:
\begin{itemize}
\item ``\emph{Abundancy 'Outlaws' of the Form} $\frac{\sigma(N) + t}{N}$" by W. G. Stanton, which was a joint research project with J. A. Holdener at Kenyon College, Gambier, OH ($2007$);
\item ``\emph{Conditions Equivalent to the Existence of Odd Perfect Numbers}" by J. A. Holdener, which appeared in Mathematics Magazine $\emph{79(5)}$ ($2006$);
\item ``\emph{Bounding the Prime Factors of Odd Perfect Numbers}", an undergraduate paper of C. Greathouse at Miami University ($2005$);
\item ``\emph{Hunting Odd Perfect Numbers: Quarks or Snarks?}" by J. McCleary of Vassar College, Poughkeepsie, NY, which consists of lecture notes first presented as a seminar to students of Union College, Schenectady, NY ($2001$);
\item ``\emph{Consecutive Perfect Numbers (actually, the Lack Thereof!)}" by J. Riggs, which was an undergraduate research project with J. A. Holdener at Kenyon College ($1998$).
\end{itemize}
\subsection{Review of Related Literature}
Benjamin Peirce was the first to prove (in $1832$) that an OPN $N$ must have at least four distinct prime factors (denoted $\omega(N) \geq 4$) \cite{P42}. Seemingly unaware of this result, James Joseph Sylvester published a paper on the same result in $1888$, hoping that research along these lines would pave the way for a general proof of the nonexistence of an OPN \cite{S55}. Later that same year, Sylvester established $\omega(N) \geq 5$ \cite{S56}. This marked the beginning of the modern era of research on OPNs. It was only in $1925$ that Gradstein was able to improve Sylvester's result to $\omega(N) \geq 6$ \cite{G13}. In the early $1970$s, Robbins and Pomerance independently established that $\omega (N) \geq 7$ \cite{P44}. Then, Chein demonstrated that $\omega(N) \geq 8$ in his $1979$ doctoral thesis \cite{C4}, which was verified independently by Hagis in a $1980$ $200$-page manuscript \cite{H16}. Recently, in a $2006$ preprint titled ``\emph{Odd perfect numbers have at least nine distinct prime factors}", Pace Nielsen was able to prove that $\omega(N) \geq 9$, and if $N_2$ is an OPN which is not divisible by $3$, that $\omega(N_2) \geq 12$. ``The proof ultimately avoids previous computational results for odd perfect numbers." \cite{N40}
\begin{paragraph}\indent Sylvester also showed (later in the year $1888$) that an OPN cannot be divisible by $105$ \cite{S56}, and Servais (in the same year) proved that the least prime divisor of a perfect number with $r$ distinct prime factors is bounded above by $r + 1$ \cite{S50}.  Gr$\ddot{u}$n \cite{G16}, Cohen and Hendy \cite{C6}, and McDaniel \cite{M38} established improvements and extensions to this last result later on.
\end{paragraph}
\begin{paragraph}\indent Dickson showed in $1913$ that there can only be finitely many OPNs with $r$ distinct prime factors, for a given positive integer $r$ \cite{D12}. Kanold then generalized Dickson's theorem in $1956$ to include any positive integer $n$ satisfying ${\frac{\sigma(n)}{n}} = {\frac{a}{b}}$, where $a$ and $b$ are positive integers and $b \neq 0$. \cite{K34} 
\end{paragraph}
\begin{paragraph}\indent (Call a number $n$ non-deficient if ${\frac{\sigma(n)}{n}} \geq 2$ .  Dickson called a number \emph{primitive non-deficient} provided that it is not a multiple of a smaller non-deficient number.) Mathematicians considered Dickson's approach (i.e. first delineating all of the finitely many primitive odd non-deficient numbers associated with a particular $r$-value and then determining which among them are equal to the sum of their positive proper divisors) to the OPN question to be impractical for most values of $r$, making it necessary to explore alternative approaches to examining the possible structure of an OPN.  Pomerance suggested the following class of theorems in $1974$: \emph{An OPN is divisible by $j$ distinct primes $> N$} \cite{P44}. Kanold was successful with $j = 1$, $N = 60$ in $1949$ and used only elementary techniques \cite{K33}.  In $1973$, with the aid of computation, Hagis and McDaniel improved Kanold's finding to $j = 1$, $N = 11200$ \cite{H18}. This was pushed to $j = 1$, $N = 100110$ by the same authors in $1975$ \cite{H19}. Pomerance showed that $j = 2$, $N = 138$ in the same year. \cite{P45}
\end{paragraph}
\begin{paragraph}\indent In a recent preprint (titled ``\emph{Odd perfect numbers have a prime factor exceeding ${10}^8$}") that appeared $2006$, authors Takeshi Goto and Yasuo Ohno report that the largest prime factor of an OPN exceeds ${10}^8$ \cite{G12}.  It betters the previous bound of ${10}^7$ established by Jenkins in $2003$ \cite{J31}.  New bounds for the second and third largest prime factors of an OPN were given by Iannucci in $1999$ \cite{I28} and $2000$ \cite{I29}, when he announced that they are larger than ${10}^4$ and ${10}^2$, respectively.
\end{paragraph}
\begin{paragraph}\indent Mathematicians also began considering estimates on the overall magnitude of an OPN by imposing lower bounds.  Turcaninov obtained the classical lower bound of $2 \cdot {10}^6$ in $1908$ \cite{B4}. The following table summarizes the development of ever-higher bounds for the smallest possible odd perfect number:
\begin{center}
\begin{tabular}{|c|c|}
\hline
Author & Bound \\
\hline
Kanold ($1957$) &	${10}^{20}$ \\
\hline
Tuckerman ($1973$) & ${10}^{36}$ \\
\hline
Hagis ($1973$) & ${10}^{50}$ \\
\hline
Brent and Cohen ($1989$) & ${10}^{160}$ \\
\hline
Brent \emph{et al.} ($1991$)	& ${10}^{300}$ \\
\hline
\end{tabular} \\
\end{center}
\vspace{0.08truein}
There is a project underway at http://www.oddperfect.org (organized by William Lipp) seeking to extend the bound beyond ${10}^{300}$.  A proof for ${10}^{500}$ is expected very soon, as all the remaining factorizations required to show this are considered ``easy", by Lipp's standards. \cite{G15}
\end{paragraph}
\begin{paragraph}\indent It would also be possible to derive upper bounds on the overall size of an OPN in terms of the number of its distinct prime factors.  Heath-Brown was able to show, in $1994$, that if $n$ is an odd number with $\sigma(n) = an$, then $n < {{(4d)}^4}^r$, where $d$ is the denominator of $a$ and $r$ is the number of distinct prime factors of $n$ \cite{H23}.  Specifically, this means that for an OPN $n$, $n < {{4}^4}^r$ which sharpens Pomerance's previous estimate of $n < {{{(4r)}^{(4r)}}^2}^{r^2}$ in $1977$ \cite{P45}. Referring to his own finding, Heath-Brown remarked that it still is too big to be of practical value. Nonetheless, it is to be noted that if it is viewed alongside the lower bound of ${10}^{300}$ given by Brent et. al. \cite{B3}, then Sylvester's $1888$ result that $\omega(n) \geq 5$ could then be demonstrated by no longer than a footnote.~\footnote{${10}^{300} < n < {{4}^4}^r$ implies that $r > 4.48$.} In $1999$, Cook enhanced Heath-Brown's result for an OPN with $r$ distinct prime factors to $n < {{D}^4}^r$ where $D = {(195)}^{1/7} \approx 2.124$ \cite{C8}. In $2003$, Pace Nielsen refined Cook's bound to $n < {{2}^4}^r$ \cite{N39}.
\end{paragraph}
\begin{paragraph}\indent Addressing the OPN question from a congruence perspective on the allowable exponents for the non-Euler prime factors, Steuerwald showed in $1937$ that if
\begin{center}
$n = {p^{\alpha}}{{q_1}^{2{\beta}_1}}{{q_2}^{2{\beta}_2}}\cdots{{q_s}^{2{\beta}_s}}$
\end{center}
was an OPN where $p, q_1, q_2, \ldots, q_s$ are distinct odd primes and $p \equiv \alpha \equiv 1 \pmod{4}$, then not all of the $\beta_i$'s can equal $1$ \cite{S52}. Further, Kanold discovered in $1941$ that it is neither possible for all $\beta_i$'s to equal $2$ nor for one of the $\beta_i$'s to be equal to $2$ while all the rest are equal to $1$ \cite{K32}. Hagis and McDaniel proved in $1972$ that not all the $\beta_i$'s can be equal to $3$ \cite{H17}. Then in $1985$, Cohen and Williams summarized all previous work done on this area by eliminating various  possibilities for the $\beta_i$'s, on the assumption that either some or all of the $\beta_i$'s are the same \cite{C7}.
\end{paragraph}
\begin{paragraph}\indent In $2003$, Iannucci and Sorli placed restrictions on the $\beta_i$'s  in order to show that $3$ cannot divide an OPN if, for all $i$, $\beta_i \equiv 1 \pmod{3}$ or $\beta_i \equiv 2 \pmod{5}$.  They also provided a slightly different analysis by giving a lower bound of 37 on the total number of prime divisors (counting multiplicities) that an OPN must have (i.e. they proved that if $n = {p^{\alpha}}{\displaystyle\prod_{i=1}^s {{q_i}^{2{\beta}_i}}}$ is an OPN, then $\Omega(n) = \alpha + 2{\displaystyle\sum_{i=1}^s {{\beta}}_i} \geq 37$) \cite{I30}. This was extended by Hare later in the year $2003$ to $\Omega(n) \geq 47$ \cite{H21}. In $2005$, Hare submitted the preprint titled ``\emph{New techniques for bounds on the total number of prime factors of an odd perfect number}" to the journal Mathematics of Computation for publication, where he announced a proof for $\Omega(n) \geq 75$ \cite{H22}.
\end{paragraph}
\begin{paragraph}\indent In order to successfully search for perfect numbers, it was found necessary to consider a rather interesting quantity called the \emph{abundancy index} or  \emph{abundancy ratio} of $n$, defined to be the quotient $\displaystyle\frac{\sigma(n)}{n}$.  Obviously, a number $n$ is perfect if and only if its abundancy index is $2$.  Numbers for which this ratio is greater than (less than) $2$ are called \emph{abundant} (\emph{deficient}) numbers.
\end{paragraph}
\begin{paragraph}\indent It can be shown that the abundancy index takes on arbitrarily large values. Also, we can make the abundancy index to be as close to 1 as we please because ${\displaystyle\frac{\sigma(p)}{p}} = {\displaystyle\frac{p + 1}{p}}$ for all primes $p$. In fact, Laatsch showed in $1986$ that the set of abundancy indices $\displaystyle\frac{\sigma(n)}{n}$ for $n > 1$ is dense in the interval $(1, \infty)$ \cite{L35}. (Let ${I(n) = \displaystyle\frac{\sigma(n)}{n}}$, and call a rational number greater than 1 an \emph{abundancy outlaw} if it fails to be in the image of the function $I$. \cite{H26}) Interestingly, Weiner proved that the set of abundancy outlaws is also dense in $(1, \infty)$! \cite{W58} It appears then that the implicit scenarios for abundancy indices and outlaws are both complex and interesting.
\end{paragraph}
\begin{paragraph}\indent In $2006$, Cruz \cite{C9} completed his M.~S.~ thesis titled ``Searching for Odd Perfect Numbers" which contained an exposition of the results of Heath-Brown \cite{H23} and Iannucci/Sorli \cite{I30}.  Cruz also proposed a hypothesis that may lead to a disproof of the existence of OPNs.
\end{paragraph}
\chapter{Preliminary Concepts}
The concept of divisibility plays a central role in that branch of pure mathematics called the theory of numbers.  Indeed, mathematicians have used divisibility and the concept of unique factorization to establish deep algebraic results in number theory and related fields where it is applied.  In this chapter, we survey some basic concepts from elementary number theory, and use these ideas to derive the possible forms for even and odd perfect numbers.
\section{Concepts from Elementary Number Theory}
For a better understanding of the topics presented in this thesis, we recall the following concepts.
\begin{defn}\label{definition1} An integer $n$ is said to be \emph{divisible} by a nonzero integer $m$, denoted by $m \mid n$, if there exists some integer $k$ such that $n = km$. The notation $m \nmid n$ is used to indicate that $n$ is \emph{not divisible} by $m$.
\end{defn}
\begin{paragraph}\indent For example, $143$ is divisible by $11$ since $143 = 11\cdot13$. In this case, we also say that $11$ and $13$ are \emph{divisors/factors} of $143$, and that $143$ is a \emph{multiple} of $11$ (and of $13$). On the other hand, $143$ is not divisible by $3$ since we will not be able to find an integer $k$ that will make the equation $3k = 143$ true.
\end{paragraph}
\begin{paragraph}\indent If $n$ is divisible by $m$, then we also say that $m$ \emph{divides} $n$.
\end{paragraph}
\begin{paragraph}\indent We list down several properties of divisibility in Theorem 2.1.1.
\end{paragraph}
\begin{thm}\label{theorem2} For integers $k$, $l$, $m$, and $n$, the following are true:
\begin{itemize}
\item{$n \mid 0$, $1 \mid n$, and $n \mid n$. (Any integer is a divisor of $0$, $1$ is a divisor of any integer, and any integer has itself as a divisor.)}
\item{$m \mid 1$ if and only if $m = \pm 1$. (The only divisors of $1$ are itself and $-1$.)}
\item{If $k \mid m$ and $l \mid n$, then $kl \mid mn$. (Note that this statement is one-sided.) }
\item{If $k \mid l$ and $l \mid m$, then $k \mid m$. (This means that divisibility is transitive.)}
\item{$m \mid n$ and $n \mid m$ if and only if $m = \pm n$. (Two integers which divide each other can only differ by a factor of $\pm 1$.)}
\item{If $m \mid n$ and $n \neq 0$, then $|m| \leq |n|$. (If the multiple of an integer is nonzero, then the multiple has bigger absolute value than the integer.)}
\item{If $k \mid m$ and $k \mid n$, then $k \mid (am + bn)$ for any integers $a$ and $b$. (If an integer divides two other integers, then the first integer divides any linear combination of the second and the third.)}
\end{itemize}
\end{thm}
\begin{paragraph}\indent A very useful concept in the theory of numbers is that of the $GCF$ or $GCD$ of two integers.
\end{paragraph}
\begin{defn}\label{definition2} Let $m$ and $n$ be any given integers such that at least one of them is not zero.  The \emph{greatest common divisor} of $m$ and $n$, denoted by $\gcd(m, n)$, is the positive integer $k$ which satisfies the following properties:
\begin{itemize}
\item{$k \mid m$ and $k \mid n$; and}
\item{If $j \mid m$ and $j \mid n$, then $j \mid k$.}
\end{itemize}
\end{defn}
\begin{exmpl}\label{example1} The positive divisors of $36$ are $1$, $2$, $3$, $4$, $6$, $9$, $12$, $18$ and $36$.  For $81$ they are $1$, $3$, $9$, $27$ and $81$. Thus, the positive divisors common to $36$ and $81$ are $1$, $3$ and $9$. Since $9$ is the largest among the common divisors of $36$ and $81$, then $\gcd(36, 81) = 9$.
\end{exmpl}
\begin{paragraph}\indent Another concept of great utility is that of two integers being relatively prime.
\end{paragraph}
\begin{defn}\label{definition3} Let $m$ and $n$ be any integers. If $\gcd(m, n) = 1$, then $m$ and $n$ are said to be \emph{relatively prime}, or \emph{coprime}.
\end{defn}
\begin{exmpl}\label{example2} Any two consecutive integers (like $17$ and $16$, or $8$ and $9$) are relatively prime.  Note that the two consecutive integers are of opposite parity (i.e. one is odd, the other is even).  If two integers are of opposite parity, but not consecutive, it does not necessarily follow that they are relatively prime. (See Example \ref{example1}.)
\end{exmpl}
\begin{paragraph}\indent It turns out that the concept of divisibility can be used to partition the set of positive integers into three classes: the unit $1$, primes and composites.
\end{paragraph}
\begin{defn}\label{definition4} An integer $P > 1$ is called a \emph{prime number}, simply a \emph{prime}, if it has no more positive divisors other than $1$ and $P$.  An integer greater than 1 is called a \emph{composite number}, simply a \emph{composite}, if it is not a prime.
\end{defn}
\begin{paragraph}\indent There are only $23$ primes in the range from $1$ to $100$, as compared to $76$ composites in the same range.  Some examples of primes in this range include $2$, $7$, $23$, $31$, $41$, and $47$.  The composites in the same range include all the larger multiples of the aforementioned primes, as well as product combinations of two or more primes from the range $1$ to $10$. (Note that we get $100$ by multiplying the two composites $10$ and $10$.) We casually remark that the integer $2$ is the only even prime.  The integer $1$, by definition, is neither prime nor composite.  We shall casually call $1$ the \emph{unit}.
\end{paragraph}
\begin{paragraph}\indent From the preceding discussion, we see that the set of prime numbers is \emph{not closed} with respect to multiplication, in the sense that multiplying two prime numbers gives you a composite. On the other hand, the set of composite numbers is closed under multiplication. On further thought, one can show that both sets are not closed under addition.  (It suffices to consider the counterexamples $2 + 7 = 9$ and $4 + 9 = 13$.)
\end{paragraph}
\begin{paragraph}\indent If $P^{\alpha}$ is the largest power of a prime $P$ that divides an integer $N$, i.~e.~ $P^{\alpha} \mid N$ but $P^{\alpha + 1} \nmid N$, then this is denoted by $P^{\alpha} || N$.
\end{paragraph}
\begin{paragraph}\indent We now list down several important properties of prime numbers as they relate to divisibility.
\end{paragraph}
\begin{thm}\label{theorem3} If $P$ is a prime and $P \mid mn$, then either $P \mid m$ or $P \mid n$.
\end{thm}
\begin{cor}\label{corollary1} If $P$ is a prime number and $P \mid m_1 m_2 \cdots m_n$, then $P \mid m_i$ for some $i$, $1 \leq i \leq n$.
\end{cor}
\begin{cor}\label{corollary2} If $P, Q_1, Q_2, \ldots, Q_n$ are all primes and $P \mid Q_1 Q_2 \cdots Q_n$, then $P = Q_i$ for some $i$, $1 \leq i \leq n$.
\end{cor}
\begin{paragraph}\indent All roads now lead to the Fundamental Theorem of Arithmetic.
\end{paragraph}
\begin{thm}\label{FTA} Fundamental Theorem of Arithmetic \\
Every positive integer $N > 1$ can be represented uniquely as a product of primes, apart from the order in which the factors occur.
\end{thm}
\begin{paragraph}\indent The ``lexicographic representation" of a positive integer as a product of primes may be achieved via what is called the \emph{canonical factorization}.
\end{paragraph}
\begin{cor}\label{corollary3} Any positive integer $N > 1$ can be written uniquely in the canonical factorization
\begin{displaymath} N = {P_1}^{{\alpha}_1} {P_2}^{{\alpha}_2} \cdots {P_r}^{{\alpha}_r} = \prod_{i=1}^r {{P_i}^{{\alpha}_i}}
\end{displaymath}
where, for $i = 1, 2, \ldots, r$, each ${\alpha}_i$ is a positive integer and each $P_i$ is a prime, with $P_1 < P_2 < \ldots < P_r$.
\end{cor}
\begin{paragraph}\indent We illustrate these with some examples.
\end{paragraph}
\begin{exmpl}\label{example3} The canonical factorization of the integer $36$ is $36 = 2^2 \cdot 3^2$. Meanwhile, the canonical factorization for the integer $1024$ is $1024 = 2^{10}$, while for $2145$ it is $2145 = 3^1 \cdot 5^1 \cdot 11^1 \cdot 13^1$, written simply as $2145 = 3 \cdot 5 \cdot 11 \cdot 13$.
\end{exmpl}
\begin{paragraph}\indent Functions which are defined for all positive integers $n$ are called \emph{arithmetic functions}, or \emph{number-theoretic functions}, or \emph{numerical functions}. Specifically, a \emph{number-theoretic function} $f$ is one whose domain is the positive integers and whose range is a subset of the complex numbers.
\end{paragraph}
\begin{paragraph}\indent We now define three important number-theoretic functions.
\end{paragraph}
\begin{defn}\label{definition5} Let $n$ be a positive integer.  Define the number-theoretic functions $d(n), \sigma(n), \phi(n)$ as follows:
\begin{center}
$d(n)$ = the number of positive divisors of $n$, \\
$\sigma(n)$ = the sum of the positive divisors of $n$, \\
{
\small
$\phi(n)$ = the number of positive integers at most $n$ which are also relatively prime to $n$.
}
\end{center}
\end{defn}
\begin{paragraph}\indent It would be good to illustrate with some examples.
\end{paragraph}
\begin{exmpl}\label{example4} Consider the positive integer $n = 28$. Since the positive divisors of $28$ are $1$, $2$, $4$, $7$, $14$ and $28$, then by definition:
\begin{center} $d(28) = 6$
\end{center}
and
\begin{displaymath}
\sigma(28) = \sum_{d \mid 28}{d} = 1 + 2 + 4 + 7 + 14 + 28 = 56.
\end{displaymath}
Note that the following list contains all the positive integers less than or equal to \\
$n = 28$ which are also relatively prime to $n$: $L = \{1, 3, 5, 9, 11, 13, 15, 17, 19, 23, 25, 27\}$. By definition, $\phi(28) = 12$.
\end{exmpl}
\begin{paragraph}\indent For the first few integers,
\begin{center} $d(1) = 1$ \hspace{0.01truein} $d(2) = 2$ \hspace{0.01truein} $d(3) = 2$ \hspace{0.01truein} $d(4) = 3$ \hspace{0.01truein} $d(5) = 2$ \hspace{0.01truein} $d(6) = 4$
\end{center}
while
\begin{center} $\sigma(1) = 1$ \hspace{0.01truein} $\sigma(2) = 3$ \hspace{0.01truein} $\sigma(3) = 4$ \hspace{0.01truein} $\sigma(4) = 7$ \hspace{0.01truein} $\sigma(5) = 6$ \hspace{0.01truein} $\sigma(6) = 12$
\end{center}
and
\begin{center} $\phi(1) = 1$ \hspace{0.01truein} $\phi(2) = 1$ \hspace{0.01truein} $\phi(3) = 2$ \hspace{0.01truein} $\phi(4) = 2$ \hspace{0.01truein} $\phi(5) = 4$ \hspace{0.01truein} $\phi(6) = 3$.
\end{center}
\end{paragraph}
\begin{paragraph}\indent Note that the functions $d(n), \sigma(n)$, and $\phi(n)$ are not monotonic, and their functional values at $n = 1$ is also $1$.  Also, for at least the first $3$ primes $p = 2$, $3$ and $5$, $d(p) = 2$, $\sigma(p) = p + 1$, and $\phi(p) = p - 1$.
\end{paragraph}
\begin{paragraph}\indent We shall now introduce the notion of a \emph{multiplicative} number-theoretic function.
\end{paragraph}
\begin{defn}\label{definition6} A function $F$ defined on $\mathbb N$ is said to be multiplicative if for all $m, n \in \mathbb N$ such that $\gcd(m, n) = 1$, we have
\begin{center} $F(mn) = F(m)F(n)$. \\
\end{center}
\end{defn}
\begin{exmpl}\label{example5} Let the function $F$ be defined by $F(n) = n^k$ where $k$ is a fixed positive integer. Then $F(mn) = (mn)^k = m^k n^k = F(m)F(n)$. We have therefore shown that $F$ is multiplicative. Moreover, the condition $\gcd(m, n) = 1$ is not even required for the series of equalities above to hold.  We call $F$ in this example a \emph{totally multiplicative function}.
\end{exmpl}
\begin{paragraph}\indent It turns out that the three number-theoretic functions we introduced in Definition \ref{definition5} provide us with more examples of multiplicative functions.
\end{paragraph}
\begin{thm}\label{theorem4} The functions $d, \sigma$ and $\phi$ are multiplicative functions.
\end{thm}
\begin{paragraph}\indent Multiplicative functions are completely determined by their values at prime powers. Given a positive integer $n$'s canonical factorization
\begin{displaymath} n = \prod_{i=1}^r {{P_i}^{{\alpha}_i}},
\end{displaymath}
then if $F$ is a multiplicative function, we have
\begin{displaymath} F(n) = \prod_{i=1}^r {F({P_i}^{{\alpha}_i})}. \end{displaymath}
This last assertion follows from the fact that prime powers derived from the canonical factorization of $n$ are pairwise relatively prime.
\end{paragraph}
\begin{paragraph}\indent The next theorem follows from Definition \ref{definition5}, and Theorem \ref{theorem4} as well.
\end{paragraph}
\begin{thm}\label{theorem5} If $n = \displaystyle\prod_{i=1}^r {{P_i}^{{\alpha}_i}}$ is the canonical factorization of $n > 1$, then
\begin{center}
$d(n) = \displaystyle\prod_{i=1}^r \left({\alpha}_i + 1\right)$, \\
$\sigma(n) = \displaystyle\prod_{i=1}^r {\sigma({P_i}^{{\alpha}_i})} = \prod_{i=1}^r \left(\frac{{P_i}^{{\alpha}_i + 1} - 1}{P_i - 1}\right)$, \\
$\phi(n) = n \displaystyle\prod_{i=1}^r \left(1 - {\frac{1}{P_i}}\right)$. \\
\end{center}
\end{thm}
\begin{paragraph}\indent We illustrate with several examples, continuing from Example \ref{example3}.
\end{paragraph}
\begin{exmpl}\label{example6} The integer $36 = 2^2\cdot3^2$ has
\begin{center} $d(36) = (2 + 1)(2 + 1) = 3\cdot3 = 9$ and $\sigma(36) = (\frac{2^3 - 1}{2 - 1})(\frac{3^3 - 1}{3 - 1}) = 7\cdot13 = 91$
\end{center}
and $\phi(36) = 36(1 - \frac{1}{2})(1 - \frac{1}{3}) = 12$, while for the integer $1024 = 2^{10}$ one has
\begin{center} $d(1024) = 10 + 1 = 11$, $\sigma(1024) = \frac{2^{11} - 1}{2 - 1} = 2047 = 23\cdot89$
\end{center}
and $\phi(1024) = 1024(1 - \frac{1}{2}) = 512 = 2^5$. 
\begin{paragraph}\indent Lastly, we have for the integer $2145 = 3\cdot5\cdot11\cdot13$ the following:
\begin{center} $d(2145) = (1 + 1)(1 + 1)(1 + 1)(1 + 1) = 2^4 = 16$
\end{center}
and
\begin{center} $\sigma(2145) =(3 + 1)(5 + 1)(11 + 1)(13 + 1) = 4032$
\end{center}
while
\begin{center} $\phi(2145) = 2145(1 - \frac{1}{3})(1 - \frac{1}{5})(1 - \frac{1}{11})(1 - \frac{1}{13}) = 960$.
\end{center}
\end{paragraph}
\end{exmpl}
\begin{paragraph}\indent The following corollary follows immediately from Theorem \ref{theorem5}:
\end{paragraph}
\begin{cor}\label{corollary4} Let $P$ be a prime number and $k$ a fixed positive integer. Then
\begin{displaymath} d(P^k) = k + 1, \sigma(P^k) = \sum_{i=0}^k {P^i} = \frac{P^{k + 1} - 1}{P - 1} \\
$$and$$ \phi(P^k) = P^k (1 - \frac{1}{P}) = P^{k - 1}(P - 1).
\end{displaymath}
\end{cor}
Note from Corollary \ref{corollary4} that for odd prime powers, the number and sum of divisors may or may not be prime, while $\phi(P^k)$ is always composite for $k > 1$.
\begin{paragraph}\indent Divisibility gives rise to an equivalence relation on the set of integers, defined by the \emph{congruence relation}.
\end{paragraph}
\begin{defn}\label{definition7} Let $m$ be a fixed positive integer. Two integers $A$ and $B$ are said to be congruent modulo $m$, written as $A \equiv B \pmod m$, if $m \mid (A - B)$; that is, provided that $A - B = km$ for some integer $k$. When $m \nmid (A - B)$, we say that $A$ is incongruent to $B$ modulo $m$, and we denote this by $A \not\equiv B \pmod m$.
\end{defn}
\begin{exmpl}\label{example7} Let us take $m = 3$. We can see that
\begin{center} $14 \equiv 5 \pmod 3$, $-9 \equiv 0 \pmod 3$, and $35 \equiv -7 \pmod 3$
\end{center}
because $14 - 5 = 3\cdot3$, $-9 - 0 = (-3)\cdot3$, and $35 - (-7) = 14\cdot3$.
\begin{paragraph}\indent On the other hand, $1200 \not\equiv 2 \pmod{3}$ because $3$ does not divide \\ $1200 - 2 = 1198$.
\end{paragraph}
\end{exmpl}
\begin{paragraph}\indent We now introduce two more additional number-theoretic functions.
\end{paragraph}
\begin{defn}\label{definition8} Let $n$ be a positive integer. Then $\omega(n)$ is the number of distinct prime factors of $n$, i.e. $\omega(n) = \sum_{{P_i} \mid n}{1}$ where each $P_i$ is prime. Furthermore, $\Omega(n)$ is the number of primes that divide $n$, counting multiplicities. That is, if $n$ has canonical factorization $n = \displaystyle\prod_{i=1}^r {{P_i}^{{\alpha}_i}}$, then
\begin{displaymath}
\Omega(n) = \alpha_1 + \alpha_2 + \ldots + \alpha_r = \sum_{i=1}^r {\alpha_i} = \sum_{{P_i}^{\alpha}||N}{\alpha}.
\end{displaymath}
\end{defn}
\begin{exmpl}\label{example8} Let us consider $n = 36 = 2^2 \cdot 3^2$. Since it has two distinct prime factors (namely $2$ and $3$), we have $\omega(36) = 2$. On the other hand, $\Omega(36) = 4$ since its total number of prime factors, counting multiplicities, is four. For $m = 1024 = 2^{10}$, we have $\omega(1024) = 1$ and $\Omega(1024) = 10$, while for $k = 2145 = 3 \cdot 5 \cdot 11 \cdot 13$, the functions have values $\omega(2145) = 4$ and $\Omega(2145) = 4$.
\end{exmpl}
\begin{paragraph}\indent Notice in Example \ref{example8} that the number of distinct prime factors is less than or equal to the total number of prime factors (counting multiplicities).  In general, it is true that $\Omega(n) \geq \omega(n)$ for all positive integers $n$.
\end{paragraph}
\begin{paragraph}\indent In number theory, \emph{asymptotic density} or \emph{natural density} is one of the possibilities to measure how large is a subset of the set of natural numbers $\mathbb N$.  Intuitively, we feel that there are ``more" odd numbers than perfect squares; however, the set of odd numbers is not in fact ``bigger" than the set of perfect squares: both sets are infinite and countable and can therefore be put in one-to-one correspondence. Clearly, we need a better way to formalize our intuitive notion.
\end{paragraph}
\begin{paragraph}\indent Let $A$ be a subset of the set of natural numbers $\mathbb N$.  If we pick randomly a number from the set $\left\{1, 2, \ldots, n\right\}$, then the probability that it belongs to $A$ is the ratio of the number of elements in the set $A \bigcap \left\{1, 2, \ldots, n\right\}$ and $n$.  If this probability tends to some limit as $n$ tends to infinity, then we call this limit the \emph{asymptotic density} of $A$. We see that this notion can be understood as a kind of probability of choosing a number from the set $A$. Indeed, the asymptotic density (as well as some other types of densities) is studied in \textbf{probabilistic number theory}.
\end{paragraph}
\begin{paragraph}\indent We formalize our definition of \emph{asymptotic density} or simply \textbf{density} in what follows:
\end{paragraph}
\begin{defn}\label{AsymptoticDensity}
A sequence $a_1, a_2, \ldots, a_n$ with the $a_j$ positive integers and \\ 
$a_j < a_{j + 1}$ for all $j$, has \textbf{natural density} or \textbf{asymptotic density} $\alpha$, where \\ 
$0 \leq \alpha \leq 1$, if the proportion of natural numbers included as some $a_j$ is asymptotic to $\alpha$.  More formally, if we define the counting function $A(x)$ as the number of $a_j$'s with $a_j < x$ then we require that $A(x) \sim {\alpha}x$ as $x \rightarrow +{\infty}$.
\end{defn}
\section{The Abundancy Index}
\begin{paragraph}\indent As discussed in the literature review, the search for perfect numbers led mathematicians to consider the rather interesting quantity called the \emph{abundancy index/ratio}.
\end{paragraph}
\begin{defn}\label{definition9} The abundancy index/ratio of a given positive integer $n$ is defined as $I(n) = \displaystyle\frac{\sigma(n)}{n}$.
\end{defn}
\begin{exmpl}\label{example9} $I(36) = \displaystyle\frac{\sigma(36)}{36} = \displaystyle\frac{91}{36}$, while
\begin{center} 
$I(1024) = \displaystyle\frac{\sigma(1024)}{1024} = \displaystyle\frac{2047}{1024}$ 
\end{center}
and 
\begin{center}
$I(2145) = \displaystyle\frac{\sigma(2145)}{2145} = \displaystyle\frac{4032}{2145}$.
\end{center}
\end{exmpl}
\begin{paragraph}\indent We note that the abundancy index is also a multiplicative number-theoretic function because $\sigma$ is multiplicative.
\end{paragraph}
\begin{paragraph}\indent Looking back at Definition \ref{perfect}, it is clear that a number $N$ is perfect if and only if its abundancy index $I(N)$ is $2$. It is somewhat interesting to consider the cases when $I(N) \neq 2$.
\end{paragraph}
\begin{defn}\label{definition10} If the abundancy index $I(N) < 2$, then $N$ is said to be \emph{deficient}, while for $I(N) > 2$, $N$ is said to be \emph{abundant}.
\end{defn}
\begin{exmpl}\label{example10} Referring to Example \ref{example9}, $36$ is abundant since $I(36) = \frac{91}{36} > 2$ while $1024$ and $2145$ are deficient since $I(1024) = \frac{2047}{1024} < 2$ and $I(2145) = \frac{4032}{2145} < 2$.
\end{exmpl}
\begin{remrk}\label{remark1} Giardus Ruffus conjectured in $1521$ that most odd numbers are deficient. In $1975$, C. W. Anderson \cite{A1} proved that this is indeed the case by showing that the density of odd deficient numbers is at least $\frac{48 - 3{\pi}^2}{32 - {\pi}^2} \approx 0.831$.  On the other hand, Marc Del$\ddot{e}$glise (in $1998$ \cite{D11}) gave the bounds $0.2474 < A(2) < 0.2480$ for the density $A(2)$ of abundant integers. Kanold ($1954$) \cite{K331} showed that the density of odd perfect numbers is 0.
\end{remrk}
\begin{paragraph}\indent We list down several important lemmas describing useful properties of the abundancy index.
\end{paragraph}
\begin{lemm}\label{lemma1} $\displaystyle\frac{\sigma(n)}{n} = \displaystyle\sum_{d \mid n}{\frac{1}{d}}$
\end{lemm}
\begin{proof} 
Straightforward: $\displaystyle\frac{\sigma(n)}{n} = {\displaystyle\frac{1}{n}}{\sum_{d \mid n}{d}} = {\displaystyle\frac{1}{n}}{\sum_{d \mid n}{\frac{n}{d}}} = \displaystyle\sum_{d \mid n}{\frac{1}{d}}$.
\end{proof}
\begin{lemm}\label{lemma2} If $m \mid n$ then $\displaystyle\frac{\sigma(m)}{m} \leq \displaystyle\frac{\sigma(n)}{n}$, with equality occurring if and only if $m = n$.
\end{lemm}
\begin{paragraph}\indent Essentially, Lemma \ref{lemma2} says that any (nontrivial) multiple of a perfect number is abundant and every (nontrivial) divisor of a perfect number is deficient.
\end{paragraph}
\begin{lemm}\label{lemma3} The abundancy index takes on arbitrarily large values.
\end{lemm}
\begin{proof} Consider the number $n!$. By Lemma \ref{lemma1}, we have $\displaystyle\frac{\displaystyle\sigma(n!)}{n!} = \displaystyle\sum_{d \mid n!}{\frac{1}{d}} \geq \displaystyle\sum_{i=1}^n {\frac{1}{i}}$. Since the last quantity is a partial sum of a harmonic series which diverges to infinity, $\displaystyle\frac{\displaystyle\sigma(n!)}{n!}$ can be made as large as we please.
\end{proof}
\begin{lemm}\label{lemma4} For any prime power $P^{\alpha}$, the following inequalities hold:
\begin{center}
$1 < \displaystyle\frac{P + 1}{P} < \displaystyle\frac{\displaystyle\sigma(P^{\alpha})}{P^{\alpha}} < \displaystyle\frac{P}{P - 1}$.
\end{center}
\end{lemm}
\begin{paragraph}\indent The proof of Lemma \ref{lemma4} follows directly from Corollary \ref{corollary4}.
\end{paragraph}
\begin{paragraph}\indent Certainly, we can find abundancy indices arbitrarily close to $1$ because \\ $I(p) = \displaystyle\frac{p + 1}{p}$ for all primes $p$. By Lemma \ref{lemma3}, and since the abundancy index of a positive integer is a rational number, one would then desire to know the ``distribution" of these ratios in the interval $(1, \infty)$. The next few results summarize much of what is known about the ``distribution" of these ratios.
\end{paragraph}
\begin{thm}\label{theorem6} (Laatsch) The set of abundancy indices $I(n)$ for $n > 1$ is dense in the interval $(1, \infty)$.
\end{thm}
\begin{paragraph}\indent However, not all of the rationals from the interval $(1, \infty)$ are abundancy indices of some integer.  This is due to the following lemma from Weiner:
\end{paragraph}
\begin{lemm}\label{lemma5} (Weiner) If $\gcd(m, n) = 1$ and $n < m < \sigma(n)$, then $\frac{m}{n}$ is not the abundancy index of any integer.
\end{lemm}
\begin{proof} Suppose $\displaystyle\frac{m}{n} = \displaystyle\frac{\sigma(k)}{k}$ for some integer $k$. Then $km = n \sigma(k)$ which implies that $n \mid km$, and so $n \mid k$ since $m$ and $n$ are coprime. Hence, by Lemma \ref{lemma2} we have
\begin{displaymath}
\frac{\sigma(n)}{n} \leq \frac{\sigma(k)}{k} = \frac{m}{n}
\end{displaymath}
which yields $\sigma(n) \leq m$ - a contradiction to the initial assumption that $m < \sigma(n)$.
\end{proof}
\begin{paragraph}\indent It is now natural to define the notion of an \emph{abundancy outlaw}.
\end{paragraph}
\begin{defn}\label{definition11} A rational number greater than $1$ is said to be an \emph{abundancy outlaw} if it fails to be in the range of the function $I(n)$.
\end{defn}
\begin{paragraph}\indent One can use the previous lemmas to establish an equally interesting theorem about the distribution of abundancy outlaws.
\end{paragraph}
\begin{thm}\label{theorem7} (Weiner, Ryan) The set of abundancy outlaws is dense in the interval $(1, \infty)$.
\end{thm}
\begin{paragraph}\indent Upon inspecting the results of Theorems \ref{theorem6} and \ref{theorem7}, it appears that the scenario for abundancy indices and outlaws is both complex and interesting. We shall take a closer look into the nature of abundancy outlaws in Chapter $4$.
\end{paragraph}
\section{Even Perfect Numbers}
\begin{paragraph}\indent The Greek mathematician \emph{Euclid} was the first to categorize the perfect numbers. He noticed that the first four perfect numbers have the very specific forms:
\begin{center}
$6 = 2^1 (1 + 2) = 2\cdot3$ \\
$28 = 2^2 (1 + 2 + 2^2) = 4\cdot7$ \\
$496 = 2^4 (1 + 2 + 2^2 + 2^3 + 2^4) = 16\cdot31$ \\
$8128 = 2^6 (1 + 2 + 2^2 + \ldots + 2^6) = 64\cdot127$.
\end{center}
\end{paragraph}
\begin{paragraph}\indent Notice that the numbers $90 = 2^3 (1 + 2 + 2^2 + 2^3) = 8\cdot15$ and \\
$2016 = 2^5 (1 + 2 + 2^2 + \ldots + 2^5) = 32\cdot63$ are missing from this list. Euclid pointed out that this is because $15 = 3\cdot5$ and $63 = 3^2\cdot7$ are both composite, whereas the numbers $3$, $7$, $31$ and $127$ are all prime.
\end{paragraph}
\begin{paragraph}\indent According to Book IX, proposition $36$ of Euclid's \emph{Elements}: ``\emph{If as many numbers as we please beginning from a unit be set out continuously in double proportion, until the sum of all becomes a prime, and if the sum multiplied into the last make some number, the product will be perfect.}" \cite{O41}
\end{paragraph}
\begin{paragraph}\indent 	This observation is stated in a slightly more compact form as follows:
\end{paragraph}
\begin{thm}\label{theorem8} (Euclid) If $2^n - 1$ is prime, then $N = 2^{n-1}(2^n - 1)$ is perfect.
\end{thm}
\begin{proof} Clearly the only prime factors of $N$ are $2^n - 1$ and $2$. Since $2^n - 1$ occurs as a single prime, we have simply that $\sigma(2^n - 1) = 1 + (2^n - 1) = 2^n$, and thus
\begin{center}
$\sigma(N) = \sigma(2^{n-1})\sigma(2^n - 1) = (\frac{2^n - 1}{2 - 1}) 2^n = 2^n (2^n - 1) = 2N$.
\end{center}
Therefore, $N$ is perfect.
\end{proof}
\begin{paragraph}\indent The task of finding perfect numbers, then, is intimately linked with finding primes of the form $2^n - 1$. Such numbers are referred to as Mersenne primes, after the $17$th-century monk Marin Mersenne, a contemporary of Descartes, Fermat, and Pascal. He investigated these unique primes as early as $1644$. Mersenne knew that $2^n - 1$ is prime for $n = 2$, $3$, $5$, $11$, $13$, $17$, and $19$ - and, more brilliantly, conjectured the cases $n = 31$, $67$, $127$, $257$. It took almost two hundred years to test these numbers.
\end{paragraph}
\begin{paragraph}\indent There is one important criterion used to determine the primality of Mersenne numbers:
\end{paragraph}
\begin{lemm}\label{lemma6} (Cataldi-Fermat)
If $2^n - 1$ is prime, then $n$ itself is prime.
\end{lemm}
\begin{proof}
Consider the factorization of $x^n - 1 = (x - 1)(x^{n - 1} + \ldots + x + 1)$. Suppose $n = rs$, where $r, s > 1$. Then $2^n - 1 = (2^r)^s - 1 = (2^r - 1)((2^r)^{s - 1} + \ldots + 2^r + 1)$, so that $(2^r - 1) \mid (2^n - 1)$ which is prime, a contradiction.
\end{proof}
\begin{paragraph}\indent 	Note that the converse of Lemma \ref{lemma6} is not true - the number $2^{11} - 1$ which is equal to $2047 = 23\cdot89$ is composite, yet $11$ is prime, for instance.
\end{paragraph}
\begin{paragraph}\indent 	Should all perfect numbers be of Euclid's type?  Leonard Euler, in a posthumous paper, proved that every \emph{even} perfect number is of this type. \cite{M37}
\end{paragraph}
\begin{thm}\label{theorem9} (Euler) If $N$ is an even perfect number, then $N$ can be written in the form $N = 2^{n-1}(2^n - 1)$, where $2^n - 1$ is prime.
\end{thm}
\begin{proof}
Let $N = 2^{n - 1} m$ be perfect, where $m$ is odd; since $2$ does not divide $m$, it is relatively prime to $2^{n - 1}$, and
\begin{center}
$\sigma(N) = \sigma(2^{n - 1} m) = \sigma(2^{n - 1})\sigma(m) = (\frac{2^n - 1}{2 - 1})\sigma(m) = (2^n - 1)\sigma(m)$.
\end{center}
$N$ is perfect so $\sigma(N) = 2N = 2(2^{n - 1} m) = 2^n m$, and with the above, \\
$2^n m = (2^n - 1)\sigma(m)$. Since $2^n - 1$ is odd, $(2^n - 1) \mid m$, so we can write $m = (2^n - 1) k$. Now $(2^n - 1) \sigma(m) = 2^n (2^n - 1) k$, which implies $\sigma(m) = 2^n k = (2^n - 1) k + k = m + k$.  But $k \mid m$ so $\sigma(m) = m + k$ means $m$ has only two (2) divisors, which further implies that $k = 1$. Therefore, $\sigma(m) = m + 1$ and $m$ is prime. Since $(2^n - 1) \mid m$, $2^n - 1 = m$. Consequently, $N = 2^{n - 1}(2^n - 1)$ where $2^n - 1$ is prime.
\end{proof}
\begin{paragraph}\indent Even perfect numbers have a number of nice little properties.  We list down several of them here, and state them without proof \cite{V56}:
\end{paragraph}
\begin{itemize}
\item If $N$ is an even perfect number, then $N$ is triangular.
\item If $N = 2^{n - 1}(2^n - 1)$ is perfect then $N = 1^3 + 3^3 + \ldots + (2^{\frac{n - 1}{2}} - 1)^3$.
\item If $N = 2^{n - 1}(2^n - 1)$ is perfect and $N$ is written in base $2$, then it has $2n - 1$ digits, the first $n$ of which are unity and the last $n - 1$ are zero.
\item Every even perfect number ends in either $6$ or $8$.
\item (\emph{Wantzel}) The iterative sum of the digits (i.e. \emph{digital root}) of an even perfect number (other than 6) is one.
\end{itemize}
\begin{paragraph}\indent Today $44$ perfect numbers are known, $2^{88}(2^{89} - 1)$ being the last to be discovered by hand calculations in $1911$ (although not the largest found by hand calculations), all others being found using a computer.  In fact computers have led to a revival of interest in the discovery of Mersenne primes, and therefore of perfect numbers.  At the moment the largest known Mersenne prime is $2^{32582657} - 1$.  It was discovered in September of $2006$ and this, the 44th such prime to be discovered, contains more than $9.8$ million digits.  Worth noting is the fact that although this is the $44$th to be discovered, it may not correspond to the $44$th perfect number as not all smaller cases have been ruled out.
\end{paragraph}
\section{Odd Perfect Numbers}
\begin{paragraph}\indent The Euclid-Euler theorem from Section $2.3$ takes care of the even perfect numbers.  What about the \emph{odd perfect numbers}?
\end{paragraph}
\begin{paragraph}\indent Euler also tried to make some headway on the problem of whether odd perfect numbers existed.  He proved that any odd perfect number $N$ had to have the form
\end{paragraph}
\begin{center} $N = (4m + 1)^{4k + 1} b^2$
\end{center}
where $4m + 1$ is prime and $\gcd(4m + 1, b) = 1$.
\begin{paragraph}\indent In Section $2.3$, we have followed some of the progress of finding even perfect numbers but there were also attempts to show that an odd perfect number could not exist. The main thrust of progress here has been to show the minimum number of distinct prime factors that an odd perfect number must have.  Sylvester worked on this problem and wrote:
\end{paragraph}
\begin{center} $\ldots$ \emph{the existence of [an OPN] - its escape, so to say, from the complex \\
web of conditions which hem it in on all sides - would be little short of a miracle}.
\end{center}
(The reader is referred to Section $1.1.2$ of this thesis for a survey of the most recent conditions which an OPN must satisfy, if any exists.)
\begin{paragraph}\indent We give a proof of Euler's characterization of OPNs here \cite{M37}:
\end{paragraph}
\begin{thm}\label{theorem10} (Euler) Let $N$ be an OPN. Then the prime factorization of $N$ takes the form $N = q^{4e + 1}{p_1}^{2a_1}\cdots{p_r}^{2a_r}$, where $q \equiv 1 \pmod 4$. 
\end{thm}
\begin{proof} Let $N = {l_1}^{e_1}{l_2}^{e_2}\cdots{l_s}^{e_s}$ for some primes $l_1, l_2, \ldots, l_s$. Since $N$ is odd, all $l_i$ are odd. Finally, $\sigma(N) = 2N$. Since $\sigma(N) = \sigma({l_1}^{e_1}{l_2}^{e_2}\cdots{l_s}^{e_s}) = \sigma({l_1}^{e_1})\sigma({l_2}^{e_2})\cdots\sigma({l_s}^{e_s})$, we take a look at $\sigma(l^e) = 1 + l + l^2 + \ldots + l^e$, a sum of $e + 1$ odd numbers. This is odd only if $e$ is even. Since $\sigma({l_1}^{e_1}{l_2}^{e_2}\cdots{l_s}^{e_s}) = \sigma({l_1}^{e_1})\sigma({l_2}^{e_2})\cdots\sigma({l_s}^{e_s}) = 2 {l_1}^{e_1}{l_2}^{e_2}\cdots{l_s}^{e_s}$, we can only get \emph{one} factor of 2. So the $e_i$ are even, all except one, say $e_1$. So $N = {l_1}^{e_1} {p_1}^{2a_1}\cdots{p_r}^{2a_r}$. \\
\begin{paragraph}\indent We have $2 \mid \sigma({l_1}^{e_1})$ but $4 \nmid \sigma({l_1}^{e_1})$. Since $l_1$ is odd, $e_1$ is odd. Now, modulo 4, we see that either $l_1 \equiv 1 \pmod 4$ or $l_1 \equiv {-1} \pmod 4$. But if $l_1 \equiv {-1} \pmod 4$, then
\begin{center}
$\sigma({l_1}^{e_1}) = 1 + l_1 + {l_1}^2 + {l_1}^3 + \ldots + {l_1}^{e_1 - 1} + {l_1}^{e_1}$    \\ $\equiv 1 + (-1) + 1 + (-1) + \ldots + 1 + (-1) \equiv 0 \pmod 4$,
\end{center}
which is clearly a contradiction since $4 \nmid \sigma({l_1}^{e_1})$. Thus, $l_1 \equiv 1 \pmod 4$. Now, $\sigma({l_1}^{e_1}) = 1 + l_1 + {l_1}^2 + {l_1}^3 + \ldots + {l_1}^{e_1 - 1} + {l_1}^{e_1} \equiv 1 + 1 + 1 + 1 + \ldots + 1 + 1 \equiv {e_1 + 1} \pmod 4$. Since $e_1$ is odd, either $e_1 + 1 \equiv 0 \pmod 4$ or $e_1 + 1 \equiv 2 \pmod 4$. If $e_1 + 1 \equiv 0 \pmod 4$, then $4 \mid \sigma({l_1}^{e_1})$ which is again a contradiction. So $e_1 + 1 \equiv 2 \pmod 4$ $\Leftrightarrow$ $e_1 + 1 = 4e + 2$, that is, $e_1 = 4e + 1$.  Consequently, $N = q^{4e + 1}{p_1}^{2a_1}\cdots{p_r}^{2a_r}$, for $q \equiv 1 \pmod 4$.
\end{paragraph}
\end{proof}
\begin{paragraph}\indent We call $q$ in Theorem \ref{theorem10} the \textbf{special/Euler prime} of $N$, while $q^{4e + 1}$ will be called the \textbf{Euler's factor} of $N$.
\end{paragraph}
\begin{paragraph}\indent Interestingly, it is possible to show that no two consecutive integers can be both perfect, using the Euclid-Euler theorem on the form of even perfect numbers from Section $2.3$ and Euler's characterization of odd perfect numbers in this section.  We shall give a discussion of the proof of this interesting result in Chapter $3$.  Moreover, we shall give there the (easy) proof of the fact that an odd perfect number must be a sum of two squares.  We also give congruence conditions for the existence of odd perfect numbers in the next chapter.  For the most part, the proofs will be elementary, requiring only an intermediate grasp of college algebra and the concepts introduced in this chapter.
\end{paragraph}
\chapter{OPN Solution Attempts 1: \\ Some Old Approaches}
In the previous chapters, we derived the possible forms of even and odd perfect numbers, and also surveyed the most recent results on the conditions necessitated by the existence of OPNs. In this chapter, we introduce the reader to the flavor of the mathematical techniques used to formulate theorems about OPNs by researchers who lived prior to the $21$st century.
\begin{paragraph}\indent The following are some of the traditional attempts made by mathematicians (both amateur and professional) to prove or disprove the OPN Conjecture in the pre-$21$st century:
\end{paragraph}
\begin{itemize}
\item Increasing the lower bound for the number of distinct prime factors, $\omega(N)$, that an OPN $N$ must have;
\item Increasing the lower bound for the magnitude of the smallest possible OPN, if one exists;
\item Deriving congruence conditions for the existence of OPNs.
\end{itemize}
\begin{paragraph}\indent All these itemized approaches attempt to derive a contradiction amongst the stringent conditions that an OPN must satisfy.
\end{paragraph}
\section{Increasing the Lower Bound for $\omega(N)$}
Recall that, from Lemma \ref{lemma4}, we have the following strict inequality for the abundancy index of a prime power: $\displaystyle\frac{\sigma(p^{\alpha})}{p^{\alpha}} < \displaystyle\frac{p}{p - 1}$.
\begin{paragraph}\indent This gives rise to the following lemma:
\end{paragraph}
\begin{lemm}\label{lemma7} If $N$ is a perfect number with canonical factorization $N = \displaystyle\prod_{i = 1}^{\omega(N)} {{P_i}^{{\alpha}_i}}$, then
\begin{displaymath} 2 < \prod_{i = 1}^{\omega(N)}{\frac{P_i}{P_i - 1}} = \prod_{i = 1}^{\omega(N)}\left(1 + \frac{1}{P_i - 1}\right).
\end{displaymath}
\end{lemm}
\begin{proof} This follows as an immediate consequence of Lemma \ref{lemma4} and the definition of perfect numbers.
\end{proof}
\begin{paragraph}\indent Note that Lemma \ref{lemma7} applies to both even and odd perfect numbers.
\end{paragraph}
\begin{paragraph}\indent The following is another very useful lemma:
\end{paragraph}
\begin{lemm}\label{lemma8} If $N$ is a perfect number with canonical factorization $N = \displaystyle\prod_{i = 1}^{\omega(N)} {{P_i}^{{\alpha}_i}}$, then
\begin{displaymath} 2 \geq \prod_{i = 1}^{\omega(N)}\left(1 + \frac{1}{P_i} + \ldots + \frac{1}{{P_i}^{{\beta}_i}}\right),
\end{displaymath}
where $0 \leq \beta_i \leq \alpha_i \hspace{0.02truein} \forall i$.
\end{lemm}
\begin{proof} This immediately follows from the definition of perfect numbers and the fact that the abundancy index for prime powers is an increasing function of the exponents.
\end{proof}
\begin{paragraph}\indent When considering OPNs, Lemmas \ref{lemma7} and \ref{lemma8} are very useful because they can yield lower bounds for $\omega(N)$. Indeed, it was in using these lemmas (together with some ingenuity) that pre-$21$st century mathematicians were able to successfully obtain ever-increasing lower bounds for the number of distinct prime factors of an OPN $N$.
\end{paragraph}
\begin{paragraph}\indent We now prove the classical result: ``\emph{An OPN must have at least three distinct prime factors}".
\end{paragraph}
\begin{thm}\label{theorem11} If $N$ is an OPN, then $\omega(N) \geq 3$.
\end{thm}
\begin{proof} Let $N$ be an OPN. Since prime powers are deficient, $\omega(N) \geq 2$. Suppose $\omega(N) = 2$. Since $N$ is odd, $P_1 \geq 3$ and $P_2 \geq 5$ where $N = {P_1}^{{\alpha}_1}{P_2}^{{\alpha}_2}$ is the canonical factorization of $N$. Using Lemma \ref{lemma7}:
\begin{displaymath}
2 < \prod_{i = 1}^{\omega(N)}{\frac{P_i}{P_i - 1}} = \prod_{i = 1}^2 {\frac{P_i}{P_i - 1}}
\end{displaymath}
\begin{displaymath} = \left(\frac{P_1}{P_1 - 1}\right)\left(\frac{P_2}{P_2 - 1}\right) = \frac{1}{1 - \frac{1}{P_1}}\frac{1}{1 - \frac{1}{P_2}} \leq \frac{1}{1 - \frac{1}{3}}\frac{1}{1 - \frac{1}{5}}
\end{displaymath}
\begin{displaymath} = \frac{1}{\frac{2}{3}}\frac{1}{\frac{4}{5}} = \frac{3}{2}\frac{5}{4} = \frac{15}{8} = 1.875 < 2
\end{displaymath}
\begin{paragraph}\indent Thus, the assumption $\omega(N) = 2$ for an OPN $N$ has resulted to the contradiction $2 < 2$. This contradiction shows that $\omega(N) \geq 3$.
\end{paragraph}
\end{proof}
\begin{paragraph}\indent More work is required to improve the result of Theorem \ref{theorem11} to $\omega(N) \geq 4$, if we are to use a similar method.
\end{paragraph}
\begin{thm}\label{theorem12} If $N$ is an OPN, then $\omega(N) \geq 4$.
\end{thm}
\begin{proof} Let $N$ be an OPN. By Theorem \ref{theorem11}, $\omega(N) \geq 3$. Assume $\omega(N) = 3$. Since $N$ is odd, $P_1 \geq 3$, $P_2 \geq 5$, and $P_3 \geq 7$ where $N = {P_1}^{{\alpha}_1}{P_2}^{{\alpha}_2}{P_3}^{{\alpha}_3}$ is the canonical factorization of $N$. Using Lemma \ref{lemma7}:
\begin{displaymath}
2 < \prod_{i = 1}^{\omega(N)}{\frac{P_i}{P_i - 1}} = \prod_{i = 1}^3 {\frac{P_i}{P_i - 1}}
\end{displaymath} 
\begin{displaymath} = \left(\frac{P_1}{P_1 - 1}\right)\left(\frac{P_2}{P_2 - 1}\right)\left(\frac{P_3}{P_3 - 1}\right) = \frac{1}{1 - \frac{1}{P_1}}\frac{1}{1 - \frac{1}{P_2}}\frac{1}{1 - \frac{1}{P_3}} \leq \frac{1}{1 - \frac{1}{3}}\frac{1}{1 - \frac{1}{5}}\frac{1}{1 - \frac{1}{7}}
\end{displaymath}
\begin{displaymath} = \frac{1}{\frac{2}{3}}\frac{1}{\frac{4}{5}}\frac{1}{\frac{6}{7}} = \frac{3}{2}\frac{5}{4}\frac{7}{6} = \frac{105}{48} = \frac{35}{16} = 2.1875
\end{displaymath}
whence we do not arrive at a contradiction.  Now, suppose $P_1 \geq 5$, $P_2 \geq 7$, and $P_3 \geq 11$. Using Lemma \ref{lemma7} again, we get $2 < (\frac{5}{4})(\frac{7}{6})(\frac{11}{10}) = \frac{77}{48}$, which is a contradiction. Consequently, we know that $P_1 = 3$. Then, with this additional information about $P_1$, if we assume that $P_2 \geq 7$, we arrive at $2 < (\frac{3}{2})(\frac{7}{6})(\frac{11}{10}) = \frac{231}{120} = \frac{77}{40}$ (by use of Lemma \ref{lemma7}) which is again a contradiction. Hence, we also know that $P_2 = 5$. Thus, $P_3 \geq 7$. Furthermore, using Lemma \ref{lemma7} again, the following inequality must be true:
\begin{displaymath} 2 < \frac{3}{2}\frac{5}{4}\frac{P_3}{P_3 - 1} \Longleftrightarrow \frac{16}{15} < \frac{P_3}{P_3 - 1}.
\end{displaymath}
Solving this last inequality, we get $P_3 < 16$. This inequality, together with the fact that $P_3$ is prime, gives us $3$ possible cases to consider: \\
\emph{Case 1}: $N = 3^{\alpha_1}5^{\alpha_2}7^{\alpha_3}$ \\
Since $N$ is odd, $4$ cannot divide $\sigma(N) = \sigma(3^{\alpha_1})\sigma(5^{\alpha_2})\sigma(7^{\alpha_3}) = 2N = 2\cdot3^{\alpha_1}5^{\alpha_2}7^{\alpha_3}$. In particular, $4 \nmid \sigma(3^{\alpha_1})$ and $4 \nmid \sigma(7^{\alpha_3})$. But $\sigma(3^{\alpha_1}) = 4$ for $\alpha_1 = 1$ and $\sigma(7^{\alpha_3}) = 8$ for $\alpha_3 = 1$. Consequently, $\alpha_1 \geq 2$ and $\alpha_3 \geq 2$. Now, by using Lemma \ref{lemma8} with $\beta_1 = 2$, $\beta_2 = 1$ and $\beta_3 = 2$, we have:
\begin{displaymath}
2 \geq \left(1 + \frac{1}{3} + \frac{1}{3^2}\right)\left(1 + \frac{1}{5}\right)\left(1 + \frac{1}{7} + \frac{1}{7^2}\right) = \frac{494}{245},
\end{displaymath}
which is a contradiction. Hence, there is no OPN of the form $N = 3^{\alpha_1}5^{\alpha_2}7^{\alpha_3}$. \\
\emph{Case 2}: $N = 3^{\alpha_1}5^{\alpha_2}{11}^{\alpha_3}$ \\
\emph{A.} Using Lemma \ref{lemma7} with $\alpha_2 = 1$ gives $2 < \displaystyle\frac{3}{2}\frac{\displaystyle\sigma(5^{\alpha_2})}{5^{\alpha_2}}\frac{11}{10} = \frac{3}{2}\frac{6}{5}\frac{11}{10} = \frac{99}{50}$, which is a contradiction. \\
\emph{B.} Let $\alpha_2 \geq 2$. For $\alpha_1 = 1, 2, 3$, we have $\sigma(3^{\alpha_1}) = 4, 13, 40$. Since $4 \nmid \sigma(N)$, $\alpha_1 \neq 1$ and $\alpha_1 \neq 3$.  Also, $13 \nmid \sigma(N) = 2N = 2\cdot3^{\alpha_1}5^{\alpha_2}{11}^{\alpha_3}$ and thus, $13 \nmid \sigma(3^{\alpha_1})$. Hence, $\alpha_1 \neq 2$, which implies that $\alpha_1 \geq 4$. Similarly, $\sigma({11}^{\alpha_3}) = 12$ for $\alpha_3 = 1$, which contradicts the fact that $4 \nmid \sigma({11}^{\alpha_3})$. Thus, $\alpha_3 \geq 2$. Now, using Lemma \ref{lemma8} with 
$\beta_1 = 4$, $\beta_2 = 2$ and $\beta_3 = 2$, we get:
\begin{displaymath}
2 \geq \left(1 + \frac{1}{3} + \frac{1}{3^2} + \frac{1}{3^3} + \frac{1}{3^4}\right)\left(1 + \frac{1}{5} + \frac{1}{5^2}\right)\left(1 + \frac{1}{11} + \frac{1}{{11}^2}\right) = \frac{4123}{2025},
\end{displaymath}
which is a contradiction. Hence, there is no OPN of the form $N = 3^{\alpha_1}5^{\alpha_2}{11}^{\alpha_3}$. \\
~\\
\emph{Case 3}: $N = 3^{\alpha_1}5^{\alpha_2}{13}^{\alpha_3}$ \\
\emph{A.} Using Lemma \ref{lemma7} with $\alpha_2 = 1$ gives $2 < \displaystyle\frac{3}{2}\frac{\displaystyle\sigma(5^{\alpha_2})}{5^{\alpha_2}}\frac{13}{12} = \frac{3}{2}\frac{6}{5}\frac{13}{12} = \frac{39}{20}$, which is a contradiction. \\
\emph{B.} Let $\alpha_2 \geq 2$. Similar to what we got from \emph{Case 2B}, we have $\alpha_1 \neq 1, 3$. \\
\emph{1.} $\alpha_1 = 2$ \\
Using Lemma \ref{lemma7}, we have: $2 < \displaystyle\frac{\displaystyle\sigma(3^{\alpha_1})}{3^{\alpha_1}}\frac{5}{4}\frac{13}{12} = \frac{13}{9}\frac{5}{4}\frac{13}{12} = \frac{845}{432}$, which is a contradiction. \\
\emph{2.} $\alpha_1 \geq 4$ \\
Suppose $\alpha_3 = 1$. Then $\sigma({13}^{\alpha_3}) = 14 = 2\cdot7$, yet $7 \nmid \sigma(N) = 2N = 2\cdot3^{\alpha_1}5^{\alpha_2}{13}^{\alpha_3}$. Therefore, $7 \nmid \sigma({13}^{\alpha_3})$, and then $\alpha_3 \geq 2$. Now, by using Lemma \ref{lemma8} with $\beta_1 = 4$, $\beta_2 = 2$ and $\beta_3 = 2$, we get:
\begin{displaymath}
2 \geq \left(1 + \frac{1}{3} + \frac{1}{3^2} + \frac{1}{3^3} + \frac{1}{3^4}\right)\left(1 + \frac{1}{5} + \frac{1}{5^2}\right)\left(1 + \frac{1}{13} + \frac{1}{{13}^2}\right) = \frac{228811}{114075},
\end{displaymath}
which is again a contradiction. Hence, there is no OPN of the form $N = 3^{\alpha_1}5^{\alpha_2}{13}^{\alpha_3}$.
\begin{paragraph}\indent Therefore, there is no OPN with exactly three ($3$) distinct prime factors, i.e. an OPN must have at least four ($4$) distinct prime factors.
\end{paragraph}
\end{proof}
\begin{remrk}\label{remark2} Using more recent findings on an upper bound for OPNs by Nielsen \cite{N39} and on a lower bound by Brent \emph{et al.} \cite{B3}, it is possible to extend the results in this section. Thus, ${10}^{300} < N < 2^{4^{\omega(N)}}$, and this gives
\begin{center}
$\omega(N) > \left(\displaystyle\frac{2 + \log(3) - \log(\log(2))}{\log(4)}\right) > 4.9804$,
\end{center} 
which implies that $\omega(N) \geq 5$ since $\omega(N)$ should be an integer. (There is a project underway at http://www.oddperfect.org, organized by William Lipp, which hopes to extend the lower bound for OPNs to ${10}^{500}$, or ${10}^{600}$ even.) Indeed, even more recently ($2006$), Nielsen \cite{N40} was able to show that $\omega(N) \geq 9$, ``ultimately [avoiding] previous computational results for [OPNs]".  
\end{remrk}
\section{Increasing the Lower Bound for an OPN $N$}
We begin with a very useful inequality (which we shall not prove here) that can yield our desired estimates for a lower bound on OPNs:
\begin{lemm}\label{lemma9} The Arithmetic Mean-Geometric Mean Inequality \\
Let $\left\{X_i\right\}$ be a sequence of (not necessarily distinct) positive real numbers. Then the following inequality must be true:
\begin{displaymath}
\frac{1}{n}\sum_{i=1}^n {X_i} \geq \left({\prod_{i=1}^n {X_i}}\right)^{\frac{1}{n}}.
\end{displaymath}
Equality holds if and only if all of the $X_i$'s are equal.
\end{lemm}
\begin{paragraph}\indent We now derive a crude lower bound for an OPN $N = \displaystyle\prod_{i=1}^{\omega(N)} {{P_i}^{\alpha_i}}$ in terms of the $\alpha_i$'s:
\end{paragraph}
\begin{lemm}\label{lemma10} Let $N = \displaystyle\prod_{i=1}^{\omega(N)} {{P_i}^{\alpha_i}}$ be an OPN. Then \\
\begin{center}
$N > \displaystyle\left(\frac{\displaystyle\prod_{i=1}^{\omega(N)}{\left(\alpha_i + 1\right)}}{\displaystyle 2}\right)^2$.
\end{center}
\end{lemm}
\begin{proof} $\displaystyle\sigma({P_i}^{\alpha_i}) = \displaystyle\sum_{j=0}^{\alpha_i}{{P_i}^j}$. Applying Lemma \ref{lemma9} and noting that prime powers of the $P_i$'s are distinct, we have for each $i$:
\begin{center} 
$\displaystyle\sum_{j=0}^{\alpha_i}{{P_i}^j} > \displaystyle(\alpha_i + 1)\prod_{j=0}^{\alpha_i}{{P_i}^{\displaystyle\frac{j}{\alpha_i + 1}}} = \displaystyle(\alpha_i + 1){\displaystyle P_i}^{\displaystyle\sum_{j=0}^{\alpha_i}{\frac{j}{\alpha_i + 1}}}$ \\ 
$= \displaystyle(\alpha_i + 1){P_i}^{\displaystyle\frac{\alpha_i(\alpha_i + 1)}{2(\alpha_i + 1)}} = \displaystyle(\alpha_i + 1){P_i}^{\displaystyle\frac{\alpha_i}{2}}$. 
\end{center}
Consequently, by multiplying across all $i$:
\begin{center}
$\displaystyle\prod_{i=1}^{\omega(N)}{\sigma({P_i}^{\alpha_i})} > \displaystyle\prod_{i=1}^{\omega(N)}{(\alpha_i + 1){P_i}^{\displaystyle\frac{\alpha_i}{2}}}$. But $\displaystyle\sqrt{N} = \displaystyle\prod_{i=1}^{\omega(N)}{{P_i}^{\displaystyle\frac{\alpha_i}{2}}}$, which means that \\
$N = \displaystyle\frac{\displaystyle\prod_{i=1}^{\omega(N)}{\sigma({P_i}^{\alpha_i})}}{2} > \displaystyle\frac{\displaystyle\prod_{i=1}^{\omega(N)}{\left(\alpha_i + 1\right)}\displaystyle\prod_{i=1}^{\omega(N)}{{P_i}^{\displaystyle\frac{\alpha_i}{2}}}}{2} = \displaystyle\sqrt{ N}\frac{\displaystyle\prod_{i=1}^{\omega(N)}{\left(\alpha_i + 1\right)}}{2}$. \\
\end{center}
Solving this last inequality for $N$ gives us the desired result. 
\end{proof}
\begin{remrk}\label{remark3} In the canonical factorization $N = \displaystyle\prod_{i=1}^{\omega(N)} {{P_i}^{\alpha_i}}$  of an OPN $N$, since $\alpha_i \geq 1$ for all $i$, from Lemma \ref{lemma10} we have the crude lower bound $N > 2^{{2\omega(N)} - 2}$, which, together with Nielsen's lower bound of $\omega(N) \geq 9$ for the number of distinct prime factors of $N$, yields the lower bound $N > 2^{16} = 65536$ for the magnitude of the smallest possible OPN. This lower bound can, of course, be improved. Indeed, Brent, \emph{et al.} \cite{B3} in $1991$ showed that it must be the case that $N > 10^{300}$.
\end{remrk}
\begin{remrk}\label{remark4} Note that nowhere in the proof of Lemma \ref{lemma10} did we use the fact that $N$ is odd. Hence, Lemma \ref{lemma10} applies to even perfect numbers as well.
\end{remrk}
\begin{paragraph}\indent We can make use of the results on the lower bound for the number of distinct prime factors (latest result is at $9$ by Nielsen), lower bound for the smallest prime factor (currently at $3$ - mathematicians are still unable to rule out the possibility that an OPN may be divisible by $3$), and the nature of the exponents (the special/Euler prime has the sole odd exponent while the rest of the primes have even exponents) to derive a larger lower bound for an OPN.
\end{paragraph}
\begin{paragraph}\indent Since all, except for one, of the distinct prime factors of an OPN $N = \displaystyle\prod_{i=1}^{\omega(N)} {{P_i}^{\alpha_i}}$ have even exponents, then $\alpha_i \geq 2$ for all but one $i$, say $i = j$, for which $\alpha_j \geq 1$. (Note that $\alpha_j$ is the exponent of the special or Euler prime.)  Thereupon, we have the following improvements to the results in Remark \ref{remark3}:
\begin{center}
$\displaystyle\prod_{i=1}^{\omega(N)} {\left(\alpha_i + 1\right)} \geq 3^{\omega(N) - 1}\cdot2$ \\
$N > \displaystyle{\left(\frac{\displaystyle\prod_{i=1}^{\omega(N)} {\left(\alpha_i + 1\right)}}{2}\right)}^2 \geq 3^{2\omega(N) - 2}$ \\
$\omega(N) \geq 9$ [\emph{Nielsen}] $\Rightarrow N > 3^{16} = 43046721$
\end{center}
Note that there is approximately a $655.84\%$ improvement in the magnitude of the bound thus obtained for the smallest possible OPN as compared to the previous crude lower bound of $2^{16}$.  The novelty of the approach of Lemma \ref{lemma10} can be realized if we consider the fact that we did NOT need to check any of the odd numbers below $43046721$ to see if they could be perfect. 
\end{paragraph}
\begin{remrk}\label{remark5} We casually remark that the lower bound of $3^{16}$ obtained for an OPN here improves on the classical bound of $2\cdot{10}^6$ obtained by Turcaninov in $1908$. However, at that time, the best-known bound for the number of distinct prime factors of an OPN $N$ was $\omega(N) \geq 5$, which was shown to be true by Sylvester in $1888$, whereas we used the bound $\omega(N) \geq 9$ by Nielsen ($2006$) here.
\end{remrk}
\begin{paragraph}\indent As mentioned in Remark \ref{remark5}, A. Turcaninov showed in 1908 that no odd number less than two million can be perfect.  The figure $2\cdot{10}^6$ is generally accepted as the minimum in standard texts.  Nonetheless, it is easy to show by means of well-known proofs that the smallest possible OPN must be greater than ten billion (i.e. ${10}^{10}$).
\end{paragraph}
\begin{paragraph}\indent Let the prime factorization of an OPN $N_0$ be given by
\begin{center} $N_0 = {{P_1}^{a_1}}{{Q_1}^{b_1}}{{Q_2}^{b_2}}\cdots{{Q_m}^{b_m}}$
\end{center}
where $a_1$ is odd and $b_1, b_2, \ldots, b_m$ are even. The following conditions must hold: 
\begin{itemize}
\item{Euler had shown that $P_1 \equiv a_1 \equiv 1 \pmod{4}$.}
\item{Sylvester demonstrated that it must be the case that $m \geq 4$.}
\item{Steuerwald proved that $b_1 = b_2 = \cdots = b_m = 2$ is not possible.}
\item{Brauer extended the last result to $b_i \neq 4$, when $b_1 = b_2 = \cdots = b_{i-1} = b_{i+1} = \cdots = b_m = 2$.}
\item{Sylvester also showed that $105 = 3\cdot5\cdot7$ does not divide $N_0$. (We give a proof of this result in Section $3.4$.)}
\end{itemize}
\end{paragraph}
\begin{paragraph}\indent The only numbers less than ten billion which satisfy all these itemized conditions are ${3^6}\cdot{5^2}\cdot{{11}^2}\cdot{{13}^2}\cdot{17}$ and ${3^6}\cdot{5^2}\cdot{{11}^2}\cdot{13}\cdot{{17}^2}$.  One can verify that each of these two is abundant by directly computing the sum of its divisors.
\end{paragraph}  
\begin{paragraph}\indent We end this section with a copy of an email correspondence between the author and Richard Brent, one of the three co-authors of the $1991$ paper which showed that an OPN $N$ must be bigger than ${10}^{300}$.  The significance of the email lies with the fact that the author of this thesis was able to show that $N > p^k \sigma(p^k)$ where $k$ is unrestricted (i.e. $k$ can be even or odd).  In fact, the author was able to show the slightly stronger statement $N \geq \frac{3}{2} p^k \sigma(p^k)$ with $k$ unrestricted.  (Although Dr. Brent did not mention it in his reply, this result could give a higher lower bound for OPNs.)  We shall give a proof of this last result in Theorem $4.2.5$.  For now, let us take a look at the author's email inquiry and how Dr. Brent responded to it:
\end{paragraph} \\
\newpage
\begin{tabular}{ll}
From: & Richard Brent [rpbrent@gmail.com] \\
Sent: & Tuesday, December $04$, $2007$ 9:38 AM \\
To:   & Jose Dris \\
Subject: & Re: Inquiry regarding the lower bound of ${10}^{300}$ that you obtained for \\ 
~ & odd perfect numbers \\
\end{tabular} \\
~ \\
Dear Arnie, \\
~ \\
~ \\
On 22/11/2007, \emph{Jose Dris} $\left\langle{Jose.Dris@safeway.com}\right\rangle$ wrote: \\
\begin{paragraph}\indent \emph{Hi Dr. Brent,}
\end{paragraph} 
\begin{paragraph}\indent I am Arnie Dris, a candidate for the degree of MS in Mathematics at DLSU here at Manila, Philippines. I am currently in the process of writing up a thesis on odd perfect numbers, and I came across your lower bound of ${10}^{300}$ for odd perfect numbers that you obtained together with two co-authors.
\end{paragraph} 
\begin{paragraph}\indent In your paper, you stated that you used the simple observation that \\
$N > p^k \sigma(p^k)$ where $p^k || N$ and $k$ is even.  I would just like to ask whether this observation was used in ALL cases that you have considered, thereby proving that $N > {10}^{300}$ in each case?
\end{paragraph} \\
~ \\ 
It was a long time ago, but as I recall we used that observation in most cases.  There were a few ``hard cases" where we could not compute the sigma function $\sigma(p^k)$ because we did not know the complete factorisation that is needed to do this, e.g. $\sigma({3221}^{42})$ was a $148$-digit composite number that we could not factorise at the time (it may have been factored since then).  In such cases we had to use a more complicated method.  See the ``Proof of Theorem 1" in the paper. \\
~ \\
You can get the paper online at \\ 
http://wwwmaths.anu.edu.au/$\sim$brent/pub/pub116.html \\
and there's also a link there to the computer-generated ``proof tree". \\
~ \\
William Lipp has a project to extend the bound ${10}^{300}$ by much the same method but with more factorisations (since computers are faster now and algorithms such as the number field sieve are available).  He hopes to go at least to ${10}^{400}$ and possibly further. \\
~ \\
Regards, \\
Richard Brent \\
-- \\
Prof R. P. Brent, ARC Federation Fellow \\
MSI, ANU, Canberra, ACT 0200, Australia \\
rpbrent@gmail.com \\
http://wwwmaths.anu.edu.anu/$\sim$brent/ \\
phone: +61-4-18104021 \\
\section{Congruence Conditions for an OPN $N$}
In this section, we revisit a theorem of Jacques Touchard dating back from $1953$.  Touchard proved that any odd perfect number must have the form $12m + 1$ or $36m +9$.  His proof relied on the fact that the numbers $\sigma(k)$ satisfy
\begin{center} ${\displaystyle\frac{{n^2}(n - 1)}{12}}\hspace{0.03truein}{\sigma(n)} = {\displaystyle\sum_{k=1}^{n-1}{\left[5k(n - k) - n^2\right]\sigma(k)\sigma(n - k)}}$,
\end{center}
a recursion relation derived by Balth. van der Pol in 1951 using a nonlinear partial
differential equation.  We give here Judy Holdener's proof of the same result in $2002$, which is much shorter and more elementary than Touchard's proof.  The proof was inspired by Francis Coghlan's solution to Problem \#$10711$ published in the American Mathematical Monthly in $2001$ regarding the nonexistence of two consecutive perfect numbers.  
\begin{paragraph}\indent First, we briefly spell out a lemma on a congruence condition for an OPN $N$:
\end{paragraph}
\begin{lemm}\label{lemma11} If $N \equiv 5 \pmod{6}$, then $N$ is not perfect.
\end{lemm}
\begin{proof} Assume that $N \equiv 5 \pmod{6}$. Then $N$ is of the form $6k + 5 = 3(2k + 1) + 2$, so $N \equiv 2 \pmod{3}$.  Since all squares are congruent to $1$ modulo $3$, $N$ is not a square.  Further, note that for any divisor $d$ of $N$, $N = d\cdot(\frac{N}{d}) \equiv 2 \equiv -1 \pmod{3}$ implies that either $d \equiv -1 \pmod{3}$ and $\frac{N}{d} \equiv 1 \pmod{3}$, or $d \equiv 1 \pmod{3}$ and $\frac{N}{d} \equiv -1 \pmod{3}$.  Either way, $d + (\frac{N}{d}) \equiv 0 \pmod{3}$, and
\begin{center} $\sigma(N) = \displaystyle\sum_{d \mid N, \hspace{0.03truein} d < \sqrt{N}}{\displaystyle\left(d + \frac{N}{d}\right)} \equiv 0 \pmod{3}$.
\end{center}
Therefore, $\sigma(N) \equiv 0 \pmod{3}$ while $2N \equiv 4 \equiv 1 \pmod{3}$.  These computations show that $N = 6k + 5$ cannot be perfect.  
\end{proof}
\begin{paragraph}\indent Using a similar argument, we can also show Euler's result that any OPN is congruent to $1$ modulo $4$.  For suppose otherwise that $N \equiv 3 \pmod{4}$.  Then, again, $N$ is not a square and
\begin{center} $\sigma(N) = \displaystyle\sum_{d \mid N, \hspace{0.03truein} d < \sqrt{N}}{\displaystyle\left(d + \frac{N}{d}\right)} \equiv 0 \pmod{4}$.
\end{center}
Hence, $\sigma(N) \equiv 0 \pmod{4}$, while $2N \equiv 6 \equiv 2 \pmod{4}$.   
\end{paragraph}
\begin{paragraph}\indent Lemma \ref{lemma11} generalizes immediately to the following Corollary:
\end{paragraph}
\begin{cor}\label{corollary5} If $M$ is a number satisfying $M \equiv 2 \pmod{3}$, then $M$ is not perfect.
\end{cor}
\begin{proof} Note that Lemma \ref{lemma11} takes care of the case when $M$ is odd.  We now show that the statement is true for even $M$. 
\begin{paragraph}\indent To this end, suppose $M \equiv 2 \pmod{3}$ is even.  We show here that $M$ cannot be perfect.  Suppose to the contrary that $M$ is even perfect.  Then by the Euclid-Euler Theorem, $M = {2^{p - 1}}(2^p - 1)$ for some prime number $p$.  If $p = 2$, then $M = 6$ which is divisible by 3.  Assume $p \geq 3$.  Then $p$ is an odd prime.  Thus, $2^p \equiv (-1)^p \equiv (-1)^{p - 1}(-1) \equiv 1\cdot(-1) \equiv -1 \equiv 2 \pmod{3}$.  Therefore, $2^p - 1 \equiv 1 \pmod{3}$.  Also, $2^{p - 1} \equiv (-1)^{p - 1} \equiv 1 \pmod{3}$ since $p$ is assumed to be an odd prime.  
\end{paragraph}
\begin{paragraph}\indent Thus, if $M$ is to be even perfect, either $M \equiv 0 \pmod{3}$ (which occurs only when $p = 2$ and $M = 6$) or $M \equiv 1 \pmod{3}$ (when $p$ is a prime $\geq 3$).  Consequently, if $M \equiv 2 \pmod{3}$ is even, it cannot be perfect.  Lemma \ref{lemma11} says that this is also true when $M$ is odd.  Hence, we have the general result: If a number $M$ satisfies $M \equiv 2 \pmod{3}$, then $M$ cannot be perfect.
\end{paragraph}
\end{proof}
\begin{paragraph}\indent We we will use the following formulation of the Chinese Remainder Theorem to prove the next major result:
\end{paragraph}
\begin{thm}\label{CRT} Chinese Remainder Theorem \\
Suppose $n_1, n_2, \ldots, n_k$ are integers which are pairwise relatively prime (or coprime).  Then, for any given integers $a_1, a_2, \ldots, a_k$, there exists an integer $x$ solving the system of simultaneous congruences
\begin{center} 
$x \equiv a_1 \pmod{n_1}$ \\
$x \equiv a_2 \pmod{n_2}$ \\
$\vdots$ \\
$x \equiv a_k \pmod{n_k}$
\end{center}
Furthermore, all solutions $x$ to this system are congruent modulo the product $n = {n_1}{n_2}\cdots{n_k}$.  Hence $x \equiv y \pmod{n_i}$ for all $1 \leq i \leq k$, if and only if $x \equiv y \pmod{n}$.  
\begin{paragraph}\indent Sometimes, the simultaneous congruences can be solved even if the $n_i$'s are not pairwise coprime.  A solution $x$ exists if and only if $a_i \equiv a_j \pmod{\gcd(n_i, n_j)}$ for all $i$ and $j$.  All solutions $x$ are then congruent modulo the \emph{least common multiple} of the $n_i$.
\end{paragraph}
\end{thm}
\begin{paragraph}\indent We can now use Lemma \ref{lemma11} to prove Touchard's theorem.
\end{paragraph} 
\begin{thm}\label{theorem13} (Touchard) An OPN must have the form $12m + 1$ or $36m + 9$.
\end{thm}
\begin{proof} Let $N$ be an OPN and apply Lemma \ref{lemma11}.  Any number of the form $6k + 5$ cannot be perfect, so $N$ must be of the form $6k + 1$ or $6k + 3$.  But from Euler's result, we know that $N$ is of the form $4j + 1$.  Hence either $N = 6k + 1$ and $N = 4j + 1$, or $N = 6k + 3$ and $N = 4j + 1$.  We now attempt to solve these two sets of simultaneous equations for $N$, thereby deriving congruence conditions for $N$:
\begin{paragraph}\indent Case 1: $N = 6k + 1$ and $N = 4j + 1$.  This means that $N - 1 = 6k = 4j = LCM(4, 6)m = 12m$ (by the Chinese Remainder Theorem) where $k = 2m$ and $j = 3m$, which implies that $N = 12m + 1$.
\end{paragraph}
\begin{paragraph}\indent Case 2: $N = 6k + 3$ and $N = 4j + 1$.  This means that $N + 3 = 6k + 6 = 4j + 4 = 6(k + 1) = 4(j + 1) = LCM(4, 6)m = 12p$ (by the Chinese Remainder Theorem) where $k + 1 = 2p$ and $j + 1 = 3p$, which implies that $N = 12p - 3$.  On setting $p = m_0 + 1$, we get $N = 12m_0 + 9$.
\end{paragraph}
\begin{paragraph}\indent Finally, note that in Case 2, if $N = 12m_0 + 9$ and $3 \nmid m_0$, then $\sigma(N) = \sigma(3(4m_0 + 3)) = \sigma(3)\sigma(4m_0 + 3) = 4\sigma(4m_0 + 3)$.  With this, we have $\sigma(N) \equiv 0 \pmod{4}$, while $2N = 2(12m_0 + 9) = 24m_0 + 18 = 4(6m_0 + 4) + 2 \equiv 2 \pmod{4}$.  Therefore, $N$ cannot be perfect if $3 \nmid m_0$ in Case 2, and we conclude that $3 \mid m_0$ in this case, and on setting $m_0 = 3m$, we get $N = 12m_0 + 9 = 12(3m) + 9 = 36m + 9$.
\end{paragraph}
\end{proof}
\begin{remrk}\label{remark6} We emphasize that Touchard's theorem is really simple.  Holdener's proof as presented here is indeed elementary; it does not make use of the concept of unique factorization nor of sigma multiplicativity (other than in showing that $3$ divides $m$ when $12m + 9$ is perfect).  Touchard's result emerges after summing divisors in pairs, and this can always be done because perfect numbers are never squares.
\end{remrk}
\begin{paragraph}\indent In January of 2008, Tim Roberts made a post at
\begin{center} http://www.unsolvedproblems.org/UP/Solutions.htm 
\end{center}
where he outlined an improvement to Theorem \ref{theorem13}.
\end{paragraph}
\begin{thm}\label{theorem14} (Roberts) Let $N$ be an OPN.  Then either one of the following three congruences must hold:
\begin{itemize}
\item{$N \equiv 1 \pmod{12}$.}
\item{$N \equiv 117 \pmod{468}$.}
\item{$N \equiv 81 \pmod{324}$.}
\end{itemize}
\end{thm}
\begin{proof} Let $N$ be an OPN.  We note that, if $3 \mid N$, then $3^k \mid N$, where $k$ is even (Euler).  If $k = 0$, then by Theorem \ref{theorem13}, $N \equiv 1 \pmod{12}$.  Also, by the factor chain approach, if $N$ is an OPN and a factor of $N$ is $3^k$, then $N$ is also divisible by $\sigma(3^k) = 1 + 3 + 3^2 + \ldots + 3^k$. If $k = 2$, then again by Theorem \ref{theorem13}, $N \equiv 9 \pmod{36}$. Further, since $N$ is an OPN, we know that $\sigma(3^2) = 1 + 3 + 3^2 = 13$ divides $N$.  Hence, $N \equiv 0 \pmod{13}$.  From the Chinese Remainder Theorem, we can deduce that $N \equiv 117 \pmod{468}$.  If $k > 2$, then $N$ is divisible by $3^4 = 81$.  Thus, (again by Theorem \ref{theorem13}) $N$ must satisfy both $N \equiv 9 \pmod{36}$ and $N \equiv 0 \pmod{81}$.  Again, from the Chinese Remainder Theorem, we can deduce that $N \equiv 81 \pmod{324}$.  Thus, if $N$ is an OPN, then either $N \equiv 1 \pmod{12}$, $N \equiv 117 \pmod{468}$ or $N \equiv 81 \pmod{324}$.   
\end{proof}
\begin{paragraph}\indent It is, of course, similarly possible to further refine the last of these results, by separately considering even values of $k$ bigger than 4.  
\end{paragraph}
\section{Some Interesting Results on Perfect Numbers}
We conclude this chapter with the following (interesting) results on perfect numbers (with emphasis on OPNs):
\begin{itemize}
\item{No two consecutive integers can be both perfect.}
\item{An odd perfect number cannot be divisible by $105$.}
\item{An odd perfect number must be a sum of two squares.}
\end{itemize}
\subsection{Nonexistence of Consecutive Perfect Numbers}
\begin{paragraph}\indent From Corollary \ref{corollary5}, we see that a number $M$ (odd or even) satisfying $M \equiv 2 \pmod{3}$ cannot be perfect.  We make use of this observation to prove that no two consecutive integers can be both perfect.  (This was shown to be true by James Riggs and Judy Holdener through a joint undergraduate research project in $1998$.) 
\end{paragraph}
\begin{paragraph}\indent  First, suppose $N$ is an OPN.  By Euler's characterization of OPNs, $N \equiv 1 \pmod{4}$.  We claim that $N + 1$ cannot be an even perfect number.  Observe that $N + 1 \equiv 2 \pmod{4}$ means that $2 \mid (N + 1)$ but $4 \nmid (N + 1)$.  The only even perfect number of the form $N + 1 = {2^{p - 1}}(2^p - 1)$ satisfying these two conditions is the one for $p = 2$, i.e. $N + 1 = 6$.  But this means that, by assumption, $N = 5$ must be an OPN, contradicting the fact that $N = 5$ is deficient.
\end{paragraph}
\begin{paragraph}\indent Next, we also claim that $N - 1$ cannot be (even) perfect, if $N$ is an OPN, where $N - 1 \equiv 0 \pmod{4}$.  From the discussion of the proof of Corollary \ref{corollary5}, since $N - 1 \equiv 0 \pmod{4}$ it follows that $N - 1 = 2^{p - 1}(2^p - 1)$ for some primes $p$ and $2^p - 1$ with $p \geq 3$.  Thus, $N - 1 \equiv 1 \pmod{3}$, which implies that $N \equiv 2 \pmod{3}$.  But our original assumption was that $N$ is an OPN, contradicting the criterion in Corollary \ref{corollary5}.  Consequently, this means that $N - 1$ is not perfect in this case.
\end{paragraph}
\begin{paragraph}\indent We have shown in the preceding paragraphs that, if $N$ is an OPN of the form $4m + 1$, then it cannot be true that $N - 1$ or $N + 1$ are also (even) perfect.  To fully prove the assertion in the title of this section, we need to show that $N - 1$ and $N + 1$ cannot be OPNs if $N$ is an even perfect number.
\end{paragraph}
\begin{paragraph}\indent To this end, suppose $N$ is even perfect, that is, $N = 2^{p - 1}(2^p - 1)$ for some primes $p$ and $2^p - 1$.  If $p = 2$, then $N = 6$, and clearly, $N - 1 = 5$ and $N + 1 = 7$ are not perfect since they are both primes (and are therefore deficient).
\end{paragraph}
\begin{paragraph}\indent Now let $p$ be a prime which is at least $3$.  Then $N \equiv 0 \pmod{4}$, whence it follows that $N - 1 \equiv 3 \pmod{4}$ and $N - 1$ cannot be an OPN by Euler's characterization.  Also, from the proof of Corollary \ref{corollary5}, note that if $N$ is an even perfect number with $p \geq 3$, then $N \equiv 1 \pmod{3}$, which implies that $N + 1 \equiv 2 \pmod{3}$, which further means that $N + 1$ cannot be perfect by the criterion in Corollary \ref{corollary5}.
\end{paragraph}
\subsection{OPNs are Not Divisible by $105$}
\begin{paragraph}\indent Mathematicians have been unable, so far, to eliminate the possibility that an odd perfect number may be divisible by $3$. However, by use of Lemma \ref{lemma8}, we can show that an odd number divisible by $3$, $5$ and $7$ cannot be perfect. (This was proved by Sylvester in $1888$.)
\end{paragraph}
\begin{paragraph}\indent To this end, suppose that $N$ is an OPN that is divisible by $3$, $5$, and $7$.  Then $N$ takes the form $N = {3^a}{5^b}{7^c}\hspace{0.03truein}n$.  Suppose that $5$ is the special/Euler prime of $N$, so that $b \geq 1$. By Euler's characterization of OPNs, $a$ and $c$ must be even, so we may take $a \geq 2$ and $c \geq 2$.  Without loss of generality, since $\omega(N) \geq 9$ (see Remark \ref{remark2}) we may safely assume that $n > {11}^2$.  Then $I(n) > 1$, and consequently, by use of Lemma \ref{lemma8}, we have:
\begin{center} {$2 \geq \displaystyle\left(1 + \displaystyle\frac{1}{3} + \displaystyle\frac{1}{3^2}\right)\displaystyle\left(1 + \displaystyle\frac{1}{5}\right)\displaystyle\left(1 + \displaystyle\frac{1}{7} + \displaystyle\frac{1}{7^2}\right)\hspace{0.03truein}I(n)$} \\ 
{$ > \displaystyle\frac{13}{9}\frac{6}{5}\frac{57}{49}\cdot1 = \displaystyle\frac{4446}{2205} > 2.0163 > 2$}
\end{center}
resulting in the contradiction $2 > 2$.  
\end{paragraph}
\begin{paragraph}\indent If, in turn, we assume that $5$ is not the special/Euler prime of $N$ (so that $b \geq 2$), then without loss of generality, since $n$ must contain the special/Euler prime and $\omega(N) \geq 9$, we can assume that $n > 13$.  In this case, it is still true that $I(n) > 1$ (also that $a \geq 2$ and $c \geq 2$).  Hence, by use of Lemma \ref{lemma8}, we obtain:
\begin{center} {$2 \geq \displaystyle\left(1 + \displaystyle\frac{1}{3} + \displaystyle\frac{1}{3^2}\right)\displaystyle\left(1 + \displaystyle\frac{1}{5} + \displaystyle\frac{1}{5^2}\right)\displaystyle\left(1 + \displaystyle\frac{1}{7} + \displaystyle\frac{1}{7^2}\right)\hspace{0.03truein}I(n)$} \\ 
{$ > \displaystyle\frac{13}{9}\frac{31}{25}\frac{57}{49}\cdot1 = \displaystyle\frac{22971}{11025} > 2.0835 > 2$}
\end{center}
resulting, again, in the contradiction $2 > 2$.  
\end{paragraph}
\begin{paragraph}\indent We are therefore led to conclude that an OPN cannot be divisible by $3\cdot5\cdot7 = 105$.
\end{paragraph}
\subsection{OPNs as Sums of Two Squares}
\begin{paragraph}\indent We borrow heavily the following preliminary material from \\ http://en.wikipedia.org/wiki/Proofs\_of\_Fermat's\_theorem\_on\_sums\_of\_two\_squares. \\ 
This is in view of the fact that the special/Euler prime $p$ of an OPN $N = {p^k}{m^2}$ satisfies $p \equiv 1 \pmod{4}$.
\end{paragraph}
\begin{paragraph}\indent Fermat's theorem on sums of two squares states that an odd prime $p$ can be expressed as $p = x^2 + y^2$ with $x$ and $y$ integers \emph{if and only if} $p \equiv 1 \pmod{4}$. It was originally announced by Fermat in $1640$, but he gave no proof. The \emph{only if} clause is trivial: the squares modulo $4$ are $0$ and $1$, so $x^2 + y^2$ is congruent to $0$, $1$, or $2$ modulo $4$. Since $p$ is assumed to be odd, this means that it must be congruent to $1$ modulo $4$.
\end{paragraph}
\begin{paragraph}\indent Euler succeeded in proving Fermat's theorem on sums of two squares in $1747$, when he was forty years old. He communicated this in a letter to Goldbach dated $6$ May $1747$. The proof relies on infinite descent, and proceeds in five steps; we state the first step from that proof below as we will be using it in the next paragraph:
\begin{itemize}
\item{The product of two numbers, each of which is a sum of two squares, is itself a sum of two squares.}
\end{itemize}
\end{paragraph}
\begin{paragraph}\indent Given an OPN $N = {p^k}{m^2}$, since $p \equiv 1 \pmod{4}$, by Fermat's theorem we can write $p$ as a sum of two squares.  By Euler's first step above, $p^k$ can likewise be expressed as a sum of two squares, $p^k$ being the product of $k \hspace{0.05truein} p$'s.  Hence, we can write $p^k = r^2 + s^2$ for some positive integers $r$ and $s$.  Multiplying both sides of the last equation by $m^2$, we get $N = {p^k}{m^2} = {m^2}(r^2 + s^2) = (mr)^2 + (ms)^2$. Hence, an odd perfect number may be expressed as a sum of two squares.
\end{paragraph}
\begin{remrk}\label{remark7} Let $\theta(n)$ be the number of integers $k \leq n$ that can be expressed as $k = a^2 + b^2$, where $a$ and $b$ are integers. Does the limit $\displaystyle\lim_{n\to\infty} {\frac{\theta(n)}{n}}$ exist and what is its value?  Numerical computations suggest that it exists.  $\theta(n)$ is approximately $\displaystyle\left(\frac{3}{4}\right)\hspace{0.03truein}\cdot\displaystyle\frac{n}{\displaystyle\sqrt{\log(n)}}$, so that the limit exists and equals zero. More precisely, the number of integers less than $n$ that are sums of two squares behaves like  $K\hspace{0.03truein}\cdot\displaystyle\frac{n}{\displaystyle\sqrt{\log(n)}}$ where $K$ is the \emph{Landau-Ramanujan} constant.  Dave Hare from the Maple group has computed $10000$ digits of $K$.  The first digits are $K \approx 0.764223653...$.  See
http://www.mathsoft.com/asolve/constant/lr/lr.html for more information.  Therefore the sums of two squares have density $0$.  Since OPNs are expressible as sums of two squares, then OPNs have density $0$, too.
\end{remrk} 
\begin{paragraph}\indent In the next chapter, we shall take a closer look into the nature of abundancy outlaws (which were first described in Section $2.2$).  We shall also describe a systematic procedure on how to bound the prime factors of an OPN $N$, using the latest current knowledge on $N$ as well as some novel results.  Lastly, we shall discuss some of the original results of the author pertaining to inequalities between the components of an OPN.
\end{paragraph} 
\chapter{OPN Solution Attempts 2: \\ Some New Approaches}
In Chapter 3, we saw how increasing the lower bounds for $\omega(N)$ (the number of distinct prime factors of an OPN $N$) and $N$ itself could potentially prove or disprove the OPN Conjecture.  We also saw how the concept of divisibility may be used to derive congruence conditions for $N$.
\begin{paragraph}\indent Here, we shall take a closer look into the following new approaches for attempting to solve the OPN Problem:
\begin{itemize}
\item{What are abundancy outlaws?  How are they related to abundancy indices?  How could one use the concept of abundancy outlaws to (potentially) disprove the OPN Conjecture?}
\item{How can one bound the prime factors of an OPN $N$?  Is there a systematic procedure on how to do this? (We discuss the author's results on the relationships between the components of $N$ in Subsection $4.2.4$.)}
\item{Can we use the abundancy index concept to ``count" the number of OPNs? (We answer this question in the negative for a particular case.)}
\end{itemize}
\end{paragraph}
\begin{paragraph}\indent The reader is advised to review Section 2.2 of this thesis prior to commencing a study of this chapter.
\end{paragraph}
\section{Abundancy Outlaws and Related Concepts}
Modern treatments of problems involving the abundancy index have been concerned with two fundamental questions: \\
~ \\
\begin{tabular}{ll}
I. & Given a rational number $\displaystyle\frac{a}{b}$, does there exist some positive integer $x$ such that \\ 
~ & $I(x) = \displaystyle\frac{\sigma(x)}{x} = \displaystyle\frac{a}{b}$? \\
II. & When does the equation $I(x) = \displaystyle\frac{a}{b}$ have exactly one solution for $x$? \\
\end{tabular}
\begin{paragraph}\indent We give various answers to these two questions in the three subsections that follow.
\end{paragraph}
\subsection{Friendly and Solitary Numbers}
\begin{paragraph}\indent If $x$ is the unique solution of $I(x) = \displaystyle\frac{a}{b}$ (for a given rational number $\displaystyle\frac{a}{b}$) then
$x$ is called a \emph{solitary number}. On the other hand, if $x$ is one
of at least two solutions of $I(x) = \displaystyle\frac{a}{b}$ (for a given rational
number $\displaystyle\frac{a}{b}$) then $x$ is called a \emph{friendly number}.  We formalize these two concepts in the following definition:
\end{paragraph}
\begin{defn}\label{definition12} Let $x$ and $y$ be distinct positive integers.  If $x$ and $y$ satisfy the equation $I(x) = I(y)$ then $(x, y)$ is called a \emph{friendly pair}.  Each member of the pair is called a \emph{friendly number}.  A number which is not friendly is called a \emph{solitary number}.
\end{defn}
\begin{paragraph}\indent We illustrate these concepts with several examples.
\end{paragraph}
\begin{exmpl}\label{example11} Clearly, if $a$ and $b$ are perfect numbers with $a \neq b$ (i.e. $\sigma(a) = 2a$, $\sigma(b) = 2b$), then $(a, b)$ is a friendly pair.
\end{exmpl} 
\begin{exmpl}\label{example12} We claim that, given a positive integer $n$ satisfying $\gcd(n, 42) = 1$, $(6n, 28n)$ is a friendly pair.  To prove this, note that $\gcd(n, 42) = 1$ means that $\gcd(n, 2) = \gcd(n, 3) = \gcd(n, 7) = 1$.  Let us now compute $I(6n)$ and $I(28n)$ separately.  Since $\gcd(n, 2) = \gcd(n, 3) = 1$, then $\gcd(n, 6) = 1$ and $I(6n) = I(6)I(n) = 2I(n)$ since $6$ is a perfect number.  Similarly, since $\gcd(n, 2) = \gcd(n, 7) = 1$, then $\gcd(n, 28) = 1$ and $I(28n) = I(28)I(n) = 2I(n)$ since $28$ is a perfect number.  These computations show that $I(6n) = I(28n) = 2I(n)$ whenever $\gcd(n, 42) = 1$, and therefore $(6n, 28n)$ is a friendly pair for such $n$.
\end{exmpl}
\begin{remrk}\label{remark8} Since there exist infinitely many positive integers $n$ satisfying \\ $\gcd(n, 42) = 1$, Example \ref{example12} shows that there exist infinitely many friendly numbers.
\end{remrk}
\begin{exmpl}\label{example13} M. G. Greening showed in $1977$ that numbers $n$ such that \\ $\gcd(n, \sigma(n)) = 1$, are solitary.  For example, the numbers $1$ through $5$ are all solitary by virtue of Greening's criterion.  There are $53$ numbers less than $100$, which are known to be solitary, but there are some numbers, such as $10$, $14$, $15$, and $20$ for which we cannot decide ``solitude".  (This is because it is, in general, difficult to determine whether a particular number is solitary, since the only tool that we have so far to make such determination, namely Greening's result, is sufficient but not necessary.  In the other direction, if any numbers up to $372$ (other than those listed in the \emph{Online Encyclopedia of Integer Sequences}) are friendly, then the smallest corresponding values of the friendly pairs are $> {10}^{30}$ \cite{H2}.)  Also, we remark that there exist numbers such as $n = 18, 45, 48$ and $52$ which are solitary but for which $\gcd(n, \sigma(n)) \neq 1$.
\end{exmpl}
\begin{paragraph}\indent We give here a proof of Greening's criterion.  Suppose that a number $n$ with $\gcd(n, \sigma(n)) = 1$ is not solitary.  Then $n$ is friendly, i.e. there exists some number $x \neq n$ such that $I(n) = I(x)$ .  This is equivalent to $x\sigma(n) = n\sigma(x)$, which implies that, since $\gcd(n, \sigma(n)) = 1$, $n \mid x$ or $x$ is a multiple of $n$.  Thus any friend of $n$ must be a (nontrivial) multiple of it (since $n \neq x$).  Hence we can write $x = mn$ where $m \geq 2$.  Write $m = jk$ with $\gcd(j, n) = 1$ and $\gcd(k, n) > 1$. Then by virtue of Lemma \ref{lemma2}, $I(x) > I(kn)$ since $kn$ is a factor of $x$ (unless $j = 1$), which implies that $I(x) > I(n)$ (since $k > 1$ and this follows from $\gcd(k, n) > 1$).  This is a contradiction.  If $j = 1$, then we have $x = kn$ with $\gcd(k, n) > 1$.  Again, by virtue of Lemma \ref{lemma2} and similar considerations as before, $I(x) > I(n)$ (unless $k = 1$, but this cannot happen since $jk = m \geq 2$) which is again a contradiction.  Thus, numbers $n$ with $\gcd(n, \sigma(n)) = 1$ are solitary.     
\end{paragraph}
\begin{exmpl}\label{example14} We claim that primes and powers of primes are solitary. It suffices to show that $p^k$ and $\sigma(p^k)$ are relatively prime.  To this end, consider the equation $(p - 1)\sigma(p^k) = p^{k+1} - 1$.  Since this can be rewritten as $(1 - p)\sigma(p^k) + p\cdot{p^k} = 1$, then we have $\gcd(p^k, \sigma(p^k)) = 1$.  By Greening's criterion, $p^k$ is solitary.  Thus, primes and powers of primes are solitary.
\end{exmpl}
\begin{remrk}\label{remark9} Since there are infinitely many primes (first proved by Euclid), and therefore infinitely many prime powers, Example \ref{example14} shows that there are infinitely many solitary numbers.
\end{remrk}
\begin{remrk}\label{remark10} While not much is known about the nature of solitary numbers, we do know that the density of friendly numbers is positive, first shown by $Erd\ddot{o}s$ \cite{E}: The number of solutions of $I(a) = I(b)$ satisfying $a < b \leq x$ equals $Cx + o(x)$, where $C > 0$ is a constant (in fact, $C \geq \frac{8}{147}$ \cite{A2}). In $1996$, Carl Pomerance told Dean Hickerson that he could prove that the solitary numbers have positive density, disproving a conjecture by Anderson and Hickerson in $1977$. However, this proof seems not to ever have been published.
\end{remrk}
\subsection{Abundancy Indices and Outlaws}
\begin{paragraph}\indent On the other hand, rational numbers $\displaystyle\frac{a}{b}$ for which $I(x) = \displaystyle\frac{a}{b}$ has no solution for $x$ are called \emph{abundancy outlaws}. (Recall Definition \ref{definition11}.)  Of course, those rationals $\displaystyle\frac{a}{b}$ for which $I(x) = \displaystyle\frac{a}{b}$ has at least one solution for $x$ are called \emph{abundancy indices}.
\end{paragraph}
\begin{paragraph}\indent It is best to illustrate with some examples.
\end{paragraph}
\begin{exmpl}\label{example15} At once, Lemma \ref{lemma5} reveals a class of abundancy outlaws.  Since that lemma says that $\displaystyle\frac{m}{n}$ is an outlaw when $1 < \displaystyle\frac{m}{n} < \displaystyle\frac{\sigma(n)}{n}$ (with $\gcd(m, n) = 1$), then we have the class $\displaystyle\frac{\sigma(N) - t}{N}$ of outlaws (with $t \geq 1$). (We shall show later that, under certain conditions, $\displaystyle\frac{\sigma(N) + t}{N}$ is also an abundancy outlaw.) 
\end{exmpl}
\begin{exmpl}\label{example16} Let $a, b, c$ be positive integers, and let $p$ be a prime such that $\gcd(a, p) = 1$, $b = p^c$ (so that $b$ is a prime power, and $\gcd(a, b) = 1$), and $a \geq \sigma(b)$.  Suppose we want to find a positive integer, $n$, such that $\displaystyle\frac{\sigma(n)}{n} = \displaystyle\frac{a}{b}$.  (That is, we want to determine if $\displaystyle\frac{a}{b}$ is an abundancy index or not.)  This problem is equivalent to the problem of finding positive integers $m, k$ such that:
\begin{itemize}
\item{$n = mp^k, k \geq c,$ and $\gcd(m, p) = 1$ (or equivalently, $\gcd(m, b) = 1$)}
\item{$\displaystyle\frac{\sigma(m)}{m} = \displaystyle\frac{a{p}^{k - c}}{\sigma(p^k)}$}.
\end{itemize}
We will formally state this result as a lemma later (where we will then present a proof), but for now let us see how we may apply this result towards showing that the fraction $\displaystyle\frac{7}{2}$ is an abundancy index.  (Indeed, we are then able to construct an explicit $n$ satisfying $\displaystyle\frac{\sigma(n)}{n} = \displaystyle\frac{a}{b}$ for a given $\displaystyle\frac{a}{b}$.)
\begin{paragraph}\indent We now attempt to find a positive integer $n$ such that $\displaystyle\frac{\sigma(n)}{n} = \displaystyle\frac{7}{2}$.  This problem, is equivalent to the problem of finding positive integers $m$, $k$ such that: \\
~ \\
\begin{tabular}{ll}
$1.1$ & $n = {2^k}m$, where $\gcd(2, m) = 1$ (that is, $k$ is the largest power of $2$ to divide $n$) \\
$1.2$ & $\displaystyle\frac{\sigma(m)}{m} = \displaystyle\frac{7\cdot{2^{k - 1}}}{\sigma(2^k)}$. \\
\end{tabular} 
\end{paragraph}
\begin{paragraph}\indent We will now check different values of $k$, attempting each time to find $m$ satisfying these conditions subject to the choice of $k$.  For each $k$, we will proceed until one of the following happens:
\begin{itemize}
\item{We find $m$ satisfying ($1.1$) and ($1.2$). In this case, $n = {2^k}m$ is a solution to our problem.}
\item{We prove that there is no $m$ satisfying ($1.1$) and ($1.2$).  In this case, there is no solution to our problem of the form $n = {2^k}m$, where $\gcd(2, m) = 1$.}
\item{The problem becomes impractical to pursue.  Often a given value of $k$ will leave us with a problem which either cannot be solved with this method, or which is too complicated to be solved in a reasonable amount of time.}
\end{itemize}
\end{paragraph}
\begin{paragraph}\indent We start with $k = 1$.  Then our conditions are: \\
~ \\
\begin{tabular}{ll}
$2.1$ & $n = 2m$, and $\gcd(m, 2) = 1$ \\
$2.2$ & $\displaystyle\frac{\sigma(m)}{m} = \displaystyle\frac{7\cdot{2^{1 - 1}}}{\sigma(2^1)} = \displaystyle\frac{7}{3}$. \\
\end{tabular}
\end{paragraph}
\begin{paragraph}\indent Thus, our problem is to find $m$ satisfying ($2.1$) and ($2.2$).  Let us carry the process one step further for the case $k = 1$.  To do this, we will treat $m$ in the same manner in which we initially treated $n$.  Our goal is to find $m_1, k_1$ such that:
\\
~ \\
\begin{tabular}{ll}
$3.1$ & $m = {3^{k_1}}{m_1}$, and $\gcd(m_1, 3) = 1$.  (Note that since $m_1 \mid m$, and $\gcd(m, 2) = 1$, \\ ~ & we actually need $\gcd(m_1, 6) = 1$.) \\
$3.2$ & $\displaystyle\frac{\sigma(m_1)}{m_1} = \displaystyle\frac{7\cdot{3^{k_1 - 1}}}{\sigma(3^{k_1})}$. \\
\end{tabular}  
\end{paragraph}
\begin{paragraph}\indent Let us now check the case of $k = 1$, $k_1 = 1$.  Our goal is to find $m_1$ such that: \\
~ \\
\begin{tabular}{ll}
$4.1$ & $m = 3{m_1}$, and $\gcd(m_1, 6) = 1$ \\
$4.2$ & $\displaystyle\frac{\sigma(m_1)}{m_1} = \displaystyle\frac{7}{4}$. \\
\end{tabular}
\end{paragraph}
\begin{paragraph}\indent Thus, we want to find some positive integer $m_1$ such that $\gcd(m_1, 6) = 1$ and $\displaystyle\frac{\sigma(m_1)}{m_1} = \displaystyle\frac{7}{4}$.  However, if $\displaystyle\frac{\sigma(m_1)}{m_1} = \displaystyle\frac{7}{4}$, then $4 \mid m_1$ (since $\gcd(4, 7) = 1$).  If $4 \mid m_1$, then $\gcd(m_1, 6) \neq 1$, a contradiction.  Therefore, there is no such $m_1$.
\end{paragraph}
\begin{paragraph}\indent Our method has shown that, in the case of $k = 1$, $k_1 = 1$, there is no positive integer $n$ which solves our original problem.  In particular, $n$ is not of the form $n = {2^1}m = {2^1}({3^1}(m_1)) = 6{m_1}$, where $\gcd(m_1, 6) = 1$.  (Another way to say this is that $6$ is not a \emph{unitary divisor} of any solution to our problem.) 
\end{paragraph}
\begin{paragraph}\indent $($Note that this does not prove the nonexistence of a solution to our problem; it only disproves the existence of a solution of the form given in the last paragraph.  In order to disprove the existence of a solution of any given problem, we have to show that no solution exists for \emph{any} value of $k$.  Here we have not even eliminated the case of $k = 1$, but only the special case where $k_1 = 1$.$)$
\end{paragraph}
\begin{paragraph}\indent Let us move on to $k = 2$.  Our goal is to find $m$ such that: \\
~ \\
\begin{tabular}{ll}
$5.1$ & $n = {2^2}m = 4m$, with $\gcd(m, 2) = 1$ \\
$5.2$ & $\displaystyle\frac{\sigma(m)}{m} = \displaystyle\frac{7\cdot{2^{2 - 1}}}{\sigma(2^2)} = \displaystyle\frac{14}{7} = 2$. \\
\end{tabular}
\end{paragraph}
\begin{paragraph}\indent Here, we must find $m$ such that $m$ is odd and $\displaystyle\frac{\sigma(m)}{m} = 2$;  that is, $m$ must be an odd perfect number.  If $m$ is an odd perfect number, then $n = 4m$ is a solution to our problem.  This is not especially helpful in our search for a solution, so we will move on to another case.
\end{paragraph}
\begin{paragraph}\indent Here we will skip the cases $k = 3$ and $k = 4$, because they are not especially interesting compared to the next case we will deal with.
\end{paragraph}
\begin{paragraph}\indent Consider $k = 5$.  Our goal now is to find $m$ such that: \\
~ \\
\begin{tabular}{ll}
$6.1$ & $n = {2^5}m = 32m$, and $\gcd(m, 2) = 1$ \\
$6.2$ & $\displaystyle\frac{\sigma(m)}{m} = \displaystyle\frac{7\cdot{2^4}}{\sigma(2^5)} = \displaystyle\frac{112}{63} = \displaystyle\frac{16}{9}$. \\ 
\end{tabular}
\end{paragraph}
\begin{paragraph}\indent We can now apply our method to $m$.  Keep in mind that the process will be slightly different this time, since the denominator of $\displaystyle\frac{16}{9}$ is a prime power, not just a prime.  Here, $b = 9 = 3^2$, so we will use $p = 3$ and $c = 2$ (as they are used in the beginning of this example).  Our goal is to find positive integers $m_1, k_1$ such that: \\
~ \\
\begin{tabular}{ll}
$7.1$ & $m = {m_1}\cdot{3^{k_1}}, k_1 \geq 2$, and $\gcd(m_1, 3) = 1$. (Note that $m_1 \mid m$, and \\ ~ & $\gcd(m, 2) = 1$, so $\gcd(m_1, 6) = 1$.) \\
$7.2$ & $\displaystyle\frac{\sigma(m_1)}{m_1} = \displaystyle\frac{16\cdot{3^{k_1 - 2}}}{\sigma(3^{k_1})}$. \\ 
\end{tabular}
\end{paragraph}
\begin{paragraph}\indent We will consider two of the possible cases here: $k_1 = 2$ and $k_1 = 3$.
\end{paragraph}
\begin{paragraph}\indent First, let $k_1 = 2$.  We get
\begin{center} $\displaystyle\frac{\sigma(m_1)}{m_1} = \displaystyle\frac{16}{\sigma(3^2)} = \displaystyle\frac{16}{13}$.
\end{center}
Carrying the process one step further, we will search for such an $m_1$.  We must find positive integers $m_2, k_2$ such that: \\
~ \\
\begin{tabular}{ll}
$8.1$ & $m_1 = {m_2}\cdot{13^{k_2}}$, and $\gcd(m_2, 13) = 1$. (Note that $m_2 \mid m_1$, so \\ ~ & $\gcd(m_2, 6) = \gcd(m_2, 13) = 1$.) \\
$8.2$ & $\displaystyle\frac{\sigma(m_2)}{m_2} = \displaystyle\frac{16\cdot{13^{k_2 - 1}}}{\sigma(13^{k_2})}$. \\
\end{tabular}
\end{paragraph}
\begin{paragraph}\indent Let $k_2 = 1$.  Then $m_1 = 13{m_2}$, and
\begin{center} $\displaystyle\frac{\sigma(m_2)}{m_2} = \displaystyle\frac{16}{14} = \displaystyle\frac{8}{7}$.
\end{center}
\end{paragraph}
\begin{paragraph}\indent Let $m_2 = 7$.  Then $m_2$ satisfies both ($8.1$) and ($8.2$) of the case $k = 5, k_1 = 2, k_2 = 1$.  This gives us a solution; all we have to do now is work backwards until we get $n$.
\end{paragraph}
\begin{paragraph}\indent First, $m_1 = 13{m_2} = {13}\cdot7$.  Next, $m = {3^{k_1}}{m_1} = {13}\cdot{7}\cdot{3^2}$.  Finally, $n = {2^5}m = {2^5}\cdot{3^2}\cdot{7}\cdot{13} = 26208$.
\end{paragraph}
\begin{paragraph}\indent Thus, we have found a solution to the problem $\displaystyle\frac{\sigma(n)}{n} = \displaystyle\frac{7}{2}$.
\end{paragraph}
\begin{paragraph}\indent Now we will consider the case $k_1 = 3$, which will give us one more solution:
\begin{center} $\displaystyle\frac{\sigma(m_1)}{m_1} = \displaystyle\frac{{16}\cdot3}{\sigma(3^3)} = \displaystyle\frac{48}{40} = \displaystyle\frac{6}{5}$.
\end{center}
That is, we need to find $m_1$ such that $\displaystyle\frac{\sigma(m_1)}{m_1} = \displaystyle\frac{6}{5}$, and $\gcd(m_1, 6) = 1$.  If we let $m_1 = 5$, then we have solved this problem, and have thus discovered another solution of the problem $\displaystyle\frac{\sigma(n)}{n} = \displaystyle\frac{7}{2}$, this time for the case $k = 5, k_1 = 3$.  Again, we work backwards until we get $n$.
\end{paragraph}
\begin{paragraph}\indent First, $m = {3^{k_1}}{m_1} = {3^3}\cdot5$.  Next, $n = {2^5}m = {2^5}\cdot{3^3}\cdot{5} = 4320$.
\end{paragraph}
\begin{paragraph}\indent This is a second solution to the problem $\displaystyle\frac{\sigma(n)}{n} = \displaystyle\frac{7}{2}$.
\end{paragraph}
\end{exmpl}
\begin{remrk}\label{remark11} Example \ref{example16} shows that $\displaystyle\frac{7}{2}$ is an abundancy index.  Also, since the equation $I(n) = \displaystyle\frac{7}{2}$ has at least two solutions, namely $n_1 = 4320$ and $n_2 = 26208$, this implies that $n_1$ and $n_2$ here are friendly.
\end{remrk}
\begin{paragraph}\indent Prior to discussing the proof of the lemma outlined in Example \ref{example16}, we review some known properties of the abundancy index:
\end{paragraph}
\begin{lemm}\label{lemma12} Properties of the Abundancy Index
\begin{itemize}
\item{If $a = \displaystyle\prod_{j=1}^k {{p_j}^{n_j}}$, where $p_1, p_2, p_3, \ldots, p_k$ are distinct primes, $k$ is a positive integer, and the integral exponents $n_1, n_2, n_3, \ldots, n_k$ are nonnegative, then $I(a) = \displaystyle\prod_{j=1}^k {\displaystyle\frac{{p_j}^{n_j + 1} - 1}{{p_j}^{n_j}(p_j - 1)}}$, where $I$ is multiplicative.}
\item{If $p$ is prime then the least upper bound for the sequence $\displaystyle\left\{I(p^n)\right\}_{n = 0}^{\infty}$ is $\displaystyle\frac{p}{p - 1}$.}
\item{If $a_1 \mid a$ and $a_1 > 0$, then $I(a) \geq I(a_1)$.}
\item{Obviously, if $a_1 \mid a$ and $I(a) = I(a_1)$, then $a = a_1$.}
\item{If $\gcd(a, \sigma(a)) = 1$, then the unique solution to $I(x) = I(a)$ is $x = a$.}
\end{itemize}
\end{lemm}
\begin{proof} Only the last assertion is not so obvious.  Note that $I(x) = I(a)$ is equivalent to $a\cdot\sigma(x) = x\cdot\sigma(a)$.  If $a$ and $\sigma(a)$ are coprime, then $a \mid x$ (so that $a$ divides every solution).  By the fourth result of Lemma \ref{lemma12}, $x = a$ is the sole solution.
\end{proof}
\begin{paragraph}\indent We discuss some known properties that can help us decide whether a particular fraction $\displaystyle\frac{r}{s}$ is an abundancy index:
\end{paragraph}
\begin{lemm}\label{lemma13} When is $\displaystyle\frac{r}{s}$ an abundancy index? \\
For which rational numbers $\displaystyle\frac{r}{s}$ will
\begin{center} $I(x) = \displaystyle\frac{r}{s}$ (*)
\end{center}
have at least one solution?  In order for (*) to have solutions, $\displaystyle\frac{r}{s}$ must be greater than or equal to one.  If $\displaystyle\frac{r}{s} = 1$, then by Lemma \ref{lemma12}, $x = 1$ is the unique solution.  Throughout the rest of this lemma, it will be assumed that $r$ and $s$ represent given positive integers which are relatively prime, and that $r > s$.  Let us now state some known results:
\begin{itemize}
\item{Note that $s$ must divide (every solution for) $x$ in (*).  This property is easy to observe since $I(x) = \displaystyle\frac{r}{s}$ implies that $s\cdot\sigma(x) = r\cdot x$, and $\gcd(r, s) = 1$.}
\item{If a solution to (*) exists, then $r \geq \sigma(s)$ (since $\displaystyle\frac{r}{s} = I(x) \geq I(s) = \displaystyle\frac{\sigma(s)}{s}$).}
\item{If $I(a) > \displaystyle\frac{r}{s}$, then (*) has no solution which is divisible by $a$.  Additionally, if $s$ is divisible by $a$, then (*) has no solution.}
\item{$\displaystyle\left\{I(b): b \in {\mathbb Z}^{+}\right\}$ is dense in the interval $(1, \infty)$.}
\item{The set of values $\displaystyle\frac{r}{s}$ for which (*) has no solution is also dense in $(1, \infty)$.}
\end{itemize}                      
\end{lemm}
\begin{proof} Here, we prove the third assertion.  Suppose (*) has a solution which is divisible by $a \in {\mathbb Z}^+$.  Then $a \mid x$, and by the third result in Lemma \ref{lemma12}, $I(x) \geq I(a)$.  But then, by assumption $I(a) > \displaystyle\frac{r}{s}$, which implies that $I(a) > \displaystyle\frac{r}{s} = I(x) \geq I(a)$, resulting in the contradiction $I(a) > I(a)$.  Thus, (*) has no solution which is divisible by $a$.  The second statement follows from this and the first result in Lemma \ref{lemma13} (i.e. $s$ must divide every solution for $x$).  
\end{proof}
\begin{paragraph}\indent We now state and prove the following lemma (taken from \cite{L37}) which was used in Example \ref{example16}.
\end{paragraph}
\begin{lemm}\label{lemma14} (Ludwick)  Let $a, b, c \in {\mathbb Z}^+$, and let $p$ be a prime such that $\gcd(a, p) = 1$, $b = p^c$ (so that $b$ is a prime power, and $\gcd(a, b) = 1$), and $a \geq \sigma(b)$.  Suppose we want to find a positive integer, $n$, such that $\displaystyle\frac{\sigma(n)}{n} = \displaystyle\frac{a}{b}$.  This problem is equivalent to the problem of finding positive integers $m, k$ such that:
\begin{itemize}
\item{$n = mp^k, k \geq c,$ and $\gcd(m, p) = 1$ (or equivalently, $\gcd(m, b) = 1$)}
\item{$\displaystyle\frac{\sigma(m)}{m} = \displaystyle\frac{a{p}^{k - c}}{\sigma(p^k)}$}.
\end{itemize}
\end{lemm}
\begin{proof} First, we will show that if we can find $n \in {\mathbb Z}^+$ satisfying $\displaystyle\frac{\sigma(n)}{n} = \displaystyle\frac{a}{b}$, then we can find $m, k \in {\mathbb Z}^+$ satisfying the two itemized conditions above.
\begin{paragraph}\indent Suppose we have $n \in {\mathbb Z}^+$ such that $\displaystyle\frac{\sigma(n)}{n} = \displaystyle\frac{a}{b}$.  Then by Lemma \ref{lemma13}, $b \mid n$; that is, $p^c \mid n$.  Since $b$ is a prime power, there is some $k \in {\mathbb Z}^+$ such that $n = mp^k$, with $\gcd(m, b) = 1$.  Here, $k$ is the largest power of $p$ that divides $n$, so clearly $k \geq c$.  This satisfies the first condition.  Now, we have 
\begin{center} $\displaystyle\frac{\sigma(n)}{n} = \displaystyle\frac{a}{b} = \displaystyle\frac{a}{p^c}$
\end{center} 
and we also have
\begin{center} $\displaystyle\frac{\sigma(n)}{n} = {\displaystyle\frac{\sigma(p^k)}{p^k}}{\displaystyle\frac{\sigma(m)}{m}}$.
\end{center}
Thus,
\begin{center} $\displaystyle\frac{a}{p^c} = {\displaystyle\frac{\sigma(p^k)}{p^k}}{\displaystyle\frac{\sigma(m)}{m}}$,
\end{center}
and so
\begin{center} $\displaystyle\frac{\sigma(m)}{m} = \displaystyle\frac{ap^{k - c}}{\sigma(p^k)}$.
\end{center}
This satisfies the second condition.  Therefore, solving $\displaystyle\frac{\sigma(n)}{n} = \displaystyle\frac{a}{b}$ for $n$ gives us $m, k \in {\mathbb Z}^+$ satisfying the two conditions.
\end{paragraph}
\begin{paragraph}\indent Conversely, we will show that if we can find $m, k \in {\mathbb Z}^+$ satisfying the two conditions, then we can find $n \in {\mathbb Z}^+$ such that $\displaystyle\frac{\sigma(n)}{n} = \displaystyle\frac{a}{b}$. 
\end{paragraph}
\begin{paragraph}\indent Suppose we have $m, k \in {\mathbb Z}^+$ satisfying the two conditions.  Let $n = mp^k$.  Then,
\begin{center} $\displaystyle\frac{\sigma(n)}{n} = {\displaystyle\frac{\sigma(m)}{m}}{\displaystyle\frac{\sigma(p^k)}{p^k}} = {\displaystyle\frac{ap^{k - c}}{\sigma(p^k)}}{\displaystyle\frac{\sigma(p^k)}{p^k}} = \displaystyle\frac{ap^{k - c}}{p^k} = \displaystyle\frac{a}{p^c} = \displaystyle\frac{a}{b}$.
\end{center}
\end{paragraph}
\end{proof}
\subsection{OPNs, Abundancy Outlaws and the Fraction $\displaystyle\frac{p + 2}{p}$}
\begin{paragraph}\indent After defining the abundancy index and exploring various known properties, we briefly discuss some related concepts.  Positive integers having integer-valued abundancy indices are said to be \emph{multiperfect numbers}.  One is the only odd multiperfect that has been discovered.  Richard Ryan hopes that his ``study of the abundancy index will lead to the discovery of other odd multiperfects", or to the proof of their nonexistence.  Since the abundancy index of a number $n$ can be thought of as a measure of its perfection (i. if $I(n) < 2$ then $n$ is deficient;  ii. if $I(n) = 2$ then $n$ is perfect;  and iii. if $I(n) > 2$ then $n$ is abundant), it is fitting to consider it a very useful tool in gaining a better understanding of perfect numbers.  In fact, Judy Holdener \cite{H24} proved the following theorem which provides conditions equivalent to the existence of an OPN:
\end{paragraph}
\begin{thm}\label{theorem15} There exists an odd perfect number if and only if there exist positive integers $p$, $n$ and $\alpha$ such that $p \equiv \alpha \equiv 1 \pmod{4}$, where $p$ is a prime not dividing $n$, and $I(n) = \displaystyle\frac{2p^{\alpha}(p - 1)}{p^{\alpha + 1} - 1}$.
\end{thm}
\begin{proof} By Euler's characterization of an OPN $N = {p^{\alpha}}{m^2}$, it must be true that $p$ is a prime satisfying $\gcd(p, m) = 1$ and $p \equiv \alpha \equiv 1 \pmod{4}$.  Hence $\sigma(N) = \sigma({p^{\alpha}}{m^2}) = \sigma(p^{\alpha})\sigma(m^2) = 2{p^{\alpha}}{m^2}$, and
\begin{center} $I(m^2) = \displaystyle\frac{\sigma(m^2)}{m^2} = \displaystyle\frac{2p^{\alpha}}{\sigma(p^{\alpha})} = \displaystyle\frac{2{p^{\alpha}}(p - 1)}{p^{\alpha + 1} - 1}$.
\end{center}
This proves the forward direction of the theorem.
\begin{paragraph}\indent Conversely, assume there is a positive integer $n$ such that $I(n) = \displaystyle\frac{2p^{\alpha}(p - 1)}{p^{\alpha + 1} - 1}$, where $p \equiv \alpha \equiv 1 \pmod{4}$ and $p$ is a prime with $p \nmid n$.  Then
\begin{center} $I(n\cdot{p^{\alpha}}) = I(n)\cdot{I(p^{\alpha})} = {\displaystyle\frac{2p^{\alpha}(p - 1)}{p^{\alpha + 1} - 1}}{\displaystyle\frac{p^{\alpha + 1} - 1}{{p^{\alpha}}(p - 1)}} = 2$.
\end{center}
So $n\cdot{p^{\alpha}}$ is a perfect number.
\end{paragraph}
\begin{paragraph}\indent Next, we claim that $n\cdot{p^{\alpha}}$ cannot be even.  Suppose to the contrary that $n\cdot{p^{\alpha}}$ is even.  Then it would have the Euclid-Euler form for even perfect numbers:
\begin{center} $n\cdot{p^{\alpha}} = 2^{m - 1}(2^m - 1)$ 
\end{center}
where $2^m - 1$ is prime.  Since $2^m - 1$ is the only odd prime factor on the RHS, $p^{\alpha} = p^{1} = 2^m - 1$.  But $p \equiv 1 \pmod{4}$ and $2^m - 1 \equiv 3 \pmod{4}$ (because $m$ must be at least $2$ in order for $2^m - 1$ to be prime).  This is clearly a contradiction, and thus $n\cdot{p^{\alpha}}$ is not even.  Consequently, $n\cdot{p^{\alpha}}$ is an OPN. 
\end{paragraph} 
\end{proof}
\begin{paragraph}\indent By Theorem \ref{theorem15}, it follows that if one could find an integer $n$ having abundancy index equal to $\displaystyle\frac{5}{3}$ (which occurs as a special case of the theorem, specifically for $p = 5$ and $\alpha = 1$), then one would be able to produce an odd perfect number.  Here we then realize the usefulness of characterizing fractions in $(1, \infty)$ that are \emph{abundancy outlaws}. (Recall Definition \ref{definition11}.)
\end{paragraph}
\begin{paragraph}\indent Let us now consider the sequence of rational numbers in $(1, \infty)$. (Note that, since the number $1$ is solitary and $I(1) = 1$, the equation $I(x) = 1$ has the lone solution $x = 1$.)  For each numerator $a > 1$, we list the fractions $\displaystyle\frac{a}{b}$, with $\gcd(a, b) = 1$, so that the denominators $1 \leq b < a$ appear in increasing order:
\begin{center}
$\displaystyle\frac{2}{1}, \displaystyle\frac{3}{1}, \displaystyle\frac{3}{2}, \displaystyle\frac{4}{1}, \displaystyle\frac{4}{3}, \displaystyle\frac{5}{1}, \displaystyle\frac{5}{2}, \displaystyle\frac{5}{3}, \displaystyle\frac{5}{4}, \displaystyle\frac{6}{1}, \displaystyle\frac{6}{5}, \displaystyle\frac{7}{1}, \displaystyle\frac{7}{2}, \displaystyle\frac{7}{3}, \displaystyle\frac{7}{4}, \displaystyle\frac{7}{5}, \displaystyle\frac{7}{6}, \ldots$
\end{center}
It is intuitive that each term in this sequence must be either an abundancy index or an abundancy outlaw, but it is, in general, difficult to determine the status of a given fraction.  We may thus partition the sequence into three (3) categories:  $(I)$ those fractions that are known to be abundancy indices,  $(II)$ those that are known to be abundancy outlaws, and $(III)$ those whose abundancy index/outlaw status is unknown.  We wish to capture outlaws from the third category, thereby increasing the size of the second category.  Since fractions of the form $\displaystyle\frac{\sigma(N) - t}{N}$ for $t \geq 1$ belong to the first category (by Lemma \ref{lemma5} and Example \ref{example15}), it is tempting to consider fractions of the form $\displaystyle\frac{\sigma(N) + t}{N}$.  Judy Holdener and William Stanton proved in $2007$ \cite{H26} that, under certain conditions, $\displaystyle\frac{\sigma(N) + t}{N}$ is an abundancy outlaw.  They noted that their original interest in such fractions stemmed from their interest in the fraction $\displaystyle\frac{5}{3} = \displaystyle\frac{\sigma(3) + 1}{3}$.  Unfortunately, the results they obtained do not allow them to say anything about fractions of the form $\displaystyle\frac{\sigma(p) + 1}{p} = \displaystyle\frac{p + 2}{p}$.  Such elusive fractions remain in category three.
\end{paragraph}
\begin{paragraph}\indent Equivalently, we may ask:  Does there exist an odd number $s \in {\mathbb Z}^+$ (with $s > 1$) such that $I(x) = \displaystyle\frac{s + 2}{s}$ has at least one solution?  The answer to this question is unknown, but we can go ahead and discuss some properties.  
\end{paragraph}
\begin{paragraph}\indent By Lemma \ref{lemma13}, $s \mid x$.  Assume $s$ is an odd composite;  then $\sigma(s) \geq 1 + s + d$, where $d$ is a divisor of $s$ satisfying $1 < d < s$.  Since $s$ is odd, $d \geq 3$, which means that $\sigma(s) \geq s + 4 > s + 2$, or $I(s) > \displaystyle\frac{s + 2}{s}$, a contradiction.  Hence, $s$ must be prime.  If $1 < c < s$, then $\gcd(c, s) = 1$ (since $s$ is prime), and we have:
\begin{center} $I(cs) = I(c)\cdot{I(s)} \geq {\displaystyle\frac{c + 1}{c}}{\displaystyle\frac{s + 1}{s}} > \displaystyle\frac{(s + 1)^2}{s^2} = {\displaystyle\frac{s(s + 2)}{s^2}} + {\displaystyle\left(\frac{1}{s}\right)}^2 > \displaystyle\frac{s + 2}{s}$.
\end{center}
Thus, $x$ does not have a factor between $1$ and $s$.  Moreover, $x$ is a perfect square;  otherwise $\sigma(x)$ would have a factor of $2$ that cannot be ``canceled" since the denominator, $x$, is odd.  (For the same reason, whenever $r$ and $s$ are both odd, any odd solution to $I(x) = \displaystyle\frac{r}{s}$ must be a perfect square.)  We also claim that $(s + 2) \nmid x$ and we prove this by showing that $I\left({s^2(s + 2)^2}\right) > \displaystyle\frac{s + 2}{s}$:
\end{paragraph}
{
\small
\begin{displaymath}
I\left({s^2(s + 2)^2}\right) \geq 1 + \frac{1}{s} + \frac{1}{s^2} + \frac{1}{s + 2} + \frac{1}{(s + 2)^2} + \frac{1}{s(s + 2)} + \frac{1}{{s^2}(s + 2)} + \frac{1}{s(s + 2)^2} + \frac{1}{{s^2}(s + 2)^2}
\end{displaymath}
\begin{displaymath}
= \frac{s^2 + s + 1}{s^2} + \frac{1}{s + 2}\left[1 + \frac{1}{s + 2} + \frac{1}{s}\right] + \frac{1}{s(s + 2)}\left[\frac{1}{s} + \frac{1}{s + 2} + \frac{1}{s(s + 2)}\right]
\end{displaymath}
\begin{displaymath} 
= \frac{s^2 + s + 1}{s^2} + \frac{s^2 + 4s + 2}{s(s + 2)^2} + \frac{2s + 3}{{s^2}(s + 2)^2}
\end{displaymath}
\begin{displaymath} 
= \frac{\left(s^2 + s + 1\right){\left(s + 2\right)}^2 + s{\left(s + 2\right)}^2 + 3}{{s^2}(s + 2)^2}
\end{displaymath}
\begin{displaymath} 
= \frac{{\left(s + 1\right)}^2{\left(s + 2\right)}^2 + 3}{{s^2}(s + 2)^2}
\end{displaymath}
\begin{displaymath} 
= {\left(\frac{s + 1}{s}\right)}^2 + \frac{3}{{s^2}(s + 2)^2} > \frac{s^2 + 2s + 1}{s^2} > \frac{s + 2}{s}
\end{displaymath}
}
\normalsize
\begin{paragraph}\indent Using some of the principles in the last paragraph, Richard Ryan wrote a simple computer program which verified that $I(x) = \displaystyle\frac{s + 2}{s}$ has no solution less than ${10}^{16}$ (when $s \in {\mathbb Z}^+$ is odd with $s > 1$) \cite{R46}.
\end{paragraph}
\begin{paragraph}\indent We now state (without proof) the conditions obtained by Holdener and Stanton \cite{H26} in order for the fraction $\displaystyle\frac{\sigma(N) + t}{N}$ to be an abundancy outlaw.
\end{paragraph}
\begin{thm}\label{theorem16} For a positive integer $t$, let $\displaystyle\frac{\sigma(N) + t}{N}$ be a fraction in lowest terms, and let $N = \displaystyle\prod_{i = 1}^{n} {{p_i}^{k_i}}$ for primes $p_1, p_2, \ldots, p_n$.  If there exists a positive integer $j \leq n$ such that $p_j < \frac{1}{t}\sigma(\frac{N}{{p_j}^{k_j}})$ and $\sigma({p_j}^{k_j})$ has a divisor $D > 1$ such that at least one of the following is true: \\
\begin{tabular}{ll}
1. & $I({p_j}^{k_j})I(D) > \displaystyle\frac{\sigma(N) + t}{N}$ and $\gcd(D, t) = 1$ \\
2. & $\gcd(D, Nt) = 1$ \\
\end{tabular} \\
then $\displaystyle\frac{\sigma(N) + t}{N}$ is an abundancy outlaw. 
\end{thm}
\begin{paragraph}\indent The following are some sequences of abundancy outlaws which can be constructed from Theorem \ref{theorem16}:
\begin{itemize}
\item{For all natural numbers $m$ and nonnegative integers $n$, and for all odd primes $p$ such that $\gcd(p, \sigma(2^m)) = 1$, the fraction $\displaystyle\frac{\sigma({2^m}{p^{2n + 1}}) + 1}{{2^m}{p^{2n + 1}}}$ is an abundancy outlaw.}
\item{For all primes $p > 3$, $\displaystyle\frac{\sigma(2p) + 1}{2p}$ is an abundancy outlaw.  If $p = 2$ or $p = 3$ then $\displaystyle\frac{\sigma(2p) + 1}{2p}$ is an abundancy index.}
\item{If $N$ is an even perfect number, $\displaystyle\frac{\sigma(2N) + 1}{2N}$ is an abundancy outlaw.}
\item{Let $M$ be an odd natural number, and let $p$, $\alpha$, and $t$ be odd natural numbers such that $p \nmid M$ and $p < \frac{1}{t}\sigma(M)$.  Then, if $\displaystyle\frac{\sigma({p^{\alpha}}M) + t}{{p^{\alpha}}M}$ is in lowest terms, $\displaystyle\frac{\sigma({p^{\alpha}}M) + t}{{p^{\alpha}}M}$ is an abundancy outlaw.}
\item{For primes $p$ and $q$, with $3 < q$, $p < q$, and $\gcd(p, q + 2) = \gcd(q, p + 2) = 1$, $\displaystyle\frac{\sigma(pq) + 1}{pq}$ is an abundancy outlaw.}
\end{itemize}
\end{paragraph}
\begin{remrk}\label{remark12} The last assertion in the preceding paragraph produces outlaws with ease.  We illustrate this using odd primes $p$ and $q$ satisfying $3 < p < q$ and $q \equiv 1 \pmod{p}$.  It follows that $p \nmid (q + 2)$ and $q \nmid (p + 2)$.  By Dirichlet's theorem on arithmetic progressions of primes, we are assured of the existence of an infinite sequence of primes $q$ satisfying $q \equiv 1 \pmod{p}$.  Thus, there is an \emph{infinite} class of abundancy outlaws corresponding to each odd prime $p > 3$.
\end{remrk}
\begin{paragraph}\indent Judy Holdener, with Laura Czarnecki, also obtained the following results in the summer of 2007 \cite{H27}:
\end{paragraph}
\begin{thm}\label{theorem17} If $\displaystyle\frac{a}{b}$ is a fraction greater than 1 in reduced form, $\displaystyle\frac{a}{b} = I(N)$ for some $N \in \mathbb N$, and $b$ has a divisor $D = \displaystyle\prod_{i = 1}^{n} {{p_i}^{k_i}}$ such that $I({p_i}D) > \displaystyle\frac{a}{b}$ for all $1 \leq i \leq n$, then $\displaystyle\frac{D}{\sigma(D)}\displaystyle\frac{a}{b}$ is an abundancy index as well.
\end{thm}
\begin{proof} Suppose that $I(N) = \displaystyle\frac{a}{b}$ for some $N \in \mathbb N$.  Then by Lemma \ref{lemma13}, $b \mid N$ since $\gcd(a, b) = 1$.  Because $I({p_i}D) > \displaystyle\frac{a}{b}$ for all $1 \leq i \leq n$, we know that it is impossible that ${{p_i}D} \mid N$ by the contrapositive of the third result in Lemma \ref{lemma12}.  However, we know (by Lemma \ref{lemma13}) that $D \mid N$, so ${p_i}^{{k_i} + 1} \nmid N$.  Thus we may write $N = \displaystyle\left(\displaystyle\prod_{i = 1}^{n} {{p_i}^{k_i}}\right)\cdot{r} = Dr$, where $\gcd(p_i, r) = 1$ for all $1 \leq i \leq n$, that is, $\gcd(D, r) = 1$.  Then, since $\sigma(N)$ is multiplicative, we may write $\displaystyle\frac{a}{b} = \displaystyle\frac{\sigma(N)}{N} = \displaystyle\frac{\sigma(D)}{D}\displaystyle\frac{\sigma(r)}{r}$.  Therefore, $I(r) = \displaystyle\frac{\sigma(r)}{r} = \displaystyle\frac{a}{b}\displaystyle\frac{D}{\sigma(D)}$.  Thus, if $I(N) = \displaystyle\frac{a}{b}$ for some $N \in \mathbb N$, then $I(r) = \displaystyle\frac{a}{b}\displaystyle\frac{D}{\sigma(D)}$ for some $r \in \mathbb N$. 
\end{proof}
\begin{cor}\label{corollary6} Let $m, n, t \in \mathbb N$.  If $\displaystyle\frac{\sigma(mn) + \sigma(m)t}{mn}$ is in reduced form with $m = \displaystyle\prod_{i = 1}^{l} {{p_i}^{k_i}}$ and $I({p_i}m) > \displaystyle\frac{\sigma(mn) + \sigma(m)t}{mn}$ for all $1 \leq i \leq l$, then $\displaystyle\frac{\sigma(n) + t}{n}$ is an abundancy index if $\displaystyle\frac{\sigma(mn) + \sigma(m)t}{mn}$ is an index.
\end{cor}
\begin{proof} The proof is very similar to that of Theorem \ref{theorem17}.  We only need to observe that, under the assumptions given in this corollary, $\gcd(m, n) = 1$.
\end{proof}
\begin{remrk}\label{remark13} If, in Corollary \ref{corollary6}, we have $t = 1$ and $n = p$ for some prime $p$, then the corollary tells us that if $\displaystyle\frac{\sigma(mp) + \sigma(m)}{mp}$ is an abundancy index, then $\displaystyle\frac{\sigma(p) + 1}{p}$ is as well.  The fractions $\displaystyle\frac{27}{14}, \displaystyle\frac{39}{22}, \displaystyle\frac{45}{26}$ and $\displaystyle\frac{57}{34}$ all illustrate this fact.  If we could determine that these are indeed indices, then we could say that $\displaystyle\frac{9}{7}, \displaystyle\frac{13}{11}, \displaystyle\frac{15}{13}$ and $\displaystyle\frac{19}{17}$, all of the form $\displaystyle\frac{\sigma(p) + 1}{p}$, are indices as well.  Fractions of the form $\displaystyle\frac{\sigma(p) + 1}{p}$ continue to elude characterization as indices or outlaws!  This is significant because Paul Weiner \cite{W58} proved that if there exists an integer $N$ with abundancy $\displaystyle\frac{5}{3}$, then $5N$ is an odd perfect number.  The fraction $\displaystyle\frac{5}{3}$ is of the form $\displaystyle\frac{\sigma(p) + 1}{p}$ for $p = 3$.
\end{remrk}
\begin{thm}\label{theorem18} If $p > q > 2$ are primes satisfying $p > q^2 - q - 1$, then $\displaystyle\frac{\sigma(qp) + q - 1}{qp}$ is an abundancy outlaw.
\end{thm}
\begin{proof} Suppose that $I(N) = \displaystyle\frac{\sigma(qp) + q - 1}{qp}$ for some $N \in \mathbb N$.  Because $\sigma(qp) + q - 1 = (q + 1)(p + 1) + q - 1 = qp + 2q + p$ and $p > q > 2$ are primes, we know that $\gcd(\sigma(qp) + q - 1, qp) = 1$, and by Lemma \ref{lemma13} we have $qp \mid N$.  If ${q^2}p \mid N$, then $\displaystyle\frac{\sigma(qp) + q - 1}{qp} \geq \displaystyle\frac{\sigma({q^2}p)}{{q^2}p}$.  Simplifying, this becomes
\begin{center} $\displaystyle\frac{(q + 1)(p + 1) + q - 1}{qp} \geq \displaystyle\frac{(q^2 + q + 1)(p + 1)}{{q^2}p}$,
\end{center}
and then $q^2 - q - 1 \geq p$.  This contradicts our hypothesis that $p > q^2 - q - 1$, and therefore ${q^2}p \nmid N$.  Consequently, we may write $N = qK$, where $\gcd(q, K) = 1$.  Then, we have $I(N) = I(q)I(K)$, which gives $I(K) = \displaystyle\frac{I(N)}{I(q)} = {\displaystyle\frac{\sigma(qp) + q - 1}{qp}}\cdot{\displaystyle\frac{q}{q + 1}} = \displaystyle\frac{p\cdot{\displaystyle\frac{q + 1}{2}} + q}{p\cdot{\displaystyle\frac{q + 1}{2}}}$.  On setting $m = p\cdot{\displaystyle\frac{q + 1}{2}} + q$ and $n = p\cdot{\displaystyle\frac{q + 1}{2}}$ and observing that $n < m < \sigma(n)$ with $\gcd(m, n) = 1$, then by Lemma \ref{lemma5} (and Example \ref{example15}), $\displaystyle\frac{m}{n} = {\displaystyle\frac{\sigma(qp) + q - 1}{qp}}\cdot{\displaystyle\frac{q}{q + 1}}$ is an abundancy outlaw.  But this contradicts the fact that $I(K) = \displaystyle\frac{I(N)}{I(q)} = {\displaystyle\frac{\sigma(qp) + q - 1}{qp}}\cdot{\displaystyle\frac{q}{q + 1}}$ for some $K \in \mathbb N$, whence it follows that $I(N) \neq \displaystyle\frac{\sigma(qp) + q - 1}{qp}$ for all $N \in \mathbb N$.  Consequently, $\displaystyle\frac{\sigma(qp) + q - 1}{qp}$ must be an abundancy outlaw under the conditions specified for the primes $p$ and $q$.
\end{proof}
\begin{paragraph}\indent These results allow us to move a few more fractions in $(1, \infty)$ from the set of infinitely many fractions that we are unable to classify (category III), into the infinite set of fractions that are abundancy outlaws (category II).  Furthermore, we can see that certain fractions are linked to others in important ways: determining the status of a given fraction can lead to the classification of new abundancy outlaws and indices.  If the converse of Theorem \ref{theorem17} could be proved, then we would be able to divide certain fractions into equivalence classes of sorts, that is, sets of fractions with the same abundancy index/outlaw status.  However, the question of the existence of an OPN (e.g. the status of the fraction $\displaystyle\frac{\sigma(p) + 1}{p}$ for an odd prime $p$) remains as elusive as ever. 
\end{paragraph}
\section{Bounds for the Prime Factors of OPNs}
In this section, bounds for each of the distinct prime factors of an OPN $N$ are derived, drawing heavily from existing works.  We do this using cases based on the total number of distinct prime factors of $N$ (i.e. $\Omega(N)$).  We also study further results in the field and give examples of various techniques used, including an in-depth and detailed discussion of the factor chain approach.  We give in Subsection $4.2.3$ explicit double-sided bounds for each of the prime factors of an OPN $N$ with $\omega(N) = 9$.  We end the section with a discussion of the author's results on the relationships between the components of an OPN $N$.
\subsection{Results on OPNs}
\begin{paragraph}\indent Let $N = {{q_1}^{a_1}}{{q_2}^{a_2}}\cdots{{q_t}^{a_t}}$ be the canonical factorization of an OPN $N$ (i.e. $q_1, q_2, \ldots, q_t$ are distinct primes with $q_1 < q_2 < \ldots < q_t$ and $t = \omega(N)$).  Then the following statements are true:
{
\small
{
\begin{itemize}
\item{$q_t \geq 100000007$ from Goto and Ohno, improving on Jenkins}
\item{$q_{t - 1} \geq 10007$ from Iannucci, improving on Pomerance}
\item{$q_{t - 2} \geq 101$ from Iannucci}
\item{$q_{i} < {2^{2^{i - 1}}}(t - i + 1)$ for $2 \leq i \leq 6$ from Kishore}
\item{$q_1 < \frac{2t + 6}{3}$ from Gr$\ddot{u}$n}
\item{${q_k}^{a_k} > {10}^{20}$ for some $k$ from Cohen, improving on Muskat}
\item{$N > {10}^{300}$ from Brent, et. al., improving on Brent and Cohen (A search is currently on in http://www.oddperfect.org to prove that $N > {10}^{500}$.)}
\item{$N \equiv 1 \pmod{12}, N \equiv 81 \pmod{324}$ or $N \equiv 117 \pmod{468}$ from Roberts, improving on Touchard and Holdener}
\item{$N < 2^{4^t}$ from Nielsen, improving on Cook}
\item{$\displaystyle\sum_{i = 1}^{t} {a_i} = \Omega(N) \geq 75$ from Hare, improving on Iannucci and Sorli}
\item{$t \geq 9$ from Nielsen, improving on Hagis and Chein}
\item{$t \geq 12$ if $q_1 \geq 5$ from Nielsen, improving on both Hagis and Kishore}
\item{$t \geq 17$ if $q_1 \geq 7$ from Greathouse, improving on Norton}
\item{$t \geq 29$ if $q_1 \geq 11$ from Greathouse, improving on Norton}  
\end{itemize}
}
}
\normalsize
\end{paragraph}
\begin{paragraph}\indent Suryanarayana and Hagis \cite{H20} showed that, in all cases, $0.596 < \sum_{p \mid N} {\frac{1}{p}} < 0.694$.  Their paper gives more precise bounds when $N$ is divisible by $3$ or $5$ (or both).  Cohen \cite{C5} also gave more strict ranges for the same sum, including an argument that such bounds are unlikely to be improved upon significantly.   
\end{paragraph}
\begin{paragraph}\indent We use the preceding facts about OPNs to derive explicit double-sided bounds for the prime factors of an OPN $N$ with $\omega(N) = 9$, in Subsection $4.2.3$.
\end{paragraph}
\subsection{Algorithmic Implementation of Factor Chains}
\begin{paragraph}\indent (We borrow heavily the following material from \cite{S}.)
\end{paragraph}
\begin{paragraph}\indent In the discussion that follows, we will let $N$ denote an OPN, assuming one exists, with the prime decomposition
\begin{center}
$N = {\displaystyle\prod_{i = 1}^{u}{{p_i}^{a_i}}}\cdot{\displaystyle\prod_{i = 1}^{v}{{q_i}^{b_i}}}\cdot{\displaystyle\prod_{i = 1}^{w}{{r_i}^{c_i}}} = {\lambda}\cdot{\mu}\cdot{\nu}$ \\
\end{center}
which we interpret as follows: each ${p_i}^{a_i}$ is a known component of $N$, each $q_i$ is a known prime factor of $N$ but the exponent $b_i$ is unknown, and each prime factor $r_i$ of $N$ and exponent $c_i$ are unknown.  By ``known", we mean either explicitly postulated or the consequence of such an assumption.  Any of $u$, $v$, $w$ may be zero, in which case we set $\lambda$, $\mu$, $\nu$, respectively, equal to $1$.  We also let $\bar{m}$ denote a proper divisor of $m$ (except $\bar{1} = 1$).
\end{paragraph}
\begin{paragraph}\indent We can now illustrate the factor chain approach via an algorithmic implementation that can be used to test a given lower bound for $\omega(N)$.  We assume that $N$ is an OPN with $\omega(N) = t$ distinct prime factors. In brief, the algorithm may be described as a progressive sieve, or ``coin-sorter", in which the sieve gets finer and finer, so that eventually nothing is allowed through.  We shall use the terminology of graph theory to describe the branching process.  Since $3 \mid N$ if $t \leq 11$, for our present purposes the even powers of $3$ are the roots of the trees.  If $3^2$ is an exact divisor of $N$, then, since $\sigma(N) = 2N$, $\sigma(3^2) = 13$ is a divisor of $N$, and so the children of the root $3^2$ are labelled with different powers of $13$.  The first of these is ${13}^1$, meaning that we assume that $13$ is an exact divisor of $N$ (and hence that $13$ is the special prime), the second ${13}^2$, then ${13}^4, {13}^5, \ldots$.  Each of these possibilities leads to further factorizations and further subtrees.  Having terminated all these, by methods to be described, we then assume that $3^4$ is an exact divisor, beginning the second tree, and we continue in this manner.  Only prime powers as allowed in Subsection $4.2.1$ are considered, and notice is taken of whether the special prime has been specified earlier in any path.  These powers are called Eulerian.
\end{paragraph}
\begin{paragraph}\indent We distinguish between \emph{initial components}, which label the nodes and comprise \emph{initial primes} and \emph{initial exponents}, and \emph{consequent primes}, which arise within a tree through factorization.  It is necessary to maintain a count of the total number of distinct initial and consequent primes as they arise within a path, and we let $k$ be this number.
\end{paragraph}
\begin{paragraph}\indent Often, more than one new prime will arise from a single factorization.  All are included in the count, within $k$, and, whenever further branching is required, the smallest available consequent prime is used as the new initial prime.  This preferred strategy will give the greatest increase in $I(\lambda\bar{\mu})$.  On the other hand a strategy of selecting the largest available consequent prime will usually give a significant increase in $k$.
\end{paragraph}
\begin{paragraph}\indent To show that $t \geq \omega$, say, we build on earlier results which have presumably shown that $t \geq \omega - 1$, and we suppose that $t = \omega - 1$.  (The reader may want to review Section $3.1$ at this point.)  If, within any path, we have $k > \omega - 1$, then there is clearly a contradiction, and that path is terminated.  This is one of a number of possible contradictions that may arise and which terminate a path.  The result will be proved when every path in every tree has been terminated with a contradiction (unless an OPN has been found).  The different possible contradictions are indicated with upper case letters.
\end{paragraph}
\begin{paragraph}\indent In the contradiction just mentioned, we have too Many distinct prime factors of $N$: this is Contradiction \emph{M1}.  If there are too Many occurrences of a single prime this is Contradiction \emph{M2};  that is, within a path an initial prime has occurred as a consequent prime more times than the initial exponent.  (So counts must also be maintained within each path of the occurrences of each initial prime as a consequent prime.) 
\end{paragraph}
\begin{paragraph}\indent If $k = \omega - 3$ but none of these $k$ primes exceeds $100$, then Iannucci's result must be (about to be) violated: this is Contradiction \emph{P3}.  If $k = \omega - 2$ and none of these primes exceeds ${10}^4$, then again, Iannucci's result is violated: Contradiction \emph{P2}.  Or if, in this case, one exceeds ${10}^4$ but no other exceeds $100$, then this is another version of Contradiction \emph{P3}.  If $k = \omega - 1$ and none of these primes exceeds ${10}^8$, then Goto/Ohno's result is violated: Contradiction \emph{P1}.  In this case, there are the following further possibilities: one prime exceeds ${10}^8$ but no other exceeds ${10}^4$, or one exceeds ${10}^8$, another exceeds ${10}^4$, but no other exceeds $100$.  These are other versions of contradictions \emph{P2} and \emph{P3}, respectively.  These, and some of the other forms of contradiction below, require only counts or comparisons, and no calculations.
\end{paragraph}
\begin{paragraph}\indent At the outset, a number $B$ is chosen, then the number of subtrees with a given initial prime $p$ is bounded by taking as initial components Eulerian powers $p^a$ with $p^{a + 1} \leq B$.  If possible, these trees are continued by factorizing $\sigma(p^a)$.  When $a$ becomes so large that $p^{a + 1} > B$, which may occur with $a = 0$, then we write $q^b$ for $p^a$ and we have one more subtree with this initial prime; it is distinguished by writing its initial component as $q^{\infty}$.  This tree must be continued differently.  In the first place, the smallest available consequent prime, which is not already an initial prime, is used to begin a new subtree.  If no such primes are available, then we opt to use the procedure that follows.
\end{paragraph}
\begin{paragraph}\indent The product of the $u$ initial components $p^a$ within a path is the number $\lambda$.  Those initial primes $q$ with exponents $\infty$, and all consequent primes which are not initial primes, are the $v$ prime factors of $\mu$.  If $k < \omega - 1$ then there are $w = \omega - k - 1$ remaining prime factors of $N$, still to be found or postulated.  These are the prime factors $r$ of $\nu$.  The numbers $u$, $v$, $w$ are not fixed;  they vary as the path develops, for example, by taking a consequent prime as another initial prime.
\end{paragraph}
\begin{paragraph}\indent If factorization can no longer be used to provide further prime factors of $N$, so, in particular, there are no consequent primes which are not initial primes, then the following result (with proof omitted) is used:
\end{paragraph} 
\begin{lemm}\label{LemmaX} Suppose $w \geq 1$, and assume $r_1 < r_2 < \cdots < r_w$.  Then
\begin{center} 
$\displaystyle\frac{I(\lambda\bar{\mu}{r_1}^{c_1 - 1})}{2 - I(\lambda\bar{\mu})} \leq r$
\end{center} 
for $r = r_1$, with strict inequality if $v \geq 1$ or $w \geq 2$.  Further, if $I(\lambda{\mu}^{\infty}) < 2$, then 
\begin{center}
$r < \displaystyle\frac{2 + I(\lambda{\mu}^{\infty})(w - 1)}{2 - I(\lambda{\mu}^{\infty})}$
\end{center}
for $r = r_1$.
\end{lemm} 
Here, $\bar{\mu}$ is taken to be the product of powers $q^{\beta}$, where $q \mid \mu$ and $\beta$ is given as follows.  Let $b_0 = \min{\left\{b : q^{b + 1} > B\right\}}$.  If $b_0 = 0$, then we proceed in a manner to be described later.  Otherwise, let
\begin{displaymath}
\beta = \left\{
\begin{array}{ll}
b_0, & ~if~b_0~is~even~(b_0 > 0), \\
b_0 + 1, & ~if~b_0~is~odd, \\
\end{array}
\right\}
\end{displaymath}
with one possible exception.  If $\pi$ is the special prime, $\pi \nmid \lambda$ and the set $Q_1 = \left\{q : q \equiv b_0 \equiv 1 \pmod{4}\right\}$ is nonempty, then take $\beta = b_0$ for $q = \min{Q_1}$.  Values of $I(p^a)$ and $I(q^{\beta})$ must be maintained, along with their product.  This is the value of $I(\lambda\bar{\mu})$ to be used in the result mentioned in this paragraph.  We shall refer to Lemma \ref{LemmaX} as \emph{Lemma X}.
\begin{paragraph}\indent \emph{Lemma X} is used to provide an interval, the primes within which are considered in turn as possible divisors of $\nu$.  If there are No primes within the interval that have not been otherwise considered, then this is Contradiction \emph{N}.  New primes within the interval are taken in increasing order, giving still further factors of $N$ either through factorization or through further applications of \emph{Lemma X}.  There will be occasions when no new primes arise through factorization, all being used earlier in the same path.  Then again \emph{Lemma X} is used to provide further possible prime factors of $N$ (or, if $k = \omega - 1$, we may have found an OPN).  This lemma specifically supplies the smallest possible candidate for the remaining primes; a still Smaller prime subsequently arising through factorization gives us Contradiction \emph{S}.
\end{paragraph}
\begin{paragraph}\indent We also denote by $q$ any consequent prime which is not an initial prime, and, for such primes, we let $Q_2 = \left\{q : q \equiv 1 \pmod{4}\right\}$.  Then, for such primes, we let $\beta = 2$ with the possible exception that, considering all primes $q$, we let $\beta = b_0$ or $1$, as relevant, for $q = \min{\left(Q_1 \cup Q_2\right)}$, if this set is nonempty.  Again, the value of $I(\lambda\bar{\mu})$, defined as before, must be maintained.  If this value exceeds $2$, we have an Abundant divisor of $N$, and the path is terminated: Contradiction \emph{A}.  This may well occur with $k < \omega - 1$.  Values of $I(q^{\infty})$ must also be maintained.  These, multiplied with the values of $I(p^a)$, give values of $I(\lambda{\mu}^{\infty})$.  If this is less than $2$ and $k = \omega - 1$ then, for all possible values of the exponents $b$, the postulated number $N$ is Deficient: Contradiction \emph{D}. 
\end{paragraph}
\begin{paragraph}\indent Contradictions \emph{A} and \emph{D} are in fact contradictions of the following lemma, which we shall refer to as \emph{Lemma Y}:
\end{paragraph}
\begin{lemm}\label{LemmaY} For any OPN $N = \lambda\mu\nu$, as given as before, we have $I(\lambda\bar{\mu}) \leq 2 \leq I(\lambda\nu{\mu}^{\infty})$.  Both inequalities are strict if $v > 0$; the left-hand inequality is strict if $w > 0$.
\end{lemm}
\begin{paragraph}\indent If, on the other hand, we have a postulated set of prime powers $p^a$ and $q^b$, for which $I(\lambda\bar{\mu}) \leq 2 \leq I(\lambda{\mu}^{\infty})$, then the main inequality in \emph{Lemma Y} is satisfied and we have candidates for an OPN.  If $v = w = 0$, so that we are talking only of known powers $p^a$, then their product \emph{is} an OPN.  Our sieving principle arises when $v > 0$.
\end{paragraph}
\begin{paragraph}\indent In every such case where we have a set of prime powers satisfying the main inequality of \emph{Lemma Y}, with $v > 0$, we increase the value of $B$ and investigate that set more closely.  With the larger value of $B$, some prime powers shift from $\mu$ to $\lambda$, and allow further factorization, often resulting quickly in Contradiction \emph{M1} or \emph{S}.  The value of $I(\bar{\mu})$ increases, so the interval given by \emph{Lemma X} shortens, and hopefully the case which led to our increasing $B$ is no longer exceptional, or Contradiction \emph{A} or \emph{D} may be enforced.  In that case, we revert to the earlier value of $B$ and continue from where we are.  Alternatively, it may be necessary to increase $B$ still further, and later perhaps further again.  When $w = 0$, since $I(\bar{\mu}) \rightarrow I({\mu}^{\infty})$ as $B \rightarrow \infty$, such cases must eventually be dispensed with, one way or the other.  
\end{paragraph}
\begin{paragraph}\indent We summarize the various contradictions in the following table: \\
\begin{tabular}{|c||l|}
\hline
\emph{A} & There is an Abundant divisor. \\
\hline
\emph{D} & The number is Deficient. \\
\hline
\emph{M1} & There are too Many prime factors. \\
\hline
\emph{M2} & A single prime has occurred too Many times (an excess of that prime). \\
\hline
\emph{N} & There is no New prime within the given interval. \\
\hline
\emph{P1} & There is no Prime factor exceeding ${10}^8$. \\
\hline
\emph{P2} & There is at most one Prime factor exceeding ${10}^4$. \\
\hline
\emph{P3} & There are at most two Prime factors exceeding $100$. \\
\hline
\emph{S} & There is a prime Smaller than the purportedly smallest remaining prime. \\
\hline
\emph{$\Pi$} & None of the primes can be the special prime. \\
\hline
\end{tabular}
\end{paragraph}
\begin{paragraph}\indent One of these, Contradiction \emph{$\Pi$}, was not discussed previously.  Within any path with $k = \omega - 1$, if $\pi$ is not implicit in an initial component and if there is no prime $q \equiv 1 \pmod{4}$, then Contradiction \emph{$\Pi$} may be invoked.
\end{paragraph}
\begin{remrk}\label{SigmaChainAlgorithm} Notice that the algorithm as presented in this subsection could be programmed directly to run on a high-speed computer (even desktop PCs).  Prior experience with such algorithmic programs, however, has shown that it can take months (or years even), to check and/or test a particular value for $t = \omega(N)$.  Current computer architecture limits our capability to carry out these tasks at a reasonable amount of time.
\end{remrk}
\begin{exmpl}\label{FactorChainApproach} \emph{Sigma chains} (otherwise known as \emph{factor chains}) are an easily automated system for proving facts about OPNs. Each line of the proof starts with a prime factor known or assumed to divide an OPN $N$, along with its exponent. Since $\sigma$ is multiplicative, knowledge of this prime power leads to knowledge of other prime powers of $N$. If an impossibility arises (see below), that chain of the proof is terminated and the next possibility is considered.
\begin{paragraph}\indent The following is the start of a proof that no OPN has a component less than ${10}^{30}$. It would take many thousands of pages to complete this proof (\textbf{and this has currently not been completed}); this merely serves as an example of how one constructs such proofs. (The best-known result in this direction is that of Cohen \cite{C51}: An OPN has a component bigger than ${10}^{20}$.)
\end{paragraph}
\begin{paragraph}\indent The factor chains are terminated (the succeeding line is not to be indented further than the preceding) if it fails in one of the following ways:
\end{paragraph}
\begin{paragraph}\indent \emph{xs}: The indicated prime appears more times than it is allowed. (e.g. If the chain assumes that $3^6 \| N$ then a chain with $7$ or more factors of $3$ is terminated.)
\end{paragraph}
\begin{paragraph}\indent \emph{overabundant}: The abundancy of the prime factors already exceeds $2$, so regardless of the other factors, $N$ will fail to be perfect.
\end{paragraph}
\begin{paragraph}\indent The factorizations of the largest half-dozen composites are due to the WIMS (WWWInteractive Multipurpose Server) 'Factoris' at wims.unice.fr. (This proof is taken from \cite{G14}.)
\end{paragraph} 
\\
{
\tiny
{
\begin{tabbing}
$3^6$ \= $-> 1093$\\
\> $1093$ \= $-> 2 * 547$ \+ \\
\> ${547}^2$ \= $-> 3 * 163 * 613$ \+ \\
\> ${163}^2$ \= $-> 3 * 7 * 19 * 67$ \+ \\
\> $7^2$ \= $-> 3 * 19$ \+ \\
\> ${19}^2$ \= $-> 3 * 127$ \+ \\
\> ${127}^2$ \= $-> 3 * 5419$ \+ \\
\> ${5419}^2$ \= $-> 3 * 31 * 313 * 1009$ \+ \\
\> ${31}^2 -> 3xs * 331$ \\
\> ${31}^4 -> 5 * 11 * 17351$ overabundant \\
\> ${31}^6$ \= $-> 917087137$ \+ \\
\> ${917087137}^2 -> 841048817767943907 = 3xs * 43 * ...$ \\
\> ${917087137}^4$ \= $-> 707363113097541065394066657400343621 = 31747594185191 * 2228\ldots9731$ \+ \\
\> ${31747594185191}^2$ \= $-> 1007\ldots1673 = 2671 * ...$ \+ \\
\> ${2671}^2 -> 3xs * 7 * 19 * 31 * 577$ \\
\> ${2671}^4 -> 5 * 11^2 * 571 * 147389551$ overabundant \\
\> ${2671}^6$ \= $-> 127 * 2860238405785894351$ \+ \\
\> ${2860238405785894351}^2 -> 8180\ldots1201 = 3xs * ...$ \\
\> ${2860238405785894351}^4 ->$ \\ 
\> $66928167\ldots72862401$ \\
\> $= 5 * 11 * 27362961781 * ...$ overabundant \\
\> ${2860238405785894351}^6 -> 5475369\ldots$ \\
\> $\ldots5453601 = 7xs * 2339 * 337498477 * 1013\ldots2827 * ...$ \- \- \\
\> ${31747594185191}^4 -> 1015882037027398808700619554107312555810842320401403361$ \\
\> $= 5 * 11 * 27581 * ...$ overabundant \\
\> ${31747594185191}^6 -> $ \\ 
\> $1023917\ldots0882641$ \= $= 29 * 68279 * 17581747 * ...$ \+ \\
\> ${29}^2$ \= $-> 13 * 67$ \+ \\
\> ${13}^2 -> 3xs * 61$ \\
\> ${13}^4$ \= $-> 30941$ \+ \\
\end{tabbing}
}
}
\normalsize
\end{exmpl}
\subsection{Explicit Double-Sided Bounds for the Prime Factors}
\begin{paragraph}\indent The results from Subsection $4.2.1$ give some restrictions on the magnitude of the prime factors of an OPN $N$.  For instance, we saw from Subsection $3.5.2$ that $N$ is not divisible by $3\cdot5\cdot7 = 105$.  Consequently, it must be true that the third smallest prime factor $q_3 \geq 11$.  To further derive bounds for the other prime factors, we will use some of the many published results on OPNs, a compendium of which has been presented in Subsection $4.2.1$.
\end{paragraph}
\begin{paragraph}\indent For the largest prime factors of an OPN, Iannucci and Jenkins have worked to find lower bounds. The largest three factors must be at least $100000007$, $10007$, and $101$.  Goto and Ohno verified that the largest factor $q_t$ must be at least $100000007$ using an extension of the methods of Jenkins.
\end{paragraph}
\begin{paragraph}\indent Nielsen, improving the bound of Hagis and Kishore, showed that if an OPN is not divisible by $3$, it must have at least $12$ distinct prime factors. Nielsen also showed that a general odd perfect number, if it exists, must have at least $9$ distinct prime factors.  Therefore, for $9 \leq t \leq 11$, we have $q_1 = 3$.
\end{paragraph}
\begin{paragraph}\indent A result by Gr$\ddot{u}$n (and perhaps, independently too, by Perisastri) will be useful for our purposes later: $q_1 < {\frac{2}{3}t} + 2$.  Results similar to those previously mentioned reduce the practicality of Gr$\ddot{u}$n's findings.  In fact, a paper by Norton published about two years after supersedes Gr$\ddot{u}$n's inequality, except that Norton's method is slightly more computationally intensive.
\end{paragraph}
\begin{paragraph}\indent With an application of Goto/Ohno's and Iannucci's results, we can modestly improve on these bounds with an otherwise straightforward utilization of the abundancy index function:
\end{paragraph}
\begin{paragraph}\indent Let $q_1 \geq 7$.  Now, suppose that $t = \omega(N) = 17$. Using the lower bounds indicated for the three largest prime factors of $N$ as before, we have:
{
\small
{
\begin{center}
$q_1 \geq 7$ \\
$q_2 \geq 11$ \\
$q_3 \geq 13$ \\
$q_4 \geq 17$ \\
$q_5 \geq 19$ \\
$q_6 \geq 23$ \\
$q_7 \geq 29$ \\
$q_8 \geq 31$ \\
$q_9 \geq 37$ \\
$q_{10} \geq 41$ \\
$q_{11} \geq 43$ \\
$q_{12} \geq 47$ \\
$q_{13} \geq 53$ \\
$q_{14} \geq 59$ \\
$q_{15} \geq 101$ \\
$q_{16} \geq 10007$ \\
$q_{17} \geq 100000007$ \\
\end{center}
}
}
\normalsize
Recall from Lemma \ref{lemma7} that $2 = \displaystyle\frac{\sigma(N)}{N} < \displaystyle\prod_{i = 1}^{t} {\left(\displaystyle\frac{q_i}{q_i - 1}\right)}$.  Also, note that
\begin{center} 
$q_i \geq a_i$ for all $i$ implies that $\displaystyle\frac{q_i}{q_i - 1} \leq \displaystyle\frac{a_i}{a_i - 1}$ for all $i$,
\end{center} 
so that we have 
\begin{center}
$2 = \displaystyle\frac{\sigma(N)}{N} < \displaystyle\prod_{i = 1}^{t} {\left(\displaystyle\frac{q_i}{q_i - 1}\right)} \leq \displaystyle\prod_{i = 1}^{t} {\left(\displaystyle\frac{a_i}{a_i - 1}\right)}$.
\end{center}
The numerator of the rightmost fraction is approximately \\  $6.4778249375254265314282935191886 \cdot {10}^{33}$, while the denominator is approximately $3.2172767985350308489460711424 \cdot {10}^{33}$, which gives a ratio of approximately \\ $2.0134496790810999533771971435786$, which is larger than $2$.  No contradiction at this point.
\end{paragraph}
\begin{paragraph}\indent As before, let $q_1 \geq 7$, but now suppose that $t = \omega(N) = 16$. Proceeding similarly as before, we have:
{
\small
{
\begin{center}
$q_1 \geq 7$ \\
$q_2 \geq 11$ \\
$q_3 \geq 13$ \\
$q_4 \geq 17$ \\
$q_5 \geq 19$ \\
$q_6 \geq 23$ \\
$q_7 \geq 29$ \\
$q_8 \geq 31$ \\
$q_9 \geq 37$ \\
$q_{10} \geq 41$ \\
$q_{11} \geq 43$ \\
$q_{12} \geq 47$ \\
$q_{13} \geq 53$ \\
$q_{14} \geq 101$ \\
$q_{15} \geq 10007$ \\
$q_{16} \geq 100000007$ \\
\end{center}
}
}
\normalsize
{
\small
{
Again, we have
\begin{center}
$2 = \displaystyle\frac{\sigma(N)}{N} < \displaystyle\prod_{i = 1}^{t} {\left(\displaystyle\frac{q_i}{q_i - 1}\right)} \leq \displaystyle\prod_{i = 1}^{t} {\left(\displaystyle\frac{a_i}{a_i - 1}\right)}$.
\end{center}
where $q_i \geq a_i$ for each $i$.  Our computations show that:
\begin{center}
$\displaystyle\prod_{i = 1}^{t} {\left(\displaystyle\frac{a_i}{a_i - 1}\right)} = \displaystyle\frac{1.0979364300890553443098802574896 \cdot {10}^{32}}{55470289629914324981828812800000}$ \\
$ = 1.9793234133339626660318209208061 < 2$. \\
\end{center}
This results in the contradiction $2 < 2$.  We therefore conclude that $t \geq 17$ if $q_1 \geq 7$.
}
}
\normalsize
\end{paragraph}
\begin{paragraph}\indent We may likewise prove, using the same method, that $t \geq 29$ if $q_1 \geq 11$.
\end{paragraph}
\begin{paragraph}\indent A result of great utility here is an earlier work of Kishore \cite{K35}, where he proves that
\begin{center}
$q_{i} < {2^{2^{i - 1}}}(t - i + 1)$ for $2 \leq i \leq 6$. \\
\end{center}
These results (by Gr$\ddot{u}$n/Perisastri and Kishore) allow us to give explicitly reduced bounds for the lowest six ($6$) prime factors for an OPN $N$ with a given number $t = \omega(N)$ of distinct prime factors.  For example, an OPN with nine ($9$) distinct divisors has $q_1 \leq 7$ (the smallest prime strictly less than $\displaystyle\frac{2 \cdot 9}{3} + 2 = 8$), $q_2 \leq 31$ (the smallest prime less than ${2^{2^{1}}}\cdot{\left(9 - 2 + 1\right)} = 32$), $q_3 \leq 109$, $q_4 < {2^8}\cdot{6} = 3\cdot{2^9}$, $q_5 < 5\cdot{2^{16}}$, and $q_6 < {2^{32}}\cdot{4} = 2^{34}$.
\end{paragraph}
\begin{paragraph}\indent By using the abundancy index function, we can further reduce the bound for $q_2$.  If $q_2 \geq 13$, then
{
\small
{
\begin{center}
$2 = I(N) \leq I({3}^{\alpha}{13}^{\beta}{17}^{\chi}{19}^{\delta}{23}^{\epsilon}{29}^{\phi}{101}^{\gamma}{10007}^{\eta}{100000007}^{\iota}) < \displaystyle\frac{3\cdot{13}\cdot{17}\cdot{19}\cdot{23}\cdot{29}\cdot{101}\cdot{10007}\cdot{100000007}}{2\cdot{12}\cdot{16}\cdot{18}\cdot{22}\cdot{28}\cdot{100}\cdot{10006}\cdot{100000006}} = \displaystyle\frac{849216193914429412851}{426034693082080051200} = 1.9933029110162608467608441119731 < 2$
\end{center}
}
}
\normalsize
so that $q_2 \leq 11$.
\end{paragraph}
\begin{paragraph}\indent We may also try reducing the bound for $q_3$.  Proceeding in the same manner as before, if $q_1 \geq 3$, $q_2 \geq 5$ and $q_3 \geq 53$, then
{
\small
{
\begin{center}
$2 = I(N) \leq I({3}^{\alpha}{5}^{\beta}{53}^{\chi}{59}^{\delta}{61}^{\epsilon}{67}^{\phi}{101}^{\gamma}{10007}^{\eta}{100000007}^{\iota}) < \displaystyle\frac{3\cdot{5}\cdot{53}\cdot{59}\cdot{61}\cdot{67}\cdot{101}\cdot{10007}\cdot{100000007}}{2\cdot{4}\cdot{52}\cdot{58}\cdot{60}\cdot{68}\cdot{100}\cdot{10006}\cdot{100000006}} = \displaystyle\frac{19375328833237423387515}{9850131125407832064000} = 1.9670122749188491347643364823597 < 2$
\end{center}
}
}
\normalsize
so that $q_3 \leq 47$.
\end{paragraph}
\begin{paragraph}\indent Notice that a major problem with the abundancy index function is that it is not capable of determining upper bounds on the prime factors of an OPN beyond the smallest three (except in some special cases).  This is due to the fact that the first three primes could be $3$, $5$, \emph{and} $11$, in which case the sigma bounds would allow an arbitrary number of additional prime divisors, but \emph{require no more}.   
\end{paragraph}
\begin{paragraph}\indent In a preprint published in the electronic journal INTEGERS in $2003$, Nielsen improved on Cook's bound by showing that $N < 2^{4^{t}}$.  Since $q_i < N$ for all $i$, Nielsen's bound is an implicit upper limit on $q_t$.  If $q_t$ is the special (or Euler) prime factor with an exponent of $1$ and the other $q_i$'s are small, then we can say little else about $q_t$.  We can, however, give tighter limits for the other prime factors. 
\end{paragraph}
\begin{paragraph}\indent Since only one of $q_t, q_{t - 1}$ can be the special prime, at least one exponent is even.  Consequently, ${q_t}{q_{t - 1}}^2 < N < 2^{4^{t}}$, so that $q_{t - 1} < 2^{\displaystyle\frac{4^t}{3}}$.  Likewise, for $1 \leq i \leq t$, we have $q_i < 2^{\displaystyle\frac{4^t}{2(t - i) + 1}}$.  This range can be limited further by considering the other prime factors.  Hare proved that there are at least 75 total primes (not distinct), so we may take the other primes to be as small as possible and raise the smallest prime to the appropriate power (and the others to the $2$nd power).  By using such a method, we may be able to reduce the bound by perhaps a million, depending on $t = \omega(N)$. 
\end{paragraph}
\begin{paragraph}\indent We have thus given explicit formulations (through Gr$\ddot{u}$n/Perisastri and Kishore) for the upper limits on the smallest $6$ prime divisors, which are augmented with sigma conditions for the lowest $3$.  The higher prime divisors are likewise restricted (through Nielsen), though not as tightly.
\end{paragraph}
\begin{paragraph}\indent Going beyond the results previously mentioned is not easy, considering the fact that both Iannucci and Jenkins used proofs based on the divisibility of cyclotomic polynomials ${F_p}(x)$ to find lower bounds for the highest prime factors, and that topic is quite hard to follow, to my knowledge.  The method used in most modern proofs is that of \emph{factor/sigma chains}.  Consider an OPN $N$ with a component $5^{4k + 1}$.  (By \emph{component}, we mean either a prime power that divides $N$, or simply a factor of $N$, which may not necessarily be a prime power.)  We know that $\sigma(5) \mid \sigma(5^{4k + 1})$ for all positive integers $k$, so we may conclude that $\sigma(5) \mid 2N$. But $\sigma(5) = 6$, so this is an indication that $3 \mid N$. This illustrates how knowledge of a particular prime power leads to knowledge of other prime powers, for an OPN, by virtue of the fact that the $\sigma$ function is multiplicative.
\end{paragraph}
\begin{paragraph}\indent We summarize the results in this subsection (which are taken from Greathouse \cite{G14}) as follows:
\begin{center}
\scriptsize
{
$\omega(N) = 9$ \\
$100000007 \leq q_9 < 2^{4^9}$ \\
$10007 \leq q_8 < 2^{\displaystyle\frac{4^9}{3}}$ \\
$101 \leq q_7 < 2^{\displaystyle\frac{4^9}{5}}$ \\
$23 \leq q_6 \leq 17179869143$ \\
$19 \leq q_5 \leq 327673$ \\ 
$13 \leq q_4 \leq 1531$ \\
$11 \leq q_3 \leq 47$ \\
$5 \leq q_2 \leq 11$ \\
$3 \leq q_1 \leq 3$ \\
}
\normalsize
\end{center}
\end{paragraph}
\begin{remrk}\label{OPNFactorBounds} These bounds, together with the algorithm presented in Subsection $4.2.2$, can (potentially) prove the conjecture that $\omega(N) \geq 10$, for a general OPN $N$.
\end{remrk}
\begin{paragraph}\indent In the next subsection, we shall discuss some of the author's own results on the relationships between the components of an OPN $N$.
\end{paragraph}
\subsection{Relationships Between OPN Components}
\begin{paragraph}\indent Throughout this subsection, we let $N = {p^k}{m^2}$ be an OPN with special/Euler prime $p$ with $p \equiv k \equiv 1 \pmod{4}$ and $\gcd(p, m) = 1$.  (Recall from Section $2.4$ that $p^k$ is called the \emph{Euler's factor} of the OPN $N$.)  It will also be useful later to consider the canonical factorization $N = \displaystyle\prod_{i = 1}^{\omega(N)}{{p_i}^{{\alpha}_i}}$, where $p_1 < p_2 < \cdots < p_t$, $t = \omega(N)$ and ${\alpha}_i > 0$ for all $i$.
\end{paragraph}
\begin{paragraph}\indent We begin with some numerical results:
\end{paragraph}
\begin{lemm}\label{IndexInequalities1}
$1 < \displaystyle\frac{\sigma(p^k)}{p^k} < \displaystyle\frac{5}{4} < \displaystyle\frac{8}{5} < \displaystyle\frac{\sigma(m^2)}{m^2} < 2$
\end{lemm}
\begin{proof} ~ \\
{
\small
{
\begin{center}
$\displaystyle\frac{\sigma(p^k)}{p^k} = \displaystyle\frac{p^k + p^{k - 1} + \ldots + p^2 + p + 1}{p^k} = 1 + \displaystyle\frac{1}{p} + {\displaystyle\left(\displaystyle\frac{1}{p}\right)}^2 + \cdots {\displaystyle\left(\displaystyle\frac{1}{p}\right)}^{k - 1} + {\displaystyle\left(\displaystyle\frac{1}{p}\right)}^k \geq 1 + \displaystyle\frac{1}{p}$
\end{center}
}
}
\normalsize
\begin{paragraph}\indent
In particular, $\displaystyle\frac{\sigma(p^k)}{p^k} > 1$.
\end{paragraph}
{
\small
{
\begin{center}
$\displaystyle\frac{\sigma(p^k)}{p^k} = \displaystyle\frac{p^{k + 1} - 1}{{p^k}\left(p - 1\right)} < \displaystyle\frac{p^{k + 1}}{{p^k}\left(p - 1\right)} = \displaystyle\frac{p}{p - 1} = \displaystyle\frac{1}{1 - {\displaystyle\frac{1}{p}}}$
\end{center}
}
}
\normalsize
\begin{paragraph}\indent But since ${p^k}{m^2}$ is an OPN, $p$ is prime and $p \equiv 1 \pmod{4}$.  This implies that $p \geq 5$, from which it follows that $1 - \displaystyle\frac{1}{p} \geq 1 -  \displaystyle\frac{1}{5} = \displaystyle\frac{4}{5}$.  Thus, $\displaystyle\frac{p}{p - 1} = \displaystyle\frac{1}{1 - {\displaystyle\frac{1}{p}}} \leq \displaystyle\frac{5}{4}$, and we have $\displaystyle\frac{\sigma(p^k)}{p^k} < \displaystyle\frac{5}{4}$.  Note that, for $k \geq 1$ and prime $p$, we have:
\begin{center} $\displaystyle\frac{p + 1}{p} \leq \displaystyle\frac{\sigma(p^k)}{p^k} < \displaystyle\frac{p}{p - 1}$
\end{center}
\end{paragraph}
\begin{paragraph}\indent Also, since ${p^k}{m^2}$ is an OPN, ${\displaystyle\left[\displaystyle\frac{\sigma(p^k)}{p^k}\right]}{\displaystyle\left[\displaystyle\frac{\sigma(m^2)}{m^2}\right]} = 2$, which implies that \\ $\displaystyle\frac{\sigma(p^k)}{p^k} = \displaystyle\frac{2{m^2}}{\sigma(m^2)}$.  But $1 < \displaystyle\frac{\sigma(p^k)}{p^k} = \displaystyle\frac{2{m^2}}{\sigma(m^2)} < \displaystyle\frac{5}{4}$.  Consequently, we have \\ $\displaystyle\frac{4}{5} < \displaystyle\frac{\sigma(m^2)}{2{m^2}} < 1$, and thus, $\displaystyle\frac{8}{5} < \displaystyle\frac{\sigma(m^2)}{m^2} < 2$. 
\end{paragraph}
\end{proof}
\begin{cor}\label{IndexInequalities2}
If $N = {p^k}{m^2}$ is an OPN with Euler's factor $p^k$, then
\begin{center} $\displaystyle\frac{p + 1}{p} \leq \displaystyle\frac{\sigma(p^k)}{p^k} < \displaystyle\frac{p}{p - 1} < \displaystyle\frac{2\left(p - 1\right)}{p} < \displaystyle\frac{\sigma(m^2)}{m^2} \leq \displaystyle\frac{2p}{p + 1}$
\end{center} 
\end{cor}
\begin{proof} The proof is similar to that for Lemma \ref{IndexInequalities1}.  We give here a proof of the inequality in the middle.  Suppose to the contrary that $\displaystyle\frac{p}{p - 1} \geq \displaystyle\frac{2\left(p - 1\right)}{p}$.  Since $p \geq 5 > 0, p^2 \geq 2(p - 1)^2$.  This implies that $p^2 - 4p + 2 \leq 0$.  This last inequality is a contradiction since it implies that $p(p - 4) + 2 \leq 0$, whereas $p \geq 5$ implies that $p(p - 4) + 2 \geq 7$.
\end{proof}
\begin{paragraph}\indent In what follows, we set $X = \displaystyle\frac{\sigma(p^k)}{p^k}$ and $Y = \displaystyle\frac{\sigma(m^2)}{m^2}$.
\end{paragraph}
\begin{lemm}\label{IndexInequalities3} $\displaystyle\frac{57}{20} < {\displaystyle\frac{\sigma(p^k)}{p^k}} + {\displaystyle\frac{\sigma(m^2)}{m^2}} < 3$
\end{lemm}
\begin{proof} By Lemma \ref{IndexInequalities1}, $1 < X < \displaystyle\frac{5}{4} < \displaystyle\frac{8}{5} < Y < 2$.  Consider $(X - 1)(Y - 1)$.  This quantity is positive because $1 < X < Y$.  Thus,  
\begin{center}
$(X - 1)(Y - 1) = XY - (X + Y) + 1 > 0$,
\end{center}
which implies that $X + Y < XY + 1$.  But $XY = 2$.  Thus, $X + Y < 3$.  Now, consider $\displaystyle\left(X - \displaystyle\frac{5}{4}\right)\displaystyle\left(Y - \displaystyle\frac{5}{4}\right)$.  This quantity is negative because $X < \displaystyle\frac{5}{4} < Y$.  Thus,
\begin{center}
$\displaystyle\left(X - \displaystyle\frac{5}{4}\right)\displaystyle\left(Y - \displaystyle\frac{5}{4}\right) = XY - {\displaystyle\frac{5}{4}}\displaystyle\left(X + Y\right) + \displaystyle\frac{25}{16} < 0$,
\end{center}
which implies that ${\displaystyle\frac{5}{4}}\displaystyle\left(X + Y\right) > XY + \displaystyle\frac{25}{16}$.  But again, $XY = 2$.  Consequently, ${\displaystyle\frac{5}{4}}\displaystyle\left(X + Y\right) > \displaystyle\frac{57}{16}$, and hence $X + Y > \displaystyle\frac{57}{20}$. 
\end{proof}
\begin{cor}\label{IndexInequalities4}
$\displaystyle\frac{3{p^2} - 4p + 2}{p(p - 1)} < {\displaystyle\frac{\sigma(p^k)}{p^k}} + {\displaystyle\frac{\sigma(m^2)}{m^2}} \leq \displaystyle\frac{3{p^2} + 2p + 1}{p(p + 1)}$
\end{cor}
\begin{proof} From Corollary \ref{IndexInequalities2}:
\begin{center}
$\displaystyle\frac{p + 1}{p} \leq X < \displaystyle\frac{p}{p - 1} < \displaystyle\frac{2\left(p - 1\right)}{p} < Y \leq \displaystyle\frac{2p}{p + 1}$
\end{center}
\begin{paragraph}\indent Consider $\displaystyle\left(X - \displaystyle\frac{p + 1}{p}\right)\displaystyle\left(Y - \displaystyle\frac{p + 1}{p}\right)$.  This quantity is nonnegative because $\displaystyle\frac{p + 1}{p} \leq X < Y$.  Thus,
\begin{center}
$\displaystyle\left(X - \displaystyle\frac{p + 1}{p}\right)\displaystyle\left(Y - \displaystyle\frac{p + 1}{p}\right) = XY - {\displaystyle\frac{p + 1}{p}}\displaystyle\left(X + Y\right) + \displaystyle\frac{(p + 1)^2}{p^2} \geq 0$
\end{center}
which implies that $\displaystyle\frac{p + 1}{p}\displaystyle\left(X + Y\right) \leq 2 + \displaystyle\frac{p^2 + 2p + 1}{p^2} = \displaystyle\frac{3{p^2} + 2p + 1}{p^2}$.  Consequently, $X + Y \leq \displaystyle\frac{3{p^2} + 2p + 1}{p(p + 1)}$.  Now, consider $\displaystyle\left(X - \displaystyle\frac{p}{p - 1}\right)\displaystyle\left(Y - \displaystyle\frac{p}{p - 1}\right)$.  This quantity is negative because $X < \displaystyle\frac{p}{p - 1} < Y$.  Thus,
\begin{center}
$\displaystyle\left(X - \displaystyle\frac{p}{p - 1}\right)\displaystyle\left(Y - \displaystyle\frac{p}{p - 1}\right) = XY - {\displaystyle\frac{p}{p - 1}}\displaystyle\left(X + Y\right) + \displaystyle\frac{p^2}{(p - 1)^2} < 0$
\end{center}
which implies that $\displaystyle\frac{p}{p - 1}\displaystyle\left(X + Y\right) > 2 + \displaystyle\frac{p^2}{p^2 - 2p + 1} = \displaystyle\frac{3{p^2} - 4p + 2}{(p - 1)^2}$.  Consequently, $\displaystyle\frac{3{p^2} - 4p + 2}{p(p - 1)} < X + Y$.
\end{paragraph}
\begin{paragraph}\indent Finally, we need to check that, indeed,
\begin{center}
$\displaystyle\frac{3{p^2} - 4p + 2}{p(p - 1)} = 3 - \displaystyle\frac{p - 2}{p(p - 1)} < 3 - \displaystyle\frac{p - 1}{p(p + 1)} = \displaystyle\frac{3{p^2} + 2p + 1}{p(p + 1)}$ 
\end{center}
Suppose to the contrary that
\begin{center}
$\displaystyle\frac{3{p^2} - 4p + 2}{p(p - 1)} \geq \displaystyle\frac{3{p^2} + 2p + 1}{p(p + 1)}$.
\end{center}
This last inequality implies that $3 - \displaystyle\frac{p - 2}{p(p - 1)} \geq 3 - \displaystyle\frac{p - 1}{p(p + 1)}$, or equivalently, $\displaystyle\frac{p - 1}{p(p + 1)} \geq \displaystyle\frac{p - 2}{p(p - 1)}$.  Since $p \geq 5 > 0$, we have $(p - 1)^2 \geq (p + 1)(p - 2)$, or equivalently, $p^2 - 2p + 1 \geq p^2 - p - 2$, resulting in the contradiction $p \leq 3$. 
\end{paragraph}
\begin{paragraph}\indent Hence, we have
\begin{center}
$\displaystyle\frac{3{p^2} - 4p + 2}{p(p - 1)} < X + Y \leq \displaystyle\frac{3{p^2} + 2p + 1}{p(p + 1)}$,
\end{center}
and we are done. 
\end{paragraph}
\end{proof}
\begin{paragraph}\indent If we attempt to improve the results of Corollary \ref{IndexInequalities4} using Corollary \ref{IndexInequalities2}, we get the following result:
\end{paragraph}
\begin{thm}\label{IndexInequalities5} The series of inequalities
\begin{center}
$L(p) < {\displaystyle\frac{\sigma(p^k)}{p^k}} + {\displaystyle\frac{\sigma(m^2)}{m^2}} \leq U(p)$
\end{center} 
with 
\begin{center}
$L(p) = \displaystyle\frac{3{p^2} - 4p + 2}{p(p - 1)}$ 
\end{center}
and 
\begin{center}
$U(p) = \displaystyle\frac{3{p^2} + 2p + 1}{p(p + 1)}$ 
\end{center} 
is best possible, for a given Euler prime $p \equiv 1 \pmod{4}$ of an OPN $N = {p^k}{m^2}$.
\end{thm}
\begin{proof} From Corollary \ref{IndexInequalities2}, we have:
\begin{center}
$\displaystyle\frac{p + 1}{p} \leq X < \displaystyle\frac{p}{p - 1}$,
\end{center}
and
\begin{center}
$\displaystyle\frac{2\left(p - 1\right)}{p} < Y \leq \displaystyle\frac{2p}{p + 1}$.
\end{center}
We remark that such bounds for $X$ and $Y$ are best possible by observing that $k \equiv 1 \pmod{4}$ implies $k \geq 1$.  Adding the left-hand and right-hand inequalities give rise to:
\begin{center}
$\displaystyle\frac{3p - 1}{p} < X + Y < \displaystyle\frac{p(3p - 1)}{(p + 1)(p - 1)}$
\end{center}
Comparing this last result with that of Corollary \ref{IndexInequalities4}, the result immediately follows if we observe that
\begin{center}
$\max{\displaystyle\left\{\displaystyle\frac{3{p^2} - 4p + 2}{p(p - 1)}, \displaystyle\frac{3p - 1}{p}\right\}} = \displaystyle\frac{3{p^2} - 4p + 2}{p(p - 1)}$
\end{center} 
and
\begin{center}
$\min{\displaystyle\left\{\displaystyle\frac{3{p^2} + 2p + 1}{p(p + 1)}, \displaystyle\frac{p(3p - 1)}{(p + 1)(p - 1)}\right\}} = \displaystyle\frac{3{p^2} + 2p + 1}{p(p + 1)}$,
\end{center}
with both results true when $p \geq 5$ (specifically when $p$ is a prime with $p \equiv 1 \pmod{4}$).
\end{proof}
\begin{paragraph}\indent The reader might be tempted to try to improve on the bounds in Lemma \ref{IndexInequalities3} using Theorem \ref{IndexInequalities5}, but such efforts are rendered futile by the following theorem:
\end{paragraph}
\begin{thm}\label{BestPossibleBounds} The bounds in Lemma \ref{IndexInequalities3} are best possible.
\end{thm}
\begin{proof} It suffices to get the minimum value for $L(p)$ and the maximum value for $U(p)$ in the interval $[5, \infty)$, or if either one cannot be obtained, the greatest lower bound for $L(p)$ and the least upper bound for $U(p)$ for the same interval would likewise be useful for our purposes here.
\begin{paragraph}\indent From basic calculus, we get the first derivatives of $L(p), U(p)$ and determine their signs in the interval $[5, \infty)$:
\end{paragraph}
\begin{center} $L'(p) = \displaystyle\frac{p(p - 4) + 2}{p^2(p - 1)^2} > 0$
\end{center}
and
\begin{center} $U'(p) = \displaystyle\frac{p(p - 2) - 1}{p^2(p + 1)^2} > 0$
\end{center}
which means that $L(p), U(p)$ are increasing functions of $p$ on the interval $[5, \infty)$.  Hence, $L(p)$ attains its minimum value on that interval at $L(5) = \displaystyle\frac{57}{20}$, while $U(p)$ has no maximum value on the same interval, but has a least upper bound of \\
$\displaystyle\lim_{p \rightarrow \infty}{U(p)} = 3$.
\begin{paragraph}\indent This confirms our earlier findings that
\begin{center} $\displaystyle\frac{57}{20} < {\displaystyle\frac{\sigma(p^k)}{p^k}} + {\displaystyle\frac{\sigma(m^2)}{m^2}} < 3$,
\end{center}
with the further result that such bounds are best possible.
\end{paragraph}  
\end{proof}
\begin{remrk}\label{MathematicaX+Y} Let
\begin{center}
$f(p, k) = X + Y = \displaystyle\frac{\sigma(p^k)}{p^k} + \displaystyle\frac{2p^k}{\sigma(p^k)}$.
\end{center}
Using Mathematica, we get the partial derivative:
\begin{center}
$\displaystyle\frac{\partial}{\partial p} f(p, k) = \displaystyle\frac{p^{-1 - k}\displaystyle\left(k - kp + p(-1 + p^k)\right)\displaystyle\left(-1 + p^k(2p + p^k(2 + (-4 + p)p))\right)}{(-1 + p)^2(-1 + p^{1 + k})^2}$ 
\end{center}
which is certainly positive for prime $p \equiv 1 \pmod{4}$ and $k$ a fixed positive integer satisfying $k \equiv 1 \pmod{4}$.  This means that $f(p, k) = X + Y$ is a strictly monotonic increasing function of $p$, for such primes $p$ and fixed integer $k$.  Lastly, $\displaystyle\lim_{p \rightarrow \infty}{X + Y} = 3$.
\end{remrk}
\begin{remrk}\label{JoshuaZelinsky} Why did we bother to focus on improving the bounds for $X + Y$ in the first place?  This is because Joshua Zelinsky, in response to one of the author's posts at the Math Forum
{
\scriptsize
{(http://www.mathforum.org/kb/message.jspa?messageID=4140071\&tstart=0),}
}
\normalsize 
said that ``[he does not] know if this would be directly useful for proving that no [OPNs] exist, [although he is not] in general aware of any sharp bounds [for $X + Y$].  Given that there are odd primitive abundant numbers $n$ of the form $n = {P^K}{M^2}$ with $P$ and $K$ congruent to $1$ modulo $4$ and $\gcd(P, M) = 1$, [he] would be surprised if one could substantially improve on these bounds.  Any further improvement of the lower bound would be equivalent to showing that there are no [OPNs] of
the form $5{m^2}$ which would be a very major result. Any improvement on
the upper bound of $3$ would have similar implications for all
arbitrarily large primes and thus [he thinks] would be a very major result.  [He's] therefore highly curious as to what [the author had] done."  In particular, by using Mathematica, if one would be able to prove that $\displaystyle\frac{43}{15} < X + Y$, then this would imply that $p > 5$ and we arrive at Zelinsky's result that ``there are no OPNs of the form $5{m^2}$".  Likewise, if one would be able to derive an upper bound for $X + Y$ smaller than $3$, say $2.9995$, so that $X + Y < 2.9995$, then this would imply that $p \leq 1999$, confirming Zelinsky's last assertion.
\end{remrk}
\begin{paragraph}\indent Our first hint at one of the relationships between the components of an OPN is given by the following result:
\end{paragraph}
\begin{lemm}\label{UnequalOPNComponents}
If $N = {p^k}{m^2}$ is an OPN, then $p^k \neq m^2$. 
\end{lemm}
\begin{proof} We will give four (4) proofs of this same result, to illustrate the possible approaches to proving similar lemmas:
\begin{itemize}
\item{If $p^k = m^2$, then necessarily $\omega(p^k) = \omega(m^2)$.  But $\omega(p^k) = 1 < 8 \leq \omega(m^2)$, where the last inequality is due to Nielsen.}
\item{Suppose $p^k = m^2$.  This can be rewritten as $p\cdot{p^{k - 1}} = m^2$, which implies that $p \mid m^2$ since $p$ is a prime.  This contradicts $\gcd(p, m) = 1$.}
\item{Assume $p^k = m^2$.  Then $N = p^{2k}$ is an OPN.  This contradicts the fact that prime powers are deficient.}
\item{Let $p^k = m^2$.  As before, $N = p^{2k}$ is an OPN.  This implies that $\sigma(N) = \sigma(p^{2k}) = 1 + p + p^2 + ... + p^{2k - 1} + p^{2k} \equiv (2k + 1) \pmod{4} \equiv 3 \pmod{4}$ (since $p \equiv k \equiv 1 \pmod{4}$).  But, since $N$ is an OPN, $\sigma(N) = 2N$.  The parity of LHS and RHS of the equation do not match, a contradiction.}
\end{itemize}
\end{proof}
\begin{paragraph}\indent By Lemma \ref{UnequalOPNComponents}, either $p^k < m^2$ or $p^k > m^2$.
\end{paragraph}
\begin{paragraph}\indent We now assign the values of the following fractions to the indicated variables, for ease of use later on:
\end{paragraph}
\begin{center}
$\rho_1 = \displaystyle\frac{\sigma(p^k)}{p^k}$ \\
$\rho_2 = \displaystyle\frac{\sigma(p^k)}{m^2}$ \\
$\mu_1 = \displaystyle\frac{\sigma(m^2)}{m^2}$ \\
$\mu_2 = \displaystyle\frac{\sigma(m^2)}{p^k}$ \\
\end{center}
From Lemma \ref{IndexInequalities1}, we have $1 < \rho_1 < \displaystyle\frac{5}{4} < \displaystyle\frac{8}{5} < \mu_1 < 2$.  Note that ${\rho_1}{\mu_1} = {\rho_2}{\mu_2} = 2$.  Also, from Lemma \ref{UnequalOPNComponents}, we get $\rho_1 \neq \rho_2$ and $\mu_1 \neq \mu_2$.
\begin{paragraph}\indent The following lemma is the basis for the assertion that ``Squares cannot be perfect", and will be extremely useful here:
\end{paragraph}
\begin{lemm}\label{SquaresNotPerfect}
Let $A$ be a positive integer.  Then $\sigma(A^2)$ is odd.
\end{lemm}
\begin{proof}
Let $A = \displaystyle\prod_{j = 1}^{R}{{q_i}^{{\beta}_i}}$ be the canonical factorization of $A$, where $R = \omega(A)$.  Then 
\begin{center}
$\displaystyle\sigma(A^2) = \displaystyle\sigma(\displaystyle\prod_{j = 1}^{R}{{q_i}^{2{\beta}_i}}) = \displaystyle\prod_{j = 1}^{R}{\displaystyle\sigma({q_i}^{2{{\beta}_i}})}$,
\end{center} 
since $\sigma$ is multiplicative.  But 
\begin{center}
$\displaystyle\prod_{j = 1}^{R}{\displaystyle\sigma({q_i}^{2{{\beta}_i}})} = \displaystyle\prod_{j = 1}^{R}{\displaystyle\left(1 + {q_i} + {q_i}^2 + \ldots + {q_i}^{2{{\beta}_i} - 1} + {q_i}^{2{\beta}_i}\right)}$.
\end{center}  
But this last product is odd regardless of whether the $q_i$'s are odd or even, i.~e.~ regardless of whether $A$ is odd or even.  Consequently, $\displaystyle\sigma(A^2)$ is odd. 
\end{proof}
\begin{paragraph}\indent A reasoning similar to the proof for Lemma \ref{SquaresNotPerfect} gives us the following lemma:
\end{paragraph}
\begin{lemm}\label{UnequalSigmas1}
Let $N = {p^k}{m^2}$ be an OPN with Euler's factor $p^k$.  Then 
\begin{center}
$\sigma(p^k) \neq \sigma(m^2)$.
\end{center}  
\end{lemm}
\begin{proof}
By Lemma \ref{SquaresNotPerfect}, $\sigma(m^2)$ is odd.  If we could show that $\sigma(p^k)$ is even, then we are done.  To this end, notice that
\begin{center}
$\sigma(p^k) = 1 + p + p^2 + ... + p^k \equiv (k + 1) \pmod{4}$
\end{center}
since $p \equiv 1 \pmod{4}$.  But $k \equiv 1 \pmod{4}$.  This means that $\sigma(p^k) \equiv 2 \pmod{4}$, i.~e.~ $\sigma(p^k)$ is divisible by $2$ but not by $4$.
\end{proof}
\begin{paragraph}\indent From Lemma \ref{UnequalSigmas1}, we get at once the following: $\rho_1 = \displaystyle\frac{\sigma(p^k)}{p^k} \neq \displaystyle\frac{\sigma(m^2)}{p^k} = \mu_2$ and $\rho_2 = \displaystyle\frac{\sigma(p^k)}{m^2} \neq \displaystyle\frac{\sigma(m^2)}{m^2} = \mu_1$.  Also, by a simple parity comparison, we get: $\rho_2 = \displaystyle\frac{\sigma(p^k)}{m^2} \neq \displaystyle\frac{\sigma(m^2)}{p^k} = \mu_2$.  Lastly, $\rho_2 \neq 1$ and $\mu_2 \neq 2$. 
\end{paragraph}
\begin{paragraph}\indent From the equation $\sigma(p^k)\sigma(m^2) = 2{p^k}{m^2}$ and Example \ref{example14}, we know that $p^k \mid \sigma(m^2)$.  This means that $\mu_2 \in {\mathbb Z}^+$.  Suppose that $\mu_2 = 1$.  This implies that $\rho_2 = 2$, and therefore, $\sigma(m^2) = p^k$ and $\sigma(p^k) = 2{m^2}$.  However, according to the paper titled ``Some New Results on Odd Perfect Numbers" by G.~G.~Dandapat, J.~L.~Hunsucker and Carl Pomerance: \emph{No OPN satisfies $\sigma(p^k) = 2{m^2}, \sigma(m^2) = p^k$}. \cite{D10}  This result implies that $\rho_2 \neq 2$ and $\mu_2 \neq 1$. But $\rho_2 \neq 1$ and $\mu_2 \neq 2$.  Since $\mu_2 \in {\mathbb Z}^+$, we then have $\mu_2 \geq 3$.  (Note that $\mu_2$ must be odd.)  Consequently, we have the series of inequalities:
\begin{center}
$0 < \rho_2 \leq \displaystyle\frac{2}{3} < 1 < \rho_1 < \displaystyle\frac{5}{4} < \displaystyle\frac{8}{5} < \mu_1 < 2 < 3 \leq \mu_2$.
\end{center}
In particular, we get the inequalities $p^k < \sigma(p^k) \leq {\displaystyle\frac{2}{3}}{m^2}$ and $\displaystyle\frac{\sigma(p^k)}{\sigma(m^2)} = \displaystyle\frac{\rho_1 + \rho_2}{\mu_1 + \mu_2} < \displaystyle\frac{5}{12}$.
\end{paragraph}
\begin{paragraph}\indent Recall that, from Lemma \ref{IndexInequalities3}, $\displaystyle\frac{57}{20} < \rho_1 + \mu_1 < 3$.  Consider $(\rho_2 - 3)(\mu_2 - 3)$. This quantity is nonpositive because $\rho_2 < 3 \leq \mu_2$.  But $(\rho_2 - 3)(\mu_2 - 3) \leq 0$ implies that ${\rho_2}{\mu_2} - 3(\rho_2 + \mu_2) + 9 \leq 0$, which means that $11 \leq 3(\rho_2 + \mu_2)$ since ${\rho_2}{\mu_2} = 2$.  Consequently, we have the series of inequalities $\displaystyle\frac{57}{20} < \rho_1 + \mu_1 < 3 < \displaystyle\frac{11}{3} \leq \rho_2 + \mu_2$.  In particular, $\rho_1 + \mu_1 \neq \rho_2 + \mu_2$.
\end{paragraph}
\begin{paragraph}\indent We summarize our results from the preceding paragraphs in the theorem that follows:
\end{paragraph}
\begin{thm}\label{OPNComponentsTheorem1} ~
\begin{center}
$0 < \rho_2 \leq \displaystyle\frac{2}{3} < 1 < \rho_1 < \displaystyle\frac{5}{4} < \displaystyle\frac{8}{5} < \mu_1 < 2 < 3 \leq \mu_2$
\end{center}
\begin{center} and
\end{center}
\begin{center}
$\displaystyle\frac{57}{20} < \rho_1 + \mu_1 < 3 < \displaystyle\frac{11}{3} \leq \rho_2 + \mu_2$
\end{center}
\end{thm}
\begin{remrk}\label{Motivation} We remark that the results here were motivated by the initial finding that $I(p^k) < I(m^2)$, i.~e.~ $\rho_1 < \mu_1$.  We prove here too that $I(p^k) < I(m)$.  We start with: For all positive integers $a$ and $b$, $\sigma(ab) \leq \sigma(a)\sigma(b)$ with equality occurring if and only if $\gcd(a, b) = 1$.  (For a proof, we refer the interested reader to standard graduate textbooks in number theory.)  It is evident from this statement that for any positive integer $x > 1$, $I(x^2) < (I(x))^2$.  In particular, from Lemma \ref{IndexInequalities1}, we have $\displaystyle\frac{8}{5} < I(m^2) < (I(m))^2$, which implies that $\displaystyle\frac{2\sqrt{10}}{5} < I(m)$.  But $I(p^k) < 1.25$ (again from Lemma \ref{IndexInequalities1}), and $\displaystyle\frac{2\sqrt{10}}{5} \approx 1.26491106406735$.  Consequently, $I(p^k) < I(m)$.  (Note that $\gcd(p, m) = 1$.)  This should motivate the succeeding discussion, which attempts to improve the result $p^k < {\displaystyle\frac{2}{3}}{m^2}$ to $p^k < m$, where again $N = {p^k}{m^2}$ is an OPN with Euler's factor $p^k$. 
\end{remrk}
\begin{paragraph}\indent We now attempt to obtain the improvement mentioned in Remark \ref{Motivation}.  Since proper factors of a perfect number are deficient, we have $I(m) < 2$.  Consequently, we have the bounds $\displaystyle\frac{2\sqrt{10}}{5} < I(m) < 2$. Likewise, since ${p^k}m$ is a proper factor of $N$ which is perfect, and from Lemma \ref{IndexInequalities1} we have $1 < I(p^k)$, hence the following must be true:
\end{paragraph}
\begin{lemm}\label{IndexInequalities6} Let $N = {p^k}{m^2}$ be an OPN with Euler's factor $p^k$.  Then
\begin{center} 
$\displaystyle\frac{2\sqrt{10}}{5} < I(p^k)I(m) = \displaystyle\frac{\sigma(p^k)}{m}\displaystyle\frac{\sigma(m)}{p^k} < 2$.
\end{center}
\end{lemm}
\begin{paragraph}\indent Now, as before, let
\begin{center}
$\rho_1 = \displaystyle\frac{\sigma(p^k)}{p^k}$ \\
$\rho_3 = \displaystyle\frac{\sigma(p^k)}{m}$ \\
$\mu_3 = \displaystyle\frac{\sigma(m)}{m}$ \\
$\mu_4 = \displaystyle\frac{\sigma(m)}{p^k}$ \\
\end{center}
From the preceding results, we get $1 + \displaystyle\frac{2\sqrt{10}}{5} < \rho_1 + \mu_3 < 3$, where $1 + \displaystyle\frac{2\sqrt{10}}{5} \approx 2.26491$ and the rightmost inequality is obtained via a method similar to that used for Lemma \ref{IndexInequalities3}.  Also, from Lemma \ref{IndexInequalities6} and Lemma \ref{lemma9} (\emph{Arithmetic Mean-Geometric Mean Inequality}), we get the lower bound $\displaystyle\frac{2{\sqrt[4]{1000}}}{5} < \rho_3 + \mu_4$, where $\displaystyle\frac{2{\sqrt[4]{1000}}}{5} \approx 2.24937$.
\end{paragraph}
\begin{paragraph}\indent We observe that $\rho_3 \neq 1$ since $\sigma(p^k) \equiv 2 \pmod{4}$ while $m \equiv 1 \pmod{2}$ since $N$ is an OPN.  Hence, we need to consider two ($2$) separate cases:
\end{paragraph}
\begin{paragraph}\indent \textbf{Case 1:} $\rho_3 < 1$.  Multiplying both sides of this inequality by $\mu_4 > 0$, we get ${\rho_3}{\mu_4} < \mu_4$.  But from Lemma \ref{IndexInequalities6}, we have $\displaystyle\frac{2\sqrt{10}}{5} < {\rho_3}{\mu_4}$.  Therefore: $\rho_3 < 1 < 1.26491106406735 \approx \displaystyle\frac{2\sqrt{10}}{5} < \mu_4$. (In particular, $\rho_3 = \displaystyle\frac{\sigma(p^k)}{m} \neq \displaystyle\frac{\sigma(m)}{p^k} = \mu_4$.)  Note that, from Lemma \ref{IndexInequalities1}: $\displaystyle\frac{p^k}{m} < \rho_3 < 1$, which implies that $p^k < m$. 
\end{paragraph}
\begin{paragraph}\indent \textbf{Case 2:} $1 < \rho_3$.  Similar to what we did in \emph{Case 1}, we get from Lemma \ref{IndexInequalities6}: $\mu_4 < 2$.  Note that, from Lemma \ref{IndexInequalities1}: $1 < \rho_3 < \displaystyle\frac{(5/4)p^k}{m}$, which implies that $p^k > {\displaystyle\frac{4}{5}}{m}$. 
\end{paragraph}
\begin{paragraph}\indent We now claim that $\rho_3 \neq \mu_4$ also holds in \emph{Case 2}.  For suppose to the contrary that $1 < \rho_3 = \mu_4 < 2$.  Since these two ratios are rational numbers between two consecutive integers, this implies that $m \nmid \sigma(p^k)$ and $p^k \nmid \sigma(m)$.  But $\rho_3 = \mu_4 \Rightarrow {p^k}{\sigma(p^k)} = {m}{\sigma(m)}$, which, together with $\gcd(p, m) = \gcd(p^k, m) = 1$, results to a contradiction.  This proves our claim.  Hence, we need to consider two ($2$) further subcases:
\end{paragraph}
\begin{paragraph}\indent \textbf{Subcase 2.1:} $1 < \rho_3 < \mu_4 < 2$
\begin{center}
$1 < \displaystyle\frac{\sigma(p^k)}{m} < \displaystyle\frac{\sigma(m)}{p^k} < 2 \Rightarrow {p^k}{\sigma(p^k)} < {m}{\sigma(m)} \Rightarrow \displaystyle\frac{{p^k}{\sigma(p^k)}}{({p^k}m)^2} < \displaystyle\frac{{m}{\sigma(m)}}{({p^k}m)^2}$ \\
$\displaystyle\frac{1}{m^2}\displaystyle\frac{\sigma(p^k)}{p^k} < \displaystyle\frac{1}{p^{2k}}\displaystyle\frac{\sigma(m)}{m} \Rightarrow \displaystyle\frac{1}{m^2} < \displaystyle\frac{1}{m^2}\displaystyle\frac{\sigma(p^k)}{p^k} < \displaystyle\frac{1}{p^{2k}}\displaystyle\frac{\sigma(m)}{m} < \displaystyle\frac{2}{p^{2k}}$ \\
$p^{2k} < 2{m^2} \Rightarrow p^k < {\sqrt{2}}{m}$ \\
\end{center}
\end{paragraph}
\begin{paragraph}\indent Thus, for \emph{Subcase 2.1}, we have: ${\displaystyle\frac{4}{5}}{m} < p^k < {\sqrt{2}}{m}$.  (Note that it may still be possible to prove either $p^k < m$ or $m < p^k$ in this subcase.  We just need to further develop the methods used.)
\end{paragraph}
\begin{paragraph}\indent Note that we can improve the bounds for $\rho_3$ and $\mu_4$ in this subcase to:
\begin{center}
$1 < \rho_3 < \sqrt{2}$ and $\displaystyle\frac{\sqrt[4]{1000}}{5} < \mu_4 < 2$, 
\end{center}
and that we can derive the upper bound
\begin{center} $\rho_3 + \mu_4 < 3$.
\end{center}
\end{paragraph}
\begin{paragraph}\indent \textbf{Subcase 2.2:} $\mu_4 < \rho_3$, $1 < \rho_3$ and $\mu_4 < 2$
\begin{center}
$\displaystyle\frac{\sigma(m)}{p^k} < \displaystyle\frac{\sigma(p^k)}{m} \Rightarrow {m}{\sigma(m)} < {p^k}{\sigma(p^k)} \Rightarrow \displaystyle\frac{{m}{\sigma(m)}}{({p^k}m)^2} < \displaystyle\frac{{p^k}{\sigma(p^k)}}{({p^k}m)^2}$ \\
$\displaystyle\frac{1}{p^{2k}}\displaystyle\frac{\sigma(m)}{m} < \displaystyle\frac{1}{m^2}\displaystyle\frac{\sigma(p^k)}{p^k} < \displaystyle\frac{1}{m^2}\displaystyle\frac{\sigma(m)}{m}$ \\
$\displaystyle\frac{1}{p^{2k}} < \displaystyle\frac{1}{m^2} \Rightarrow m^2 < p^{2k} \Rightarrow m < p^k$ \\
\end{center}
\end{paragraph}
\begin{paragraph}\indent It is here that the author's original conjecture that $p^k < m$ in all cases is disproved.
\end{paragraph}
\begin{paragraph}\indent Note also the following improvements to the bounds for $\rho_3$ and $\mu_4$ in this subcase: $\displaystyle\frac{\sqrt[4]{1000}}{5} < \rho_3$ and $\mu_4 < \sqrt{2}$.
\end{paragraph}
\begin{paragraph}\indent At this point, the author would like to set the following goals to treat this \emph{Subcase 2.2} further:
\begin{itemize}
\item{Obtain an upper bound for $\rho_3$.}
\item{Obtain a lower bound for $\mu_4$.}
\item{Obtain an upper bound for $\rho_3 + \mu_4$.}
\end{itemize}
\end{paragraph}
\begin{paragraph}\indent We summarize our results in the following theorem:
\end{paragraph}
\begin{thm}\label{OPNComponentsTheorem2} Let $N = {p^k}{m^2}$ be an OPN with Euler's factor $p^k$.  Then $\rho_3 \neq \mu_4$, and the following statements hold:
\begin{itemize}
\item{If $\rho_3 < 1$, then $p^k < m$.}
\item{Suppose that $1 < \rho_3$.
			\begin{itemize}
			\item{If $\rho_3 < \mu_4$, then ${\displaystyle\frac{4}{5}}{m} < p^k < 					{\sqrt{2}}{m}$.}
			\item{If $\mu_4 < \rho_3$, then $m < p^k$.}
			\end{itemize}
			}
\end{itemize}
\end{thm}
\begin{paragraph}\indent We now state and prove here the generalization to $\rho_2 \leq \displaystyle\frac{2}{3}$ mentioned in Section $3.2$:
\end{paragraph}
\begin{thm}\label{OPNComponentsTheorem3}
Let $N = \displaystyle\prod_{i = 1}^{\omega(N)}{{p_i}^{{\alpha}_i}}$ be the canonical factorization of an OPN $N$, where $p_1 < p_2 < \cdots < p_t$ are primes, $t = \omega(N)$ and ${\alpha}_i > 0$ for all $i$.  Then $\sigma({p_i}^{{\alpha}_i}) \leq {\displaystyle\frac{2}{3}}{\displaystyle\frac{N}{{p_i}^{{\alpha}_i}}}$ for all $i$.
\end{thm}
\begin{proof} Let $N = {{p_i}^{{\alpha}_i}}M$ for a particular $i$.  Since ${{p_i}^{{\alpha}_i}} || N$ and $N$ is an OPN, then $\sigma({p_i}^{{\alpha}_i})\sigma(M) = 2{{p_i}^{{\alpha}_i}}M$.  From Example \ref{example14}, we know that ${{p_i}^{{\alpha}_i}} \mid \sigma(M)$ and we have $\sigma(M) = h{{p_i}^{{\alpha}_i}}$ for some positive integer $h$.  Assume $h = 1$.  Then $\sigma(M) = {p_i}^{{\alpha}_i}$, forcing $\sigma({{p_i}^{{\alpha}_i}}) = 2M$.  Since $N$ is an OPN, $p_i$ is odd, whereupon we have an odd $\alpha_i$ by considering parity conditions from the last equation.  But this means that ${p_i}^{{\alpha}_i}$ is the Euler's factor of $N$, and we have ${p_i}^{{\alpha}_i} = p^k$ and $M = m^2$.  Consequently, $\sigma(m^2) = \sigma(M) = {p_i}^{{\alpha}_i} = p^k$, which contradicts the fact that $\mu_2 \geq 3$.  Now suppose that $h = 2$.  Then we have the equations $\sigma(M) = 2{p_i}^{{\alpha}_i}$ and $\sigma({{p_i}^{{\alpha}_i}}) = M$.  (Note that, since $M$ is odd, $\alpha_i$ must be even.)  Applying the $\sigma$ function to both sides of the last equation, we get $\displaystyle\sigma(\sigma({{p_i}^{{\alpha}_i}})) = \sigma(M) = 2{p_i}^{{\alpha}_i}$, which means that ${p_i}^{{\alpha}_i}$ is an odd superperfect number.  But Kanold \cite{S53} showed that odd superperfect numbers must be perfect squares (no contradiction at this point, since $\alpha_i$ is even), and Suryanarayana \cite{S54} showed in $1973$ that ``There is no odd super perfect number of the form ${p}^{2{\alpha}}$" (where $p$ is prime).  Thus $h = \displaystyle\frac{\sigma(M)}{{p_i}^{{\alpha}_i}} \geq 3$, whereupon we have the result $\sigma({p_i}^{{\alpha}_i}) \leq {\displaystyle\frac{2}{3}}M = {\displaystyle\frac{2}{3}}{\displaystyle\frac{N}{{p_i}^{{\alpha}_i}}}$ for the chosen $i$.  Since $i$ was arbitrary, we have proved our claim in this theorem.
\end{proof}
\begin{paragraph}\indent The following corollary is a direct consequence of Theorem \ref{OPNComponentsTheorem3}:
\end{paragraph}
\begin{cor}\label{OPNComponentsCorollary1}
Let $N$ be an OPN with $r = \omega(N)$ distinct prime factors.  Then 
\begin{center}
$N^{2 - r} \leq \displaystyle\left(\displaystyle\frac{1}{3}\right)\displaystyle\left(\displaystyle\frac{2}{3}\right)^{r - 1}$.
\end{center}
\end{cor}
\begin{paragraph}\indent In the next section, we attempt to ``count" the number of OPNs by trying to establish a bijective map between OPNs and points on a certain hyperbolic arc.  We prove there that such a mapping (which is based on the concept of the abundancy index) is neither surjective nor injective, utilizing results on solitary numbers in the process.
\end{paragraph}
\section{``Counting" the Number of OPNs}
In this section, we will be disproving the following conjecture:
\begin{conj}\label{ConjectureCountOPNs}
For each $N = {p^k}{m^2}$ an OPN with $N > {10}^{300}$, there corresponds exactly one ordered pair of rational numbers $\displaystyle\left(\displaystyle\frac{\sigma(p^k)}{p^k}, \displaystyle\frac{\sigma(m^2)}{m^2}\displaystyle\right)$ lying in the region $1 < \displaystyle\frac{\sigma(p^k)}{p^k} < \displaystyle\frac{5}{4}$, $\displaystyle\frac{8}{5} < \displaystyle\frac{\sigma(m^2)}{m^2} < 2$, and $\displaystyle\frac{57}{20} < \displaystyle\frac{\sigma(p^k)}{p^k} + \displaystyle\frac{\sigma(m^2)}{m^2} < 3$, and vice-versa.
\end{conj}
\begin{paragraph}\indent We begin our disproof by observing that, for prime $p$, $\displaystyle\frac{\sigma(p^k)}{p^k}$ is a decreasing function of $p$ (for constant $k$) and is also an increasing function of $k$ (for constant $p$).  These observations imply that $\displaystyle\frac{\sigma({p_1}^{k_1})}{{p_1}^{k_1}} \neq \displaystyle\frac{\sigma({p_2}^{k_2})}{{p_2}^{k_2}}$ for the following cases: \\ \textbf{(1)} $p_1 \neq p_2, k_1 = k_2$ and \textbf{(2)} $p_1 = p_2, k_1 \neq k_2$.  To show that the same inequality holds for the case \textbf{(3)} $p_1 \neq p_2, k_1 \neq k_2$, we proceed as follows:  Suppose to the contrary that $\displaystyle\frac{\sigma({p_1}^{k_1})}{{p_1}^{k_1}} = \displaystyle\frac{\sigma({p_2}^{k_2})}{{p_2}^{k_2}}$ but $p_1 \neq p_2, k_1 \neq k_2$. Then 
\begin{center}
${{p_2}^{k_2}}({p_2} - 1)({p_1}^{{k_1} + 1} - 1) = {{p_1}^{k_1}}({p_1} - 1)({p_2}^{{k_2} + 1} - 1)$.  
\end{center}
Since $p_1$ and $p_2$ are distinct (odd) primes, $\gcd(p_1, p_2) = 1$ which implies that 
\begin{center} 
${p_1}^{k_1} \mid ({p_2} - 1)({p_1}^{{k_1} + 1} - 1)$ and ${p_2}^{k_2} \mid ({p_1} - 1)({p_2}^{{k_2} + 1} - 1)$.
\end{center}
Now, let us compute $\gcd({p^k}, p^{k + 1} - 1)$ using the Euclidean Algorithm:
\begin{center}
$p^{k + 1} - 1 = p\cdot{p^k} - 1 = (p - 1)\cdot{p^k} + (p^k - 1)$ \\
$p^k = 1\cdot(p^k - 1) + 1 \longrightarrow$ \textbf{last nonzero remainder} \\
$p^k - 1 = (p^k - 1)\cdot{1} + 0$ \\
\end{center}
Consequently, $\gcd({p^k}, p^{k + 1} - 1) = 1$.  Hence, we have that 
\begin{center}
${p_1}^{k_1} \mid ({p_2} - 1)$ and ${p_2}^{k_2} \mid ({p_1} - 1)$.
\end{center}
Since $p_1$ and $p_2$ are primes, this means that ${p_1}{p_2} \mid ({p_1} - 1)({p_2} - 1)$, which further implies that ${p_1}{p_2} \leq ({p_1} - 1)({p_2} - 1) = {p_1}{p_2} - {(p_1 + p_2)} + 1$, resulting in the contradiction $p_1 + p_2 \leq 1$.
\end{paragraph}
\begin{remrk}\label{OPNMapping1} An alternative proof of the results in the preceding paragraph may be obtained by using the fact that prime powers are solitary (i.~e.~ $I(p^k) = I(x)$ has the sole solution $x = p^k$). 
\end{remrk}
\begin{paragraph}\indent Let $X = \displaystyle\frac{\sigma(p^k)}{p^k}$ and $Y = \displaystyle\frac{\sigma(m^2)}{m^2}$.  It is straightforward to observe that, since the abundancy index is an arithmetic function, then for each $N = {p^k}{m^2}$ an OPN (with $N > {10}^{300}$), there corresponds exactly one ordered pair of rational numbers $(X, Y)$ lying in the hyperbolic arc $XY = 2$ bounded as follows: $1 < X < 1.25$, $1.6 < Y < 2$, and $2.85 < X + Y < 3$.  (Note that these bounds are the same ones obtained in Subsection $4.2.4$.)
\end{paragraph}
\begin{paragraph}\indent We now disprove the backward direction of Conjecture \ref{ConjectureCountOPNs}.  We do this by showing that the mapping $X = \displaystyle\frac{\sigma(p^k)}{p^k}$ and $Y = \displaystyle\frac{\sigma(m^2)}{m^2}$ is neither surjective nor injective in the specified region.
\end{paragraph}
\begin{paragraph}\indent \textbf{(X, Y) is not surjective.}  We prove this claim by producing a rational point $(X_0, Y_0)$ lying in the specified region, and which satisfies $X_0 \neq \displaystyle\frac{\sigma(p^k)}{p^k}$ for all primes $p$ and positive integers $k$.  It suffices to consider the example $X_0 = \displaystyle\frac{\sigma(pq)}{pq}$ where $p$ and $q$ are primes satisfying $5 < p < q$.  Notice that $1 < X_0 = \displaystyle\frac{(p + 1)(q + 1)}{pq} = \displaystyle\left({1 + \displaystyle\frac{1}{p}}\displaystyle\right)\displaystyle\left({1 + \displaystyle\frac{1}{q}}\displaystyle\right) \leq \displaystyle\frac{8}{7}\displaystyle\frac{12}{11} < 1.2468 < 1.25$.  Now, by setting $Y_0 = \displaystyle\frac{2}{X_0}$, the other two inequalities for $Y_0$ and $X_0 + Y_0$ would follow.  Thus, we now have a rational point $(X_0, Y_0)$ in the specified region, and which, by a brief inspection, satisfies $X_0 \neq \displaystyle\frac{\sigma(p^k)}{p^k}$ for all primes $p$ and positive integers $k$ (since prime powers are solitary).  Consequently, the mapping defined in the backward direction of Conjecture \ref{ConjectureCountOPNs} is not surjective.  
\end{paragraph}
\begin{remrk}\label{OPNMapping2} Since the mapping is not onto, there are rational points in the specified region which do not correspond to any OPN.
\end{remrk}
\begin{paragraph}\indent \textbf{(X, Y) is not injective.}  It suffices to construct two distinct OPNs \\ $N_1 = {{p_1}^{k_1}}{{m_1}^2}$ and $N_2 = {{p_2}^{k_2}}{{m_2}^2}$ (assuming at least two such numbers exist) that correspond to the same rational point $(X, Y)$.  Since it cannot be the case that ${p_1}^{k_1} \neq {p_2}^{k_2}, {m_1}^2 = {m_2}^2$, we consider the scenario ${p_1}^{k_1} = {p_2}^{k_2}, {m_1}^2 \neq {m_2}^2$.  Thus, we want to produce a pair $(m_1, m_2)$ satisfying $I({m_1}^2) = I({m_2}^2)$. (A computer check by a foreign professor using Maple produced no examples for this equation in the range $1 \leq m_1 < m_2 \leq 300000$.  But then again, in pure mathematics, absence of evidence is not evidence of absence.)  Now, from the inequalities $p^k < m^2$ and $N = {p^k}{m^2} > {10}^{300}$, we get $m^2 > {10}^{150}$.  A nonconstructive approach to finding a solution to $I({m_1}^2) = I({m_2}^2)$ would then be to consider ${10}^{150} < {m_1}^2 < {m_2}^2$ and Erd$\ddot{o}$s' result that ``\textbf{The number of solutions of $I(a) = I(b)$ satisfying $a < b \leq x$ equals $Cx + o(x)$ where $C \geq \displaystyle\frac{8}{147}$.}" (\cite{E}, \cite{A2}) (Note that $C$ here is the same as the (natural) density of friendly integers.)  Given Erd$\ddot{o}$s' result then, this means that eventually, as $m_2 \rightarrow \infty$, there will be at least ${{10}^{150}}{\displaystyle\frac{8}{147}}$ solutions $(m_1, m_2)$ to $I({m_1}^2) = I({m_2}^2)$, a number which is obviously greater than 1.  This finding, though nonconstructive, still proves that the mapping defined in the backward direction of Conjecture \ref{ConjectureCountOPNs} is not injective.   
\end{paragraph} 
\begin{paragraph}\indent Thus, we have failed our attempt to ``count" the number of OPNs by showing that our conjectured correspondence is not actually a bijection.  However, this should not prevent future researchers from conceptualizing other correspondences that may potentially be bijections.  The author is nonetheless convinced that, ultimately, such bijections would have to make use of the concept of the abundancy index one way or another.
\end{paragraph}
\chapter{Analysis and Recommendations}
Up to this point, the OPN problem is still open.  We have discussed some of the old as well as new approaches to solving this problem, with the \textbf{factor chain approach} taking center stage in mainstream OPN research.  But it appears that, for the ultimate solution of the problem using such approaches, we need fresh ideas and what may be called a ``paradigm shift".  For example, the index/outlaw status of the fraction $\displaystyle\frac{\sigma(p) + 1}{p}$, which may potentially disprove the OPN Conjecture, might probably require concepts from other disparate but related fields of mathematics, such as real analysis.  But then again, this is just an educated guess.
\begin{paragraph}\indent Two open problems necessitating further investigations would come to mind, if the reader had read and fully understood the contents of this thesis:
\begin{itemize}
\item{Prove or disprove: If $N = {p^k}{m^2}$ is an OPN with Euler's factor $p^k$, then $p^k < m$.}
\item{There exists a bijection $(X, Y)$ from the set of OPNs to the hyperbolic arc $XY = 2$ lying in the region $1 < X < 2$, $1 < Y < 2$, $2\sqrt{2} < X + Y < 3$, and this bijection makes use of the concept of abundancy index to define the mappings $X$ and $Y$.}
\end{itemize}
The underlying motivation for these open problems have been described in sufficient detail in Chapter $4$ and should set the mood for extensive investigation into their corresponding solutions.
\end{paragraph}
\begin{paragraph}\indent Lastly, we warn future researchers who would be interested in pursuing this topic that while a multitude of evidence certainly suggests that no OPNs exist, \textbf{neither heuristic nor density arguments} but only a (mathematical) proof of their nonexistence could lay the problem to a ``conclusive" rest.  The author has tried, and although he failed like the others, eventually somebody will give a final end to this problem that had defied solution for more than three centuries. 
\end{paragraph}


\begin{thebibliography}{73}
\bibitem{A1} Anderson, C. W., \textit{Advanced Problem 5967: Density of Odd Deficient Numbers}, American Mathematical Monthly, vol 82, no 10, (1975) pp 1018--1020   
\bibitem{A2} Anderson, C. W., Hickerson, Dean, and Greening, M. G., \textit{Advanced Problem 6020: Friendly Integers}, American Mathematical Monthly, vol 84, no 1, (1977) pp 65--66
\bibitem{B2} Bezverkhnyev, Slava, \textit{Perfect Numbers and Abundancy Ratio}, Undergraduate project in a course in Analytic Number Theory (submitted April 3, 2003) at Carleton University, Ottawa, Canada,  Available online: http://www.slavab.com/EN/math.html, Viewed: July 2007
\bibitem{B3} Brent, R. P., Cohen, G. L., and te Riele, H. J. J., \textit{Improved techniques for lower bounds for odd perfect numbers}, Math. of Comp., vol 57, (1991) pp 857--868
\bibitem{B4} Burton, D. M., \textit{Elementary Number Theory}, Third Edition, Wm. C. Brown Publishers (1994)
\bibitem{C4} Chein, E. Z., \textit{An odd perfect number has at least eight prime factors}, Ph.D. Thesis, Pennsylvania State University, (1979)
\bibitem{C5} Cohen, G. L., \textit{On odd perfect numbers. II. Multiperfect numbers and quasiperfect numbers}, Journal of the Australian Mathematical Society, vol 29, no 3, (1980) pp 369--384
\bibitem{C51} Cohen, G.L., \textit{On the largest component of an odd perfect number}, Journal of the Australian Mathematical Society, vol 42, no 2, (1987) pp 280–-286 
\bibitem{C6} Cohen, G. L. and Hendy, M. D., \textit{On odd multiperfect numbers}, Math. Chron., vol 10, (1981) pp 57--61
\bibitem{C7} Cohen, G. L. and Williams, R. J., \textit{Extensions of some results concerning odd perfect numbers}, Fibonacci Quarterly, vol 23, (1985) pp 70--76
\bibitem{C8} Cook, R. J., \textit{Bounds for odd perfect numbers}, CRM Proc. and Lect. Notes, Centres de Recherches Mathématiques, vol 19, (1999) pp 67--71
\bibitem{C9} Cruz, Christopher Thomas R., \textit{Searching for Odd Perfect Numbers}, M.~S.~ Thesis, De La Salle University, Manila, (2006) 
\bibitem{C10} Curtiss, D. R., \textit{On Kellogg's Diophantine Problem}, American Mathematical Monthly, vol 29, no 10, (1922) pp 380--387
\bibitem{D11} Del$\acute{e}$glise, Marc, \textit{Bounds for the density of abundant integers}, Experiment. Math., vol 7, no 2, (1998) pp 137--143
\bibitem{D10} Dandapat, G. G., Hunsucker, J. L., and Pomerance, C., \textit{Some New Results on Odd Perfect Numbers}, Pacific J. Math., vol 57, (1975) pp 359--364, Available online: http://projecteuclid.org/Dienst/UI/1.0/Display/euclid.pjm/1102905990, Viewed: 2006
\bibitem{D12} Dickson, Leonard, \textit{Finiteness of the odd perfect and primitive abundant numbers with $n$ distinct prime factors}, Amer. J. Math., vol 35, (1913) pp 413--422
\bibitem{E} Erd$\ddot{o}$s, P., \textit{Remarks on number theory, II, Some problems on the $\sigma$ function}, Acta Arith., vol 5, (1959) pp 171--177
\bibitem{G12} Goto, Takeshi and Ohno, Yasuo, \textit{Odd perfect numbers have a prime factor exceeding ${10}^8$}, Preprint, 2006, Available online:  http://www.ma.noda.tus.ac.jp/u/tg/perfect.html,  Viewed: June 2007
\bibitem{G13} Gradstein, I. S., \textit{O ne$\check{c}$etnych sover$\check{s}$ennych $\check{c}$islah}, Mat. Sb., vol 32, (1925) pp 476--510
\bibitem{G14} Greathouse, C., \textit{Bounding the Prime Factors of Odd Perfect Numbers}, An undergraduate paper at Miami University, (2005)
\bibitem{G15} Greathouse, Charles and Weisstein, Eric W, \textit{Odd Perfect Number}, From MathWorld--A Wolfram Web Resource, http://mathworld.wolfram.com/OddPerfectNumber.html
\bibitem{G16} Gr$\ddot{u}$n, O., \textit{$\ddot{U}$ber ungerade vollkommene Zahlen}, Math. Zeit., vol 55, (1952) pp 353--354
\bibitem{H16} Hagis Jr., Peter, \textit{Outline of a proof that every odd perfect number has at least eight prime factors}, Math. Comp., (1980) pp 1027--1032
\bibitem{H17} Hagis Jr., Peter and McDaniel, Wayne, \textit{A new result concerning the structure of odd perfect numbers}, Proc. Amer. Math. Soc., vol 32, (1972) pp 13--15
\bibitem{H18} Hagis Jr., Peter and McDaniel, Wayne, \textit{On the largest prime divisor of an odd perfect number}, Math. of Comp., vol 27, (1973) pp 955--957
\bibitem{H19} Hagis Jr., Peter and McDaniel, Wayne, \textit{On the largest prime divisor of an odd perfect number}, Math. of Comp., vol 29, (1975) pp 922--924
\bibitem{H20} Hagis Jr., Peter and Suryanarayana, D., \textit{A theorem concerning odd perfect numbers}, Fibonacci Quarterly,  vol 8, no 4, (1970) pp 337--346
\bibitem{H21} Hare, K. G. \textit{More on the Total Number of Prime Factors of an Odd Perfect Number}, Math. Comput., (2003)
\bibitem{H22} Hare, K. G. \textit{New Techniques for Bounds on the Total Number of Prime Factors of an Odd Perfect Number}, Math. Comput., vol 74, (2005) pp 1003--1008
\bibitem{H23} Heath-Brown, D. R., \textit{Odd perfect numbers}, Math. Proc. Camb. Phil. Soc.,  vol 115, (1994) pp 191--196
\bibitem{H2} Hickerson, Dean (ed.), \textit{A074902: Known Friendly Numbers}, From the On-Line Encyclopedia of Integer Sequences, Available online: http://www.research.att.com/~njas/sequences/A074902, Viewed: 2008
\bibitem{H24} Holdener, J. A., \textit{Conditions Equivalent to the Existence of Odd Perfect Numbers}, Mathematics Magazine, vol 79, no 5, (2006)
\bibitem{H25} Holdener, J. A. and Riggs, J., \textit{Consecutive Perfect Numbers (actually, the Lack Thereof!)}, An undergraduate research project at Kenyon College, Gambier, OH (1998)
\bibitem{H26} Holdener, J. A. and Stanton, W. G., \textit{Abundancy 'Outlaws' of the Form $\displaystyle\frac{\sigma(N) + t}{N}$}, A joint research project at Kenyon College, Gambier, OH (2007)
\bibitem{H27} Holdener, J. A. and Czarnecki, L., \textit{The Abundancy Index: Tracking Down Outlaws}, A joint research project at Kenyon College, Gambier, OH (2007)
\bibitem{I28} Iannucci, Douglas E., \textit{The second largest prime divisor of an odd perfect number exceeds ten thousand}, Math. of Comp., vol 68, (1999) pp 1749--1760
\bibitem{I29} Iannucci, Douglas E., \textit{The third largest divisor of an odd perfect number exceeds one hundred}, Math. of Comp., vol 69, (2000) pp 867--879
\bibitem{I30} Iannucci, Douglas E. and Sorli, R. M., \textit{On the total number of prime factors of an odd perfect number}, Math. of Comp., vol 72, (2003) pp 2077--2084
\bibitem{J31} Jenkins, Paul M., \textit{Odd perfect numbers have a prime factor exceeding ${10}^7$}, Math.of Comp., vol 72, (2003) pp 1549--1554
\bibitem{K32} Kanold, H. J., \textit{Untersuchungen $\ddot{u}$ber ungerade vollkommene Zahlen}, J. Reine Angew. Math, vol 183, (1941) pp 98--109
\bibitem{K33} Kanold, H. J., \textit{Folgerungen aus dem vorkommen einer Gauss'schen primzahl in der primfactorenzerlegung einer ungeraden vollkommenen Zahl}, J. Reine Angew. Math., vol 186, (1949) pp 25--29
\bibitem{K331} Kanold, H. J., \textit{$\ddot{U}$ber die Dichten der Mengen der vollkommenen und der befreundeten Zahlen}, Math. Z., vol 61, (1954) pp 180--185
\bibitem{K34} Kanold, H. J., \textit{$\ddot{U}$ber einen Satz von L. E. Dickson, II}, Math. Ann, vol 132, (1956) p 273
\bibitem{K35} Kishore, M. \textit{On Odd Perfect, Quasiperfect, and Odd Almost Perfect Numbers}, Math. Comput., vol 36, (1981) pp 583--586
\bibitem{L35} Laatsch, R., \textit{Measuring the Abundancy of Integers}, Mathematics Magazine, vol 59, (1986) pp 84--92
\bibitem{L36} Lipp, W., \textit{OddPerfect.Org}, Available online: http://www.oddperfect.org,  Viewed: 2008
\bibitem{L37} Ludwick, Kurt, \textit{Analysis of the ratio $\displaystyle\frac{\sigma(n)}{n}$},  Undergraduate honors thesis at Penn State University, PA, May, 1994,  Available online: http://www.math.temple.edu/$\sim$ludwick/thesis/thesisinfo.html, Viewed: 2006
\bibitem{M37} McCleary, J. (Vassar College, Poughkeepsie, NY), \textit{Hunting Odd Perfect Numbers: Quarks or Snarks?}, Lecture notes first presented as a seminar to students of Union College, Schenectady, NY, (2001)
\bibitem{M38} McDaniel, Wayne, \textit{On odd multiply perfect numbers}, Boll. Un. Mat. Ital., vol 2, (1970) pp 185--190
\bibitem{N39} Nielsen, Pace P., \textit{An upper bound for odd perfect numbers}, Integers: Electronic Journal of Combinatorial Number Theory, vol 3, (2003) p A14
\bibitem{N40} Nielsen, Pace P.,  \textit{Odd perfect numbers have at least nine distinct prime factors}, Mathematics of Computation, in press, 2006, arXiv:math.NT/0602485
\bibitem{O41} O'Connor, J. J. and Robertson, E. F.  \textit{History Topic: Perfect Numbers},  Available online: http://www-history.mcs.st-andrews.ac.uk/HistTopics/Perfect\_numbers.html,   Viewed: July 2007
\bibitem{P42} Peirce, Benjamin, \textit{On perfect numbers}, New York Math. Diary, vol 2, no XIII, (1832) pp 267--277
\bibitem{P43} Perisastri, M.,  \textit{A note on odd perfect numbers}, Mathematics Student, vol 26, (1958) pp 179--181
\bibitem{P44} Pomerance, Carl, \textit{Odd perfect numbers are divisible by at least seven distinct primes}, Acta Arithmetica, vol XXV, (1974) pp 265--300
\bibitem{P45} Pomerance, Carl, \textit{The second largest divisor of an odd perfect number}, Math. Of Comp., vol 29, (1975) pp 914--921
\bibitem{P46} Pomerance, Carl, \textit{Multiply perfect numbers, Mersenne primes, and effective computability}, Math. Ann., vol 226, (1977), pp 195--226
\bibitem{R46} Ryan, R. F., \textit{Results concerning uniqueness for $\displaystyle\frac{\sigma(x)}{x} = \displaystyle\frac{\sigma({p^n}{q^m})}{{p^n}{q^m}}$ and related topics}, International Mathematical Journal, vol 2, no 5, (2002) pp 497--514
\bibitem{R47} Ryan, R. F.,  \textit{A Simpler Dense Proof Regarding the Abundancy Index}, Mathematics Magazine, vol 76, (2003) pp 299--301
\bibitem{S48} Sandor, Jozsef and Crstici, Borislav (eds.), \textit{Perfect Numbers: Old and New Issues; Perspectives}, (2004), Handbook of Number Theory, vol II, pp 15--98, Dordrecht, The Netherlands: Kluwer Academic Publishers, Retrieved June 10, 2007, from Gale Virtual Reference Library via Thomson Gale: { \tiny{http://find.galegroup.com/gvrl/infomark.do?\&contentSet=EBKS\&type=retrieve\&tabID=T001\&prodId=GVRL\&docId=CX\\2688000009\&source=gale\&userGroupName=gvrlasia30\&version=1.0}}
\normalsize
\bibitem{S49} Sandor, Jozsef, Mitrinovic, Dragoslav, and Crstici, Borislav (eds.), \textit{Sum-of-Divisors Function, Generalizations, Analogues; Perfect Numbers and Related Problems}, (2006), Handbook of Number Theory, vol I, (2nd ed., pp 77--120), Dordrecht, The Netherlands: Springer, Retrieved June 10, 2007, from Gale Virtual Reference Library via Thomson Gale: {
\tiny{http://find.galegroup.com/gvrl/infomark.do?\&contentSet=EBKS\&type=retrieve\&tabID=T001\&prodId=GVRL\&docId=CX\\2594000011\&source=gale\&userGroupName=gvrlasia30\&version=1.0}}
\normalsize
\bibitem{S50} Servais, C., \textit{Sur les nombres parfaits}, Mathesis, vol 8, (1888) pp 92--93
\bibitem{S} Sorli, Ronald M., \textit{Algorithms in the Study of Multiperfect and Odd Perfect Numbers}, Ph.D. Thesis, University of Technology, Sydney, (2003) 
\bibitem{S51} Starni, P., \textit{On the Euler's Factor of an Odd Perfect Number}, J. Number Theory, vol 37, no 3, (1991) pp 366--369
\bibitem{S52} Steuerwald, Rudolf, \textit{Verscharfung einen notwendigen Bedeutung f$\ddot{u}$r die Existenz einen ungeraden volkommenen Zahl}, Bayer. Akad. Wiss. Math. Natur., (1937) pp 69--72
\bibitem{S53} Suryanarayana, D., \textit{Super Perfect Numbers}, Elem. Math., vol 24, (1969) pp 16--17 
\bibitem{S54} Suryanarayana, D., \textit{There Is No Odd Super Perfect Number of the Form $p^{2{\alpha}}$}, Elem. Math., vol 24, (1973) pp 148--150
\bibitem{S55} Sylvester, James Joseph, \textit{Sur les nombres parfaits}, Comptes Rendus, vol CVI, (1888) pp 403--405
\bibitem{S56} Sylvester, James Joseph, \textit{Sur l'impossibilitie de l'existence d'un nombre parfait qui ne contient pas au 5 diviseurs premiers distincts}, Comptes Rendus, vol CVI, (1888) pp 522--526
\bibitem{V55} Various contributors, \textit{Great Internet Mersenne Prime Search}, Available online: http://www.mersenne.org/prime.htm, Viewed: 2008
\bibitem{V56} Various contributors, \textit{Perfect Number},  From Wikipedia - the free encyclopedia, http://en.wikipedia.org/wiki/Perfect\_number
\bibitem{V57} Voight, J., \textit{Perfect Numbers: An Elementary Introduction}, Available online: http://magma.maths.usyd.edu.au/$\sim$voight/notes/perfelem.pdf, Viewed: 2006
\bibitem{W58} Weiner, P. A., \textit{The Abundancy Ratio, A Measure of Perfection},  Mathematics Magazine, vol 73, (2000) pp 307--310
\end{thebibliography}
\end{document}